\UseRawInputEncoding
\documentclass{amsart}
\usepackage{amsfonts,amssymb,amsmath,amsthm,amscd, geometry, lipsum}
\usepackage{fullpage}
\usepackage{tikz}
\usepackage{url}
\usepackage{enumerate}

\newtheorem{theorem}{Theorem}[section]
\newtheorem{lemma}[theorem]{Lemma}
\newtheorem{proposition}[theorem]{Proposition}
\newtheorem{corollary}[theorem]{Corollary}
\newtheorem{hypothesis}[theorem]{Hypothesis}
\theoremstyle{definition}

\newenvironment{remark}[1][Remark]{\begin{trivlist}
\item[\hskip \labelsep {\bfseries #1}]}{\end{trivlist}}

\DeclareFontEncoding{OT2}{}{}
\newcommand{\textcyr}[1]{{\fontencoding{OT2}\fontfamily{wncyr}\fontseries{m}\fontshape{n}\selectfont #1}}
\newcommand{\Sha}{{\mbox{\textcyr{Sh}}}}

\author{Jeanine Van Order}
\address{Fakult\"at f\"ur Mathematik, Universit\"at Bielefeld}
\email{jvanorder@math.uni-bielefeld.de}

\begin{document}

\title{Rankin-Selberg $L$-functions in cyclotomic towers, III}

\begin{abstract} 

Let $\pi$ be a cuspidal automorphic representation of $\operatorname{GL}_2$ over a totally real number field $F$. 
Let $K$ be a totally imaginary quadratic extension of $F$. We estimate central values of the 
$\operatorname{GL}_2 \times \operatorname{GL}_2$ Rankin-Selberg $L$-functions associated to $\pi$ times representations 
induced from Hecke characters of $K$ which are ramified only at a given prime ideal $\mathfrak{p}$ of $F$. 
More specifically, we use spectral decompositions of shifted convolution sums and relations to Fourier-Whittaker 
coefficients of genuine and non-genuine 
metaplectic forms to obtain nonvanishing estimates, averaging over primitive ring class characters of a given exact order. 
When $\pi$ corresponds to a holomorphic Hilbert modular form of arithmetic weight $k \geq 2$, 
we then derive finer results from the rationality theorems of Shimura, together with the existence of suitable 
$\mathfrak{p}$-adic $L$-functions. This allows us to generalize the theorems of Rohrlich, Vatsal, and Cornut-Vatsal 
to this setting. Finally, in a self-contained appendix, we explain how to use these results to deduce bounds 
for Mordell-Weil ranks of the associated $\operatorname{GL}_2$-type abelian varieties via existing Iwasawa main conjecture divisibilities. 

\end{abstract}

\maketitle
\tableofcontents

\section{Introduction}

Let $F$ be a totally real number field of degree $d = [F: {\bf{Q}}]$, ring of integers $\mathcal{O}_F$, and ring of adeles ${\bf{A}}_F$. 
Let $\pi = \otimes_v \pi_v$ be a cuspidal automorphic representation of $\operatorname{GL}_2({\bf{A}}_F)$ of conductor
$c(\pi) \subset \mathcal{O}_F$ and unitary central character $\omega = \omega_{\pi}$. 
Let $K$ be a totally imaginary quadratic extension of $F$ of relative discriminant 
$\mathfrak{D} = \mathfrak{D}_{K/F} \subset \mathcal{O}_F$, absolute discriminant $D_K = {\bf{N}} \mathfrak{D}$, 
and associated idele class character $\eta = \eta_{K/F}$ of $F$. 
Let $\mathcal{W}$ be a finite-order Hecke character of $K$ of the following type. 
We consider the product $\mathcal{W} = \rho \chi \circ {\bf{N}}$ of a ring class character $\rho$ of $K$ 
times a character $\chi \circ {\bf{N}}$ arising via composition with the norm homomorphism 
${\bf{N}} = {\bf{N}}_{K/{\bf{Q}}}: K \rightarrow {\bf{Q}}$ from some Dirichlet character $\chi$ (the basechange of $\chi$ to $K$). 
The ring class character $\rho$ is a character of the class group of the $\mathcal{O}_F$-order 
$\mathcal{O}_{\mathfrak{c}} = \mathcal{O}_F + \mathfrak{c} \mathcal{O}_K$ of some conductor $\mathfrak{c} \subset \mathcal{O}_F$, 
and determines a finite-order idele class character\footnote{Composing with the reciprocity map of class field theory, 
such a character $\rho$ factors through the Galois group of the ring class field $K[\mathfrak{c}]$ of conductor $\mathfrak{c}$ over $K$, which is 
of generalized dihedral type over $F$, and moreover linearly disjoint over $K$ to the cyclotomic ${\bf{Z}}_p$-extension of $K$ for any prime $p$. 
For this reason, such a character $\rho$ is often said to be {\it{dihedral}} or {\it{anticyclotomic}}.} 
\begin{align*} \rho: {\bf{A}}_K^{\times}/ K_{\infty}^{\times} K^{\times} \widehat{\mathcal{O}}_{\mathfrak{c}}^{\times} &\longrightarrow {\bf{S}}^1. \end{align*} 
Here, ${\bf{A}}_K$ denotes the ring of adeles of $K$, with archimedean local component $K_{\infty} = K \otimes_{\bf{Q}} {\bf{C}}$.
We write ${\bf{A}}_K^{\times}$ to denote the ideles of $K$, and $\widehat{\mathcal{O}}_{\mathfrak{c}}^{\times}$ 
the units of the profinite completion of the order $\mathcal{O}_{\mathfrak{c}}$.
The character $\chi \circ {\bf{N}}$ on the other hand factors through some cyclotomic extension of $K$, 
and so we shall sometimes call it ``cyclotomic". To each such Hecke character $\mathcal{W}$ of $K$, 
we have an induced automorphic representation $\pi(\mathcal{W}) = \otimes_v \pi(\mathcal{W})_v$ of $\operatorname{GL}_2 ({\bf{A}}_F)$, 
and can then consider the corresponding $\operatorname{GL}_2 \times \operatorname{GL}_2$ Rankin-Selberg $L$-function
\begin{align}\label{fei} L(s, \pi \times \mathcal{W}) 
&= L(s, \pi \times \pi(\mathcal{W})) = \prod_{v \nmid \infty} L(s, \pi_v \times \pi(\mathcal{W})_v). \end{align} 
This degree-four $L$-function has an analytic continuation thanks to the theory of Jacquet \cite{Ja} 
and Jacquet-Langlands \cite{JL}; its completed $L$-function $\Lambda(s, \pi \times \pi(\mathcal{W})) = L_{\infty}(s) L(s, \pi \times \mathcal{W})$ 
satisfies a functional equation
\begin{align*}\Lambda(s, \pi \times \mathcal{W}) 
&= \epsilon(s, \pi \times \mathcal{W}) \Lambda(1-s, \widetilde{\pi} \times \mathcal{W}^{-1}),\end{align*} 
where $\epsilon(s, \pi \times \mathcal{W}) = c(\pi \times \mathcal{W})^{s - \frac{1}{2}} \epsilon(1/2, \pi \times \mathcal{W})$
denotes the $\epsilon$-factor, with $c(\pi \times \mathcal{W})$ the conductor of the $L$-function, and 
$\epsilon(1/2, \pi \times \mathcal{W}) \in {\bf{S}}^1$ the root number of $L(s, \pi \times \mathcal{W})$. 
If $\pi \cong \widetilde{\pi}$ is self-contragredient and $\mathcal{W} = \rho$ is a ring class character, 
then this functional equation $(\ref{fei})$ is symmetric in the sense that it relates the same $L$-function on each side,
\begin{align}\label{symmetricFE} \Lambda(s, \pi \times \rho) = \epsilon(s, \pi \times \rho) \Lambda(1-s, \pi \times \rho), \end{align}
and the root number $\epsilon(1/2, \pi \times \rho) \in \lbrace \pm 1 \rbrace$ is real-valued. This is a consequence of the fact
that ring class characters are equivariant with respect to complex conjugation (cf.~\cite{Ro2}). 
Hence if $\epsilon(1/2, \pi \times \rho) = -1$ in this setting, then the functional equation $(\ref{symmetricFE})$ forces the vanishing 
of the central value:  $\Lambda(s, \pi \times \rho) = L(s, \pi \times \rho) = 0$. 

Let us now fix a prime ideal $\mathfrak{p} \subset \mathcal{O}_F$ with underlying rational prime $p$. 
It can be seen via direct calculation (see e.g.~$(\ref{RNF})$ below) that as one varies over ring class 
characters $\rho$ of $K$ of $\mathfrak{p}$-power conductor, the root number $\epsilon(1/2, \pi \times \rho) \in {\bf{S}}^1$ 
is generically independent of the choice of ring class character $\rho$. In particular, when $\pi \cong \widetilde{\pi}$ 
is self-contragredient, there exists an integer $\nu = \nu(\pi, K, p) \in \lbrace 0, 1 \rbrace$ such that $\epsilon(1/2, \pi \times \rho) = (-1)^{\nu}$ 
for all but finitely many ring class characters $\rho$ of $K$ of $\mathfrak{p}$-power conductor. This allows us to characterize 
(and discard\footnote{The methods we develop here can be applied to study the central derivative values $L'(1/2, \pi \times \rho)$
in this setting, but at the expense of clarity, and so we save this task for another work. Note that this setting has already been addressed
when $\pi$ corresponds to a holomorphic Hilbert modular form of parallel weight two by Cornut \cite{Cor} and Cornut-Vatsal \cite{CV}.}) 
the setting of forced vanishing via the functional equation $(\ref{fei})$ corresponding to the case of 
$\nu = \nu(\pi, D, p) =1$ when $\pi \cong \widetilde{\pi}$ is self-contragredient. In all other settings, 
we seek to determine how seldom the values $L(1/2, \pi \times \mathcal{W}) = L(1/2, \pi \times \rho \chi \circ {\bf{N}})$ vanish as (i) $\rho$ varies 
over ring class characters of $\mathfrak{p}$-power conductor and (ii) $\chi$ varies over Dirichlet characters of a given $p$-power conductor. 

We begin by fixing a (possibly trivial) primitive Dirichlet character $\chi \bmod p^{\beta}$ for some integer $\beta \geq 0$.
We consider the average of central values of the Rankin-Selberg $L$-function of $\pi \otimes \chi$ times $\pi(\rho)$, where 
$\rho$ ranges over primitive ring class characters of $K$ of exact order $p^x$ (and some corresponding conductor $\mathfrak{p}^{\alpha}$). 
That is, for a sufficiently large integer $\alpha \gg 1$, we consider averages over primitive ring class characters of conductor
$\mathfrak{p}^{\alpha}$, as well as certain subaverages corresponding to ring class characters with a given ``tamely ramified" part (cf.~\cite{Va}, \cite{CV})
or a given exact $p$-power order.
We then explain how to refine these results, using Shimura's rationality theorem and specializations of multivariable $p$-adic $L$-functions. 
This allows us to generalize the nonvanishing theorems of Rohrlich \cite{Ro}, \cite{Ro2}, Vatsal \cite{Va}, and Cornut-Vatsal \cite{CV} 
to the central values $L(1/2, \pi \times \rho \chi \circ {\bf{N}})$, when $\pi$ corresponds to a holomorphic Hilbert modular form of weight 
$k = (k_j)_{j=1}^d$ with each $k_j \geq 2$. As we explain, this has various applications to $p$-adic $L$-functions constructions in the 
Iwasawa theory literature, and has to date been treated as a standard hypothesis. This in fact forms the main motivation for the work. 
The main arguments however come from analytic number theory, and moreover develop ideas from the analytic theory of automorphic forms. 

We first show the following analytic result. Here, we derive integral presentations for our underlying averages we consider in terms of Fourier 
coefficients of certain automorphic forms on $\operatorname{GL}_2({\bf{A}}_F)$ and its two-fold metaplectic cover. 
This allows us to use spectral decompositions of the corresponding automorphic forms to derive bounds, leading us to the following implications. 
Given a sufficiently large integer $\alpha \gg 1$, we consider the set $C(\alpha)^{\vee}$ of ring 
class characters of conductor $\mathfrak{p}^{\alpha}$, 
and in particular the subset $P(\alpha) \subset C(\alpha)^{\vee}$ of primitive characters which do not factor through $C(\alpha-1)^{\vee}$. 
We also consider the following subaverages of these primitive characters, in the style of Cornut-Vatsal \cite{CV}. 
Writing $C(\infty) = \varprojlim_{\alpha \geq 0} C(\alpha)$ to denote the inverse limit, let $C_0 = C(\infty)_{\operatorname{tors}}$
denote the finite torsion subgroup. Given a character $\rho_0$ of $C_0$, we also consider weighted 
subaverages over primitive characters $P(\alpha, \rho_0) \subset P(\alpha) \subset C(\alpha)^{\vee}$ of conductor $\mathfrak{p}^{\alpha}$
which induce this chosen character $\rho_0$ on the torsion subgroup $C_0$ of $C(\infty)$. 
As well, writing $x = \operatorname{ord}_p(\#C(\alpha))$,
we consider the subaverages over primitive ring class characters of $C(\alpha)$ of exact order $p^x$.

\begin{theorem}[Theorem \ref{RCOD}, Corollary \ref{RCGAnv}]\label{MAIN} 

Let $\pi$ be a cuspidal $\operatorname{GL}_2({\bf{A}}_F)$-automorphic representation of conductor $c(\pi) \subset \mathcal{O}_F$,
which we assume to be self-dual, $\pi \cong \widetilde{\pi}$.
Fix a prime ideal $\mathfrak{p} \subset \mathcal{O}_F$ with underlying rational prime $p$. Let $K$ be a totally imaginary quadratic 
extension of $F$ of relative discriminant $\mathfrak{D} \subset \mathcal{O}_F$ and associated character $\eta$ of $F$. 
Assume that $(c(\pi), \mathfrak{D}\mathfrak{p}) = (\mathfrak{p}, \mathfrak{D}) = 1$. 
Fix a primitive even Dirichlet character $\chi \bmod p^{\beta}$ for some integer $\beta \geq 0$. 
In the special case where $\beta=0$, let us also also assume that the generic root number $\epsilon(1/2, \pi \times \rho) = (-1)^{\eta}$ 
described above is equal to $1$ (as opposed to $-1$).\footnote{to rule out forced vanishing from the functional equation}
If the exponent $\alpha \gg 1$ is sufficiently large relative to $\beta$, then there exists a primitive ring class character $\rho$ of 
conductor $\mathfrak{p}^{\alpha}$, for which $L(1/2, \pi \times \rho \chi \circ {\bf{N}}) \neq 0$. 
Moreover, for any choice of character $\rho_0$ of $C_0$, there exists a character $\rho(\alpha, \rho_0)$ 
for which $L(1/2, \pi \times \rho \chi \circ {\bf{N}}) \neq 0$. Writing $x = \operatorname{ord}_p(\#C(\alpha))$,
we can also find a primitive ring class character $\rho$ of conductor $\mathfrak{p}^{\alpha}$ and exact order $p^x$
for which $L(1/2, \pi \times \rho \chi \circ {\bf{N}}) \neq 0$. \end{theorem}

More generally, using these main results as input, we then go on to show the following refinements. 
First, we use the rationality theorem of Shimura \cite{Sh} to deduce the following generic nonvanishing properties, 
generalizing the theorems of Rohrlich \cite{Ro}, \cite{Ro2}, \cite{Ro3}, Vatsal \cite{Va}, and Cornut-Vatsal \cite{CV}. 

\begin{theorem}[Theorem \ref{GAnv}] 

Let $\pi \cong \widetilde{\pi}$ be a cuspidal $\operatorname{GL}_2({\bf{A}}_F)$-automorphic representation of conductor $c(\pi) \subset \mathcal{O}_F$ 
which corresponds to a holomorphic Hilbert modular form of weight $k = (k_j)_{j=1}^d$ with each $k_j \geq 2$. 
Fix a prime ideal $\mathfrak{p} \subset \mathcal{O}_F$ with underlying rational prime $p$.
Let $K/F$ be a totally imaginary quadratic extension of relative discriminant $\mathfrak{D} \subset \mathcal{O}_F$
and absolute discriminant $D_K$.
Assume that $(c(\pi), \mathfrak{D}\mathfrak{p}) = (\mathfrak{p}, \mathfrak{D}) = 1$. 
Assume the Hecke field ${\bf{Q}}(\pi)$ of $\pi$ is linearly disjoint over ${\bf{Q}}$ 
to the cyclotomic tower obtained by adjoining all $p$-power roots of unity ${\bf{Q}}(\zeta_{p^{\infty}}) = \bigcup_{n \geq 1} {\bf{Q}}(\zeta_{p^n})$. 
Fix a primitive even Dirichlet character $\chi \bmod p^{\beta}$ for some integer $\beta \geq 0$. 
In the event that $\beta =0$ (hence $\chi$ trivial), let 
us also assume that the generic root number $\epsilon(1/2, \pi \times \rho)$ for $\rho$ 
ranging over primitive ring class characters characters of conductor $\mathfrak{p}^{\alpha}$ with $\alpha \gg 1$ sufficiently large 
is not equal to $-1$. Then for each sufficiently large integer $\beta \geq 1$, there exists a primitive ring class character $\rho$ 
of conductor $\mathfrak{p}^{\alpha}$ for which the Galois average $G_{[\rho \chi \circ {\bf{N}}]}(\pi)$ does not vanish, 
and so $L(1/2, \pi \times \mathcal{W}) = L(1/2, \pi \times \rho \chi \circ {\bf{N}})$ does not vanish for $\mathcal{W} = \rho \chi \circ {\bf{N}}$ 
ranging over such Hecke characters taking values in roots of unity of exact order $\operatorname{lcm}(p^{\beta}, \operatorname{ord}(\rho))$,
i.e.~where $\operatorname{ord}(\rho) \mid (\#C(\alpha) - \#C (\alpha-1))$ denotes the exact order of the character $\rho$. 
\end{theorem}

We can also derive the following results via specialization of suitable multivariable $\mathfrak{p}$-adic $L$-functions.
Recall that a for a fixed prime ideal $\mathfrak{p} \subset \mathcal{O}_F$ with underlying rational prime $p$, 
a $\operatorname{GL}_2({\bf{A}}_F)$-automorphic representation is said to be $\mathfrak{p}$-ordinary if its 
Hecke eigenvalue at $\mathfrak{p}$ is an algebraic number whose image under a fixed embedding 
$\overline{\bf{Q}} \rightarrow \overline{\bf{Q}}_p$ is a $p$-adic unit. Let $K_{\infty}^{(\mathfrak{p} )}$ denote the compositum 
of the tower of all ring class extensions of $K$ of $\mathfrak{p}$-power conductor with the cyclotomic extension obtained by 
adjoining all $p$-power roots of unity to $K$. Let $\mathcal{G} = \operatorname{Gal}(K_{\infty}^{(\mathfrak{p})}/K)$ denote its Galois group. 
Composing with the Artin reciprocity map, we view the characters $\mathcal{W}$ of $K$ described above as finite order
characters factoring through $\mathcal{G}$. Let us also write $\Omega \approx {\bf{Z}}_p^{\delta}$ to denote the Galois group of the anticyclotomic 
${\bf{Z}}_p^{\delta}$-extension of $K$, where $\delta = \delta_{\mathfrak{p}} = [F_{\mathfrak{p}}: {\bf{Q}}_p]$ 
denotes the residue degree of $\mathfrak{p}$, and $\Gamma \cong {\bf{Z}}_p$ the Galois group of the cyclotomic ${\bf{Z}}_p$-extension of $K$. 

\begin{theorem}[Theorem \ref{ACmu}, cf.~Corollary \ref{cycmu}]

Let $\pi \cong \widetilde{\pi}$ be a cuspidal $\operatorname{GL}_2({\bf{A}}_F)$-automorphic representation of conductor $c(\pi) \subset \mathcal{O}_F$ 
which corresponds to a holomorphic Hilbert modular form of weight $k = (k_j)_{j=1}^d$ with each $k_j \geq 2$. 
Fix a prime $\mathfrak{p} \subset \mathcal{O}_F$ with underlying rational prime $p$.
Let $K/F$ be a totally imaginary quadratic extension of relative discriminant $\mathfrak{D}\subset \mathcal{O}_F$ 
and absolute discriminant $D_K$.
Assume that $\pi$ is $\mathfrak{p}$-ordinary, that $(c(\pi), \mathfrak{D}\mathfrak{p}) = (\mathfrak{p}, \mathfrak{D}) = 1$, 
and that the residual Galois representation associated to $\pi$ by constructions Carayol \cite{Ca}, Taylor \cite{Ta}, and Wiles \cite{Wi} is absolutely irreducible. 
Fix a character $\mathcal{W}_0$ of the torsion subgroup $\mathcal{G}_{\operatorname{tors}}$ of $\mathcal{G}$. There exists a minimal exponent 
$\alpha_0 \geq 0$ such that for all characters $\rho_w$ of $\Omega$ of exact order $p^{\alpha}$ with $\alpha \geq \alpha_0$, the central value 
$L(1/2, \pi \times \mathcal{W}_0 \rho_w \psi_w)$ does not vanish for any character $\psi_w$ of the cyclotomic Galois group $\Gamma \cong {\bf{Z}}_p$.

\end{theorem} 

Finally, we derive the following arithmetical applications of these results in a self-contained appendix. 
Namely, we use work of Xin Wan \cite{XW} on the Iwasawa main conjecture together with existing work on anticyclotomic main
conjectures (e.g.~\cite{Nek}) to deduce the following results subject to standard technical hypotheses. 
Fix $\pi$ a cuspidal $\operatorname{GL}_2({\bf{A}}_F)$-automorphic representation corresponding 
to a holomorphic Hilbert modular form of parallel weight two and trivial character. Writing ${\bf{Q}}(\pi)$
again to denote the Hecke field of $\pi$, le us assume (as is often known) that we can attach to $\pi$ an 
abelian variety $A = A_{\pi}$ defined over $F$ such that (i) the dimension of $A$ is equal to the degree of ${\bf{Q}}(\pi)$,
(ii) the ring of endomorphisms $\operatorname{End}_F(A)$ is given by the ring of integers $\mathcal{O}_{{\bf{Q}}(\pi)}$ of $\pi$, 
and (iii) the Hasse-Weil $L$-function $L(s, A/F)$ of $A/F$ is given by that of $\pi$, in other words: $L(s-1/2, \pi) = (2/2\pi) \Gamma(s) L(s, A/F)$.
Let $K_{\infty}$ denote the compositum of the anticyclotomic ${\bf{Z}}_p^{\delta}$-extension of $K$ with the 
cyclotomic ${\bf{Z}}_p$-extension of $K$, where $\delta = [F_{\mathfrak{p}}: {\bf{Q}}_p]$ again denotes the residue degree of $\mathfrak{p}$.
Hence, the corresponding Galois group $G = \operatorname{Gal}(K_{\infty}/K)$ is isomorphic as a topological group to ${\bf{Z}}_p^{\delta + 1}$. 
We can then consider the Mordell-Weil group $A(K_{\infty})$ of $K_{\infty}$-rational points of $A$. 
We can also consider the corresponding Tate-Shafarevich group $\Sha(A/K_{\infty})$, 
or more precisely its $p$-primary subgroup $\Sha(A/K_{\infty})[p^{\infty}]$. 
We deduce the following theorems for these groups, subject to the various hypotheses in \cite{XW} (as summarized below) and our discussion of $p$-adic $L$-functions: 

\begin{theorem}[Theorem \ref{isotyp}]

Let $\pi \cong \widetilde{\pi}$ be a cuspidal $\operatorname{GL}_2({\bf{A}}_F)$-automorphic representation corresponding to a holomorphic 
Hilbert modular form of parallel weight two and trivial character. Assume that $\pi$ has associated to it an abelian variety 
$A = A_{\pi}$ as described above, and fix a prime $p \geq 5$. 
Let $K/F$ be a totally imaginary quadratic extension of 
relative discriminant $\mathfrak{D}\subset \mathcal{O}_F$ and absolute discriminant $D_K$.
Assume the conditions of Hypothesis \ref{XW} and Theorem \ref{IMC} 
below are met, that $(c(\pi), \mathfrak{D}\mathfrak{p}) = (\mathfrak{p}, \mathfrak{D}) = 1$, and that the Hecke field ${\bf{Q}}(\pi)$ is linearly 
disjoint over ${\bf{Q}}$ the cyclotomic tower obtained by adjoining all $p$-power roots of unity ${\bf{Q}}(\zeta_{p^{\infty}})$. 
Let $\rho = \rho_w$ be a ring class character factoring through the Galois group $G = \operatorname{Gal}(K_{\infty}/K)$. 
There exists an integer $\beta_0(\rho)$ such that for all characters $\psi = \psi_w$ of the cyclotomic Galois group 
$\Gamma = \operatorname{Gal}(K^{\operatorname{cyc}}/K)$ of exact order $p^{\beta}$ with $\beta \geq \beta_0(\rho)$, 
the central value $L(1/2, \pi \times \rho \psi)$ does not vanish, and hence the corresponding $\rho \psi$-isotypical components 
of both $A(K_{\infty})$ and $\Sha(A/K_{\infty})[p^{\infty}]$ are finite. 
\end{theorem}

To state the second result in a concise way, let us also write $\epsilon = \epsilon(1/2, A/K)$ to denote the sign in the functional equation
of the Hasse-Weil $L$-function $L(s, A/K)$ of $A$ over $K$.

\begin{theorem}[Theorem \ref{rank}]

Let $\pi \cong \widetilde{\pi}$ be a cuspidal $\operatorname{GL}_2({\bf{A}}_F)$-automorphic representation corresponding to a holomorphic 
Hilbert modular form of parallel weight two and trivial character. 
Let $K/F$ be a totally imaginary quadratic extension of relative discriminant $\mathfrak{D}\subset \mathcal{O}_F$ 
and absolute discriminant $D_K$.
Assume that $\pi$ has associated to it an abelian variety $A = A_{\pi}$ 
as described above, and that $(c(\pi), \mathfrak{D}\mathfrak{p}) = (\mathfrak{p}, \mathfrak{D}) = 1$. 
Assume as well that the following conditions hold: (1) if $A$ acquires CM after basechange to some quadratic extension 
$K_{\pi}/F$, then this extension $K_{\pi}$ is not contained in $K_{\infty}$ when $\epsilon = +1$, 
and (2) $A$ has good ordinary reduction at all primes above $p$ in $F$. 
Finally, if the residue degree $\delta = [F_{\mathfrak{p}}: {\bf{Q}}_p]$ is greater than one, let us assume additionally that the conditions of Theorems 
\ref{GAnv}, \ref{ACmu}, and \ref{IMC} below (including the vanishing of the anticyclotomic $\mu$-invariant) hold, and that the absolute discriminant 
$D_K$ is sufficiently large. Then, $A(K_{\infty})$ is finitely generated if $\epsilon = +1$, 
and $A(K_{\infty})/A(D_{\infty})$ is finitely generated if $\epsilon = -1$. \end{theorem}

\subsection{Outline of the proof of Theorem \ref{MAIN}}

Let us now give a high-level sketch of how the main analytic result Theorem \ref{MAIN} is derived. 
We consider the weighted average $P_{\alpha}(\pi, \chi)$ over central values $L(1/2, \pi \times \rho \chi \circ {\bf{N}})$ 
with $\rho$ varying over primitive ring class characters of conductor $\mathfrak{p}^{\alpha}$, as well as the subaverage 
$P_{\alpha, \rho}(\pi, \chi)$ with $\rho \in P(\alpha, \rho_0)$ varying over those characters inducing a given character $\rho_0$ 
on the torsion subgroup $C_0 = C(\infty)_{\operatorname{tors}}$.  
More precisely, we describe the values in the average using an unbalanced approximate functional equation, 
and this reduces us to looking at sums of the form $D_{A, 1}(\pi, \chi; Z) + D_{A, 2}(\pi, \chi; Z)$ for any choice of real parameter $Z >0$, where 
\begin{align*} D_{A, 1}(\pi, \chi; Z) &= \frac{1}{w_K} \sum_{\mathfrak{m} \subset \mathcal{O}_F} 
\frac{\omega \eta(\mathfrak{m}) \chi^2({\bf{N}} \mathfrak{m})}{{\bf{N}} \mathfrak{m}} 
\sum_{a, b \in \mathcal{O}_F / / \mathcal{O}_F^{\times}} \frac{ \lambda(f_A(a, b))\chi(f_A(a,b)) }{{\bf{N}} f_A(a, b)^{\frac{1}{2}}} 
V_1 \left( {\bf{N}}(\mathfrak{m}^2 f_A(a,b)) Z\right) \end{align*} and 
\begin{align*} D_{A, 2}(\pi, \chi; Z) &= \frac{1}{w_K} \sum_{\mathfrak{m} \subset \mathcal{O}_F} 
\frac{ \overline{\omega} \eta(\mathfrak{m}) \overline{\chi}^2({\bf{N}} \mathfrak{m})}{{\bf{N}} \mathfrak{m}} 
\sum_{a, b \in \mathcal{O}_F / \mathcal{O}_F^{\times}} \frac{\overline{ \lambda(f_A(a, b))\chi(f_A(a, b)) } }{{\bf{N}}f_A(a, b)^{\frac{1}{2}}} 
V_2 \left( \frac{ {\bf{N}}(\mathfrak{m}^2 f_A(a,b))}{ Z C  }\right). \end{align*}
Here, the sums run over pairs of $F$-integers $a, b \in \mathcal{O}_F$ modulo the action of units $\mathcal{O}_F^{\times}$,
and can be viewed equivalently as sums over principal ideals in $(a), (b) \subset \mathcal{O}_F$.
The coefficients $\lambda = \lambda_{\pi}$ are the $L$-functions coefficients of $\pi$, 
so that the finite part of the $L$-function of $\pi$ has the expansion 
$L(s, \pi) = \sum_{\mathfrak{n} \subset \mathcal{O}_F} \lambda(\mathfrak{n}) {\bf{N}} \mathfrak{n}^{-s}$ for $\Re(s) >1$, 
writing ${\bf{N}}$ as usual to denote the absolute norm. The function $f_A(x, y) = a_A x^2 + b_A xy + c_A y^2$ 
denotes a fixed $F$-rational positive definite binary quadratic form ($a_A, b_A, c_A \in \mathcal{O}_F$) representing
the class $A$ in the class group $C(\alpha)$ of the order
$\mathcal{O}_{\mathfrak{p}^{\alpha}} = \mathcal{O}_F + \mathfrak{p}^{\alpha} \mathcal{O}_K$, and $w_K$ denotes the number of automorphs of $f_A$.
We assume that $f_A(x, y)$ is the reduced class representative, and hence ${\bf{N}} b_A \leq {\bf{N}} a_A \leq {\bf{N}}c_A$.
Let us remark as well that only the principal class contributes to the primitive average $P_{\alpha}(\pi, \chi)$, 
and only classes factoring through the torsion subgroup $C_0$ to the corresponding subaverage $P_{\alpha, \rho_0}(\pi, \chi)$.
As we explain later, there are constraints on the possible coefficients $a_A$ we can consider with this method, 
although a variation which we develop via decompositions into Poincar\'e series in the style of \cite{VB} (see Theorem \ref{SCS2}) 
allows us to proceed in a conceptually similar way irrespective of the relative sizes of these coefficients. 
The $\mathfrak{m}$-sums run over nonzero integral ideals in $\mathcal{O}_F$, 
and the $a, b$-sums over nonzero $F$-integers. The functions $V_j$ are smooth and rapidly decaying 
cutoff functions coming from our choice of approximate functional equation (Lemma \ref{AFE}).
Finally, writing $c(\pi_K) \subset \mathcal{O}_K$ to denote the conductor of the quadratic basechange
representation $\pi_K = \operatorname{BC}_{K/{\bf{Q}}}(\pi)$ associated to $\pi$, 
\begin{align*} C = {\bf{N}}(\mathfrak{D}^2 c(\pi_K) \cdot c(\rho \cdot \chi \circ {\bf{N}})^2) =
{\bf{N}}\left( \mathfrak{D}^2 c(\pi_K) \cdot \left( \operatorname{lcm}(\mathfrak{p}^{\alpha}, p^{\beta} \mathcal{O}_F) \mathcal{O}_K\right)^2 \right)
= {\bf{N}}(\mathfrak{D}^2 c(\pi_K)) \cdot p^{4d \max(\alpha, \beta)}\end{align*} 
denotes the conductor of each $L$-function appearing in each of the averages $P_{\alpha}(\pi, \chi)$ and $G_{\alpha}(\pi, \chi; x)$. 

We present two methods of estimating the sums $D_{A, j}(\pi, \chi; Z)$, both using Kirillov models to derive novel integral
presentations for the sums we consider in terms of Fourier-Whittaker coefficients of some distinct (non-$\mathcal{K}$-finite) automorphic
forms which can then be decomposed spectrally to derive estimates. 
On the one hand, we can approximate off-diagonal sums in terms of metaplectic Fourier-Whittaker coefficients after 
taking the unbalancing parameter $Z $ of size approximately ${\bf{N}}(\mathfrak{D} \mathfrak{p}^{2 \alpha})^{-1}$ 
to be the inverse of the discriminant of the order $\mathcal{O}_{\mathfrak{p}^{\alpha}} = \mathcal{O}_F + \mathfrak{p}^{\alpha} \mathcal{O}_K$, 
i.e.~essentially the inverse of the square root of the conductor $C$. 
Here, the contributions from $b=0$ terms are estimated in terms of a residual Dirichlet series related to the symmetric square 
$L$-function $L(s, \operatorname{Sym}^2 \pi \otimes \chi \circ {\bf{N}})$ at $s=1$ (Proposition \ref{RCres}). 
These values in particular do not vanish, and moreover can be bounded from below 
in terms of the conductor (see \cite{GHL}, \cite[(1.5)]{CM}, and also $(\ref{lower})$ below). 
As we explain for Theorem \ref{SCS} and Theorem \ref{RCOD} below (see also Theorem \ref{binary}), 
the remaining coefficients in the expression for $D_{A, 1}(\pi, \chi; Z)$ can be described equivalently in terms of Fourier coefficients of  
certain automorphic forms on $\operatorname{GL}_2({\bf{A}}_F)$ and its two fold metaplectic cover $\overline{G}({\bf{A}}_F)$. 
This enables us to use spectral decompositions of such forms to estimate these sums. 
For instance, taking the $Z = Y^{-1} = C^{- \frac{1}{2}}$ of size approximately ${\bf{N}}(\mathfrak{D} \mathfrak{p}^{2 \alpha})^{-1}$ so that the length
of the sums is equal to the square root of the conductor (corresponding to a balanced approximate functional equation formula), 
and assuming $D_K \equiv 0 \bmod 4$, we use spectral decompositions of shifted convolution sums to derive the following 
estimate for the average over primitive ring class characters (see Proposition \ref{RCres} and Theorem \ref{RCOD} (i)): 
For $\alpha \gg 1$ sufficiently large, 
\begin{equation*}\begin{aligned} &P_{\alpha}(\pi, \chi)\\ 
&= \left( 1 - \frac{1}{{\bf{N}} \mathfrak{p}} \right) \frac{1}{w_K} \left( L(1, \omega \eta \chi^2 \circ {\bf{N}}) \cdot 
\frac{ L_{\bf{1}}^{\star}(1, \operatorname{Sym}^2 \pi \otimes \chi \circ {\bf{N}}) }{ L_{\bf{1}}^{\star}(2, \omega \chi^2 \circ {\bf{N}})} 
+ \epsilon \cdot \frac{\widetilde{L}_{\infty}(\frac{1}{2})}{L_{\infty}(\frac{1}{2})} \cdot L(1, \overline{\omega} \eta \overline{\chi}^2 \circ {\bf{N}}) \cdot 
\frac{ L_{\bf{1}}^{\star}(1, \operatorname{Sym}^2 \widetilde{\pi} \otimes \overline{\chi} \circ {\bf{N}}) }{ L_{\bf{1}}^{\star}(2, \overline{\omega} \overline{\chi}^2 \circ {\bf{N}})} \right) \\ 
&+ O_{\pi, \varepsilon} \left(  \left( D_K p^{2 d \beta} \right)^{\frac{1}{4} - \frac{(1-2\theta_0)}{16} + \varepsilon}  Y^{-\frac{1}{4} - \varepsilon} \right)
+ O_{\pi, \chi, \varepsilon} \left( Y^{\frac{1}{4} + \delta_0 + \varepsilon} {\bf{N}}(\mathfrak{D} \mathfrak{p}^{2 \alpha})^{-\frac{1}{2}} \right), \end{aligned}\end{equation*}
Here, $L_{\bf{1}}^{\star}(s, \operatorname{Sym}^2 \pi \otimes \chi \circ {\bf{N}})$ is essentially the partial Dirichlet series expansion over 
principal ideals of the symmetric square $L$-function $L(s, \operatorname{Sym}^2 \pi \otimes \chi \circ {\bf{N}})$ 
(up to a convergent finite product of Euler factors), and $L^{\star}_{\bf{1}}(s, \omega \chi^2 \circ {\bf{N}})$ is defined 
similarly with respect to the Hecke $L$-function $L(s, \omega \chi^2 \circ {\bf{N}})$. 
The \begin{align*} \epsilon = \epsilon(1/2, \pi \times \rho \chi \circ {\bf{N}}) &= 
\omega(\operatorname{lcm}(\mathfrak{p}^{\alpha}, p^{\beta} \mathcal{O}_F)) \cdot \eta(p^{4 \beta} \mathfrak{d}c(\pi)) \cdot 
\epsilon(1/2, \pi) \cdot \chi( {\bf{N}}(\mathfrak{d}^2 c(\pi)^2 \mathfrak{D}^8)) 
\cdot \left( \frac{\tau(\chi^2)}{p^{\frac{\beta}{2}}} \right)^{4d}\end{align*} 
denotes the (stable) root number for each primitive ring class character $\rho$ of conductor $\mathfrak{p}^{\alpha}$ in the average
(see Proposition \ref{RNF} and Lemma \ref{RCGA}), $L_{\infty}(s)$ the archimedean local factor of each completed $L$-function 
$\Lambda(s, \pi \times \rho \times \chi \circ {\bf{N}})$, and $\widetilde{L}_{\infty}(s)$ that of 
each $\Lambda(s, \widetilde{\pi} \times \rho \overline{\chi} \circ {\bf{N}})$. 
As well, we write $0 \leq \theta_0 \leq 1/2$ to denote the best approximation to the 
generalized Ramanujan conjecture for $\operatorname{GL}_2({\bf{A}}_F)$-automorphic
forms, with $\theta_0 = 7/64$ an admissible choice by theorem of Blomer-Brumley \cite{BBRP}. 
The Burgess-like exponent $1/4 - (1-2\theta_0)/16$ of Wu \cite{Wu2} in the first error term can then be taken to be $206/1024$. 
We refer to Theorems \ref{SCS} and \ref{RCOD} for how the off-diagonal bounds are proved.
The exponent in the error term here has the more organic form $- 1/4 + \delta_0 + \varepsilon$,
where $0 \leq \delta_0 \leq 1/4$ denotes the best approximation to the generalized Lindel\"of hypothesis for
$\operatorname{GL}_2({\bf{A}}_F)$-automorphic forms in the level aspect, or what is the same -- by the theorem
of Kohnen-Zagier \cite{KZ} and more generally Baruch-Mao \cite{BM07} -- the best exponent approximation for the 
Fourier coefficients of half-integral weight forms. That is, this exponent reflects our approximation of the off-diagonal sum 
over $b \neq 0$ contributions (in the region of moderate decay for $V_1$) by Fourier coefficients of genuine metaplectic forms, 
which we decompose spectrally to derive such a bound. We then use the admissible approximation of $\delta_0 = 103/512$ of 
Blomer-Harcos \cite[Corollary]{BH10}. 
We also develop a distinct and flexible variation of this idea via Theorem \ref{SCS2} using decompositions into Poincar\'e series
in the style of the argument of Blomer \cite{VB},
which allows us to deal with the coefficients appearing in the reduced form representative $f_A(x, y)$ corresponding to $A \in C(\alpha)$.
This allows us to deal with an inherent limitation in the standard setup described above which requires $a_A=1$ or 
small relative to $c_A$ and $b_A=0$, and in particular to derive estimates for the sums corresponding to any class $A \in C(\alpha)$.
This latter feature in turn allows us to estimate the tame and Galois sub-averages directly, via purely analytic methods.\footnote{When
the cyclotomic character $\chi$ is trivial, then the theorems of Cornut and Vatsal prove the nonvanishing of the tame sub-averages 
using ergodic theory, i.e.~using the theorems of Ratner and Margulis-Tomanov on $p$-adic unipotent flows.}

In all our main estimates for the sums $D_{A, j}(\pi, \chi; Z)$, the surjectivity of the archimedean local Kirillov map plays a starring role. 
As we explain in Proposition \ref{choice} (cf.~ also $(\ref{or})$), this property allows us to find 
automorphic forms whose Fourier-Whittaker coefficients describe the sums $D_{A, 1}(\pi, \chi; Z)$ and $D_{A, 2}(\pi, \chi; Z)$ we 
consider exactly. Once such presentations are known, the door is open to using spectral decompositions of the corresponding forms, 
and in particular to deriving estimates for both of the sums $D_{A, 1}(\pi, \chi; Z)$ and $D_{A, 2}(\pi, \chi; Z)$ with a flexible choice of 
unbalancing parameter $Z >0$. 

\subsubsection*{Acknowledgements} I am grateful to Valentin Blomer (especially), Henri Darmon, Dorian Goldfeld, Ralph Greenberg, 
Gergely Harcos, Philippe Michel, Djordje Milicevic, Peter Sarnak, Michael Spiess, and Otmar Venjakob for helpful discussions.
I am also grateful to anonymous referees for helpful constructive criticism. 

\section{Mean values} 

\subsection{Rankin-Selberg $L$-functions}

We first review some relevant background from the theory of Jacquet \cite{Ja} and Jacquet-Langlands \cite{JL} for the Rankin-Selberg $L$-functions we consider. 

\subsubsection{Setup and definitions} 

Let us fix a totally imaginary quadratic extension $K$ of $F$ of relative discriminant $\mathfrak{D} = \mathfrak{D}_{K/F}$ 
and associated idele class character $\eta = \eta_{K/F}$ of $F$. Let us also write $D_K = {\bf{N}} \mathfrak{D}$ to denote 
the absolute discriminant of $K$. Let $\mathcal{W}$ denote a Hecke character of $K$, 
with $\pi(\mathcal{W}) = \otimes_v \pi(\mathcal{W})_v$ the associated induced $\operatorname{GL}_2({\bf{A}}_F)$-automorphic representation. 
Equivalently, $\pi(\mathcal{W})$ denotes the automorphic representation of $\operatorname{GL}_2({\bf{A}}_F)$ 
generated by the Hilbert modular theta series $\theta(\mathcal{W})$ associated to $\mathcal{W}$. 
We shall consider the Rankin-Selberg $L$-function of $\pi$ times the induced representation $\pi(\mathcal{W})$, 
whose Euler product over finite places as a function of $s \in {\bf{C}}$ with $\Re(s) > 1$ we express as 
\begin{align*} L(s, \pi \times \mathcal{W}) = L(s, \pi \times \pi(\mathcal{W})) 
&= \prod_{v < \infty} L(s, \pi_v \times \pi(\mathcal{W})_v). \end{align*} 
Here, for primes $v$ not dividing the conductor $c(\pi \times \mathcal{W})$ of $L(s, \pi \times \mathcal{W})$, 
the corresponding local factor $L(s, \pi_v \times \mathcal{W}_v)$ takes the form 
\begin{align*} L(s, \pi_v \times \mathcal{W}_v) &= \det(I - A_v \otimes B_v {\bf{N}} v^{-s})^{-1}, \end{align*} 
where $A_v$ denotes the Satake parameter of $\pi_v$ and $B_v$ that of $\pi(\mathcal{W})_v$. 
For primes $v$ where one of $\pi$ or $\pi(\mathcal{W})$ is ramified, the corresponding local factor $L(s, \pi_v \times \pi(\mathcal{W})_v)$ 
takes the form of $P_v( {\bf{N}} v^{-s})^{-1}$ for $P_v(x)$ a polynomial of degree at most four such that $P_v(0)=1$. 
Let us also write $L(s, \pi_{\infty} \times \pi(\mathcal{W})_{\infty})$ to denote the archimedean component of this $L$-function. 
If $\mathcal{W} = \rho$ is a ring class character, or more generally if $\mathcal{W} = \rho \chi \circ {\bf{N}}$ 
is the product of a ring class character $\rho$ with the composition $\chi \circ {\bf{N}}$ of a primitive even Dirichlet character $\chi$
with the norm homomorphism {\bf{N}}, then $\mathcal{W}$ determines 
a wide ray class character with ``trivial archimedean component" $\mathcal{W}_{\infty} \equiv {\bf{1}}$. 
Consequently, the archimedean local factor of the $L$-function does not depend on the choice of $\mathcal{W}$,
i.e.~as $\pi(\mathcal{W})_{\infty} \equiv {\bf{1}}$ for any such character, and we are justified in dropping the $\mathcal{W}$ from the notation. 
Hence, we write $L_{\infty}(s) = L(s, \pi_{\infty} \times \pi(\mathcal{W})_{\infty})$. The completed $L$-function 
\begin{align*} \Lambda(s, \pi \times \mathcal{W}) &= L(s, \pi \times \mathcal{W}) L(s, \pi_{\infty} \times \pi(\mathcal{W})_{\infty}) \end{align*} 
is entire unless $\pi(\mathcal{W}) \approx \widetilde{\pi} \otimes \vert \cdot \vert^{t}$ for some $t \in {\bf{R}}$, 
and in any case holomorphic except for simple poles at $s=0$ and $1$. It satisfies the functional equation 
\begin{align*} \Lambda(s, \pi \times \mathcal{W}) &= \epsilon(s, \pi \times \mathcal{W}) 
\Lambda(1-s, \widetilde{\pi} \times \overline{\mathcal{W}}),\end{align*} 
where 
\begin{align*} \epsilon(s, \pi \times \mathcal{W}) 
&= \epsilon(1/2, \pi \times \mathcal{W}) c(\pi \times \mathcal{W})^{\frac{1}{2}-s} \end{align*}
denotes the $\epsilon$-factor of $\Lambda(s, \pi \times \mathcal{W})$. 
Here, $\epsilon(1/2, \pi \times \mathcal{W}) \in {\bf{S}}^1$ is the root number. 
Note that this root number also admits an Euler product decomposition 
\begin{align}\label{RN}\epsilon(1/2, \pi \times \mathcal{W}) 
&= \prod_v \epsilon (1/2, \pi_v \times \mathcal{W}_v, \psi_v) = \prod_v \epsilon(1/2, \pi_v \times \pi(\mathcal{W})_v, \psi_v)\end{align} 
(where the local Euler factors are defined with respect to any fixed choice of additive character $\psi = \otimes_v \psi_v$),
and can be given by a more explicit formula when we assume that the conductor $c(\pi)$ is prime to that of $\pi(\mathcal{W})$ (see below). 
As well, we shall write 
\begin{align}\label{conductor} c(\pi \times \mathcal{W}) &=  {\bf{N}} (\mathfrak{D}_{K}^2 c(\pi_K) c(\mathcal{W})^2) \end{align} 
to denote the conductor of $\Lambda(s, \pi \times \mathcal{W})$, where $c(\pi_K)$ denotes that of the basechange 
$\pi_K$ of $\pi$ to $\operatorname{GL}_2({\bf{A}}_K)$, and $c(\mathcal{W})$ that of the Hecke 
character $\mathcal{W}$, viewed as an ideal of $\mathcal{O}_K$ (cf.~\cite[(16)]{BR}). We shall sometimes work with the square 
root of this quantity $Y := c(\pi \times \mathcal{W})^{\frac{1}{2}}$ for our arguments below, and remind the reader that taking (relative) 
norms of the conductor $c(\mathcal{W}) \subset \mathcal{O}_K$ to $\mathcal{O}_F$ or ${\bf{Z}}$ leads to fourth powers of the moduli.

\begin{remark}[Relations to basechange $L$-functions.] 

Note that the $\operatorname{GL}_2({\bf{A}}_F) \times \operatorname{GL}_2({\bf{A}}_F)$ 
Rankin-Selberg $L$-function $\Lambda(s, \pi \times \mathcal{W})$ is equivalent to the 
$\operatorname{GL}_2({\bf{A}}_K) \times \operatorname{GL}_1({\bf{A}}_K)$ $L$-function given by $\Lambda(s, \pi_K \otimes \mathcal{W})$, 
where $\pi_K$ denotes the basechange of $\pi$ to $\operatorname{GL}_1({\bf{A}}_K)$, and $L(s, \pi_K \otimes \mathcal{W})$ 
the $L$-function of $\pi_K$ twisted by the Hecke character $\mathcal{W}$ of $K$. If $\mathcal{W} = {\bf{1}}_K$ is the trivial (class group) Hecke character of $K$, 
then we also have the Artin decomposition $\Lambda(s, \pi_K \times {\bf{1}}_K) = \Lambda(s, \pi) \Lambda(s, \pi \otimes \eta)$. 
In any case, the formula $(\ref{conductor})$ for the conductor $c(\pi \times \mathcal{W})$ is equivalent to that of the conductor of the 
basechange $L$-function $c(\pi_K \otimes \mathcal{W})$, and we have taken the relevant formula for the latter basechange conductor 
as described in \cite[(16)]{BR} (cf.~\cite{Ro3}). \end{remark}

\subsubsection{Explicit description of the root number} 

Let us now give a more explicit description of the root number $(\ref{RN})$ defined above 
(cf.~ \cite[$2.1$]{LRS}, \cite[Proposition 4.1]{BR}). We assume from now on that $\mathcal{W}$ 
is a Hecke character of $K$ of the form described above, hence $\mathcal{W} = \rho \chi \circ {\bf{N}}$
with $\rho$ a primitive ring class character of conductor $\mathfrak{p}^{\alpha}$ for some integer $\alpha \geq 0$, 
and $\chi$ some primitive even Dirichlet character of conductor $p^{\beta}$ for some integer $\beta \geq 0$.
Note that $\mathcal{W}$ is then always a wide ray class character in our setup, and the archimedean local component 
$L_{\infty}(s) = L(s, \pi_{\infty} \times \pi(\mathcal{W})_{\infty})$ does not depend on the choice of such a character $\mathcal{W}$.
That is, the archimedean local component $L(s, \pi_{\infty} \times \pi(\mathcal{W}))$ of the completed Rankin-Selberg $L$-function
\begin{align*} \Lambda(s, \pi \times \mathcal{W}) = \Lambda(s, \pi \times \pi(\mathcal{W})) 
= \Lambda(s, \pi_{\infty} \times \pi(\mathcal{W})_{\infty}) L(s, \pi \times \pi(\mathcal{W})) =: L_{\infty}(s) L(s, \pi \times \mathcal{W}) \end{align*}
does not change as we vary over all Hecke characters $\mathcal{W} = \rho \chi \circ {\bf{N}}$ of $K$ 
with $\rho$ a primitive ring class character and $\chi$ a primitive even Dirichlet character. Note that we shall always
identify such a wide ray class character $\mathcal{W}$ with its corresponding idele class character of $K$.

Let us now assume that the conductor $\operatorname{lcm}(\mathfrak{p}^{\alpha}, p^{\beta} \mathcal{O}_F) \subset \mathcal{O}_F$ 
of $\mathcal{W} = \rho \chi \circ {\bf{N}}$ (viewed as an ideal of $\mathcal{O}_F$) is coprime to the conductor $c(\pi) \subset \mathcal{O}_F$ 
and the relative discriminant $\mathfrak{D} = \mathfrak{D}_{K/F} \subset \mathcal{O}_F$, 
and that $c(\pi)$ and $\mathfrak{D}$ are coprime. We have by the Rankin-Selberg 
theory (cf.~\cite[$\S 2$]{LRS}) the generic root number formula
\begin{align*} \epsilon(1/2, \pi \times \mathcal{W}) &= \omega(c(\mathcal{W})) 
\cdot \mathcal{W} \vert_{ {\bf{A}}_F^{\times}}(c(\pi)) \cdot \epsilon(1/2, \pi) 
\cdot \epsilon(1/2, \mathcal{W})^4 \in {\bf{S}}^1, \end{align*} 
which in our setting is given more explicitly by 
\begin{align}\label{rn1} \epsilon(1/2, \pi \times \rho \chi \circ {\bf{N}}) 
&= \omega( \operatorname{lcm}(\mathfrak{p}^{\alpha}, p^{\beta} \mathcal{O}_F)) 
\cdot \eta \chi^2 \circ {\bf{N}}(c(\pi)) \cdot \epsilon(1/2, \pi) \cdot \epsilon(1/2, \rho \chi \circ {\bf{N}})^4. \end{align}
Here, $\epsilon(1/2, \pi)$ denotes the root number of the $L$-function $L(s, \pi)$ of $\pi$, which appears 
in the functional equation $\Lambda(s, \pi) = \epsilon(1/2, \pi) \Lambda(1-s, \widetilde{\pi})$ 
of the corresponding completed $L$-function $\Lambda(s, \pi)$ of $L(s, \pi)$. 
As well, $\epsilon(1/2, \mathcal{W}) = \epsilon(1/2, \rho \chi \circ {\bf{N}})$ 
denotes the root number of the Hecke $L$-function $L(s, \mathcal{W}) = L(s, \rho \chi \circ {\bf{N}})$. 
Now, it can be deduced from the classically-known properties of the corresponding theta series 
$\theta(\mathcal{W}) = \theta(\rho \chi \circ {\bf{N}})$ 
(a Hilbert modular form of parallel weight one, 
level $\mathfrak{D} \cdot \operatorname{lcm}(\mathfrak{p}^{\alpha}, p^{\beta} \mathcal{O}_F) \subset \mathcal{O}_F$, 
and character $\eta \chi^2 \circ {\bf{N}}$) that this latter root number is given by 
\begin{align}\label{rn0} \epsilon(1/2, \rho \chi \circ {\bf{N}}) &= \eta \chi^2 \circ {\bf{N}} (\mathfrak{d}) \cdot 
\frac{\tau(\eta \chi^2 \circ {\bf{N}})}{ {\bf{N}}(\mathfrak{D} p^{\beta} \mathcal{O}_F)^{\frac{1}{2}} } 
= \eta( \mathfrak{d} ) \cdot \chi^2( {\bf{N}} \mathfrak{d} ) 
\cdot \frac{ \tau ( \eta \chi^2 \circ {\bf{N}} ) }{ {\bf{N}}(\mathfrak{D} p^{\beta} \mathcal{O}_F)^{\frac{1}{2}} }, \end{align}
where $\mathfrak{d} = \mathfrak{d}_F \subset \mathcal{O}_F$ denotes the different of $F$. 
Here, the Gauss sum $\tau(\eta \chi^2 \circ {\bf{N}})$ is defined in this generality as follows (see e.g.~\cite[(65)]{Ro3}). 
Let $e$ denote the function $e(x)= \exp(2 \pi i x)$. Given a primitive Hecke character $\xi$ of $F$ of some conductor 
$c(\xi) = \mathfrak{q} \subset \mathcal{O}_F$, the Gauss sum $\tau(\xi)$ of $\xi$ is defined by 
\begin{align*} \tau(\xi) &= \sum_{x \bmod {\bf{N}} \mathfrak{q} } \xi(x \mathcal{O}_F) e \left( \frac{x^*}{ {\bf{N}} \mathfrak{q} }\right),\end{align*} 
where $x$ runs over invertible classes modulo ${\bf{N}} \mathfrak{q}$, and $x^*$ is determined uniquely as follows: If
\begin{align*} x \equiv \prod_{v \mid {\bf{N}} \mathfrak{q}} x_v \bmod {\bf{N}} \mathfrak{q} \end{align*} 
with $x_v \equiv 1 \bmod {\bf{N}} \mathfrak{q} /v$ for each prime divisor $v$ of ${\bf{N}} \mathfrak{q}$, then 
$x^* \equiv \sum_{v \mid {\bf{N}} \mathfrak{q}} x_v {\bf{N}} \mathfrak{q} / v \bmod {\bf{N}} \mathfrak{q}$.
More generally, the root number $\epsilon(1/2, \xi)$ of the corresponding Hecke $L$-function $L(s, \xi)$ is then given by the formula 
\begin{align*} \epsilon(1/2, \xi) &= {\bf{N}} \mathfrak{q}^{-\frac{1}{2}} \xi(\mathfrak{d}) \tau(\xi) =
{\bf{N}} \mathfrak{q}^{-\frac{1}{2}} \xi(\mathfrak{d}) \sum_{x \bmod {\bf{N}} \mathfrak{q}} \xi(x \mathcal{O}_F) 
e \left(\frac{x^*}{  {\bf{N}}\mathfrak{q} }\right). \end{align*} 
Now, we can give the following more explicit description of the root numbers $(\ref{rn1})$ we consider in this work. 

\begin{proposition}\label{RNF}

Let $\pi$ be a cuspidal automorphic representation of $\operatorname{GL}_2({\bf{A}}_F)$ of level $c(\pi) \subset \mathcal{O}_F$, 
central character $\omega = \omega_{\pi}$, and root number $\epsilon(1/2, \pi)$. 
Let $K$ be a totally imaginary quadratic extension of $F$ of relative discriminant 
$\mathfrak{D}  = \mathfrak{D}_{K/F} \subset \mathcal{O}_F$ and associated idele class character $\eta = \eta_{K/F}$ of $F$. 
Fix a prime ideal $\mathfrak{p} \subset \mathcal{O}_F$ with underlying rational prime $p$.
Assume that $(c(\pi), \mathfrak{D}\mathfrak{p}) = (\mathfrak{p}, \mathfrak{D}) = 1$. 
Let $\mathcal{W} = \rho \chi \circ {\bf{N}}$ be a wide ray class 
Hecke character of $K$ as described above, with $\rho$ a primitive ring class character of conductor $\mathfrak{p}^{\alpha}$
for some integer $\alpha \geq 0$, and $\chi$ a primitive even Dirichlet character of conductor $p^{\beta}$ for some integer $\beta \geq 0$.
Then, the root number $\epsilon(1/2, \pi \times \rho \chi \circ {\bf{N}})$ of the corresponding Rankin-Selberg $L$-function 
$L(s, \pi \times \rho \times \chi \circ {\bf{N}}) = L(s, \pi \times \pi(\rho \chi \circ {\bf{N}}))$ is given by 
\begin{align*} \epsilon(1/2, \pi \times \rho \chi \circ {\bf{N}}) 
&= \omega(\operatorname{lcm}(\mathfrak{p}^{\alpha}, p^{\beta} \mathcal{O}_F)) \cdot \eta(p^{4 \beta} \mathfrak{d}c(\pi)) \cdot 
\epsilon(1/2, \pi) \cdot \chi( {\bf{N}}(\mathfrak{d}^2 c(\pi)^2 \mathfrak{D}^8)) 
\cdot \left( \frac{\tau(\chi^2)}{p^{\frac{\beta}{2}}} \right)^{4d}, \end{align*}
where 
\begin{align*} \tau(\chi^2) &= \sum_{ x \bmod p^{\beta}} \chi^2(x) e \left( \frac{x}{p^{\beta}} \right) \end{align*}
denotes the standard Gauss sum defined over coprime residue classes $x \bmod p^{\beta}$.
Observe in particular that this formula does not depend on the particular choice of ring character 
$\rho$ (i.e.~does not depend on its values), and for $\alpha \gg 1$ sufficiently large is
in fact completely independent of the choice of the ring class character $\rho$ 
(i.e.~as $\omega(\operatorname{lcm}(\mathfrak{p}^{\alpha}, p^{\beta} \mathcal{O}_F))  = 1$ 
for $\alpha \gg 1$, since the homomorphism $\omega$ takes values in roots of unity). 

\end{proposition}

\begin{proof}

Let us first consider the Gauss sum in the formula $(\ref{rn0})$. Using the twisted multiplicativity relation (see e.g.~\cite[(3.16)]{IK}),
we can decompose this as 
\begin{align}\label{tm} \tau(\eta \chi^2 \circ {\bf{N}}) 
&= \eta(p^{\beta} \mathcal{O}_F) \cdot \chi^2( {\bf{N}}\mathfrak{D}) \cdot \tau(\eta) \cdot \tau(\chi^2 \circ {\bf{N}}). \end{align}
On the other hand, we can unravel definitions to find the simplification of the Gauss sum
\begin{align*} \tau(\chi^2 \circ {\bf{N}}) 
&= \sum_{ x \bmod {\bf{N}} (p^{\beta} \mathcal{O}_F) } \chi^2 ( {\bf{N}}(x \mathcal{O}_F)) e \left( \frac{x^*}{ {\bf{N}} (p^{\beta} \mathcal{O}_F) } \right)
= \sum_{x \bmod p^{\beta}} \chi^2( {\bf{N}}(x \mathcal{O}_F)) \prod_{v \mid p^{d \beta}} e \left( \frac{ x_vp^{d\beta}/v }{p^{d \beta}} \right) \\
&= \sum_{x \bmod p^{\beta}} \prod_{\iota: F \hookrightarrow {\bf{C}} } \chi^2(x) e \left(\frac{x}{p^{\beta}} \right) = \tau(\chi^2)^d. \end{align*}
Note that this simplification can also be deduced more directly via the classical definition of the Gauss sum (see e.g.~\cite[(6.3)]{Ne}), 
which recall for any choice of representative $y \in p^{\beta} \mathfrak{d}^{-1} \subset \mathcal{O}_F$ takes the form 
\begin{align*} \sum_{x \bmod p^{\beta}} \chi^2({\bf{N}}(x \mathcal{O}_F)) e(\operatorname{Tr}(xy)) 
= \sum_{x \bmod p^{\beta}} \prod_{\iota: F \hookrightarrow {\bf{C}}} \chi^2(\iota(x \mathcal{O}_F)) e(\iota(x y)) = 
\chi^2({\bf{N}} \mathfrak{d}) \prod_{\iota: F \hookrightarrow {\bf{C}}} \sum_{x \bmod p^{\beta}} \chi^2(x)e \left( \frac{x}{p^{\beta}} \right). \end{align*}
Using this simplification $\tau(\chi^2 \circ {\bf{N}}) = \tau(\chi^2)^d$ in $(\ref{tm})$ then gives us the relation
\begin{align*} \tau(\eta \chi^2 \circ {\bf{N}}) 
&= \eta(p^{\beta} \mathcal{O}_F) \cdot \chi^2( {\bf{N}} \mathfrak{D}) \cdot \tau(\eta) \cdot \tau(\chi^2)^{d}, \end{align*}
from which it follows that 
\begin{align*}
\left( \frac{ \tau (\eta \chi^2 \circ {\bf{N}} ) }{ {\bf{N}} ( \mathfrak{D} p^{\beta} \mathcal{O}_F)^{\frac{1}{2}} } \right)^4 
&= \eta(p^{4 \beta} \mathcal{O}_F) \cdot \chi({\bf{N}} \mathfrak{D}^8) \cdot 
\left( \frac{\tau(\eta)}{ {\bf{N}} \mathfrak{D}^{\frac{1}{2}} } \right)^4 \cdot \left( \frac{\tau(\chi^2)}{p^{\frac{\beta}{2}}} \right)^{4d}. \end{align*}
Hence, we see that $(\ref{rn0})$ takes the more explicit form 
\begin{align*}\epsilon(1/2, \rho \chi \circ {\bf{N}}) &= \eta(p^{4 \beta} \mathfrak{d}) \cdot 
\chi( {\bf{N}} ( \mathfrak{d}^2 \mathfrak{D}^8) ) \cdot \left( \frac{\tau(\eta)}{ {\bf{N}} \mathfrak{D}^{\frac{1}{2}} } \right)^4 
\cdot \left( \frac{\tau(\chi^2)}{p^{\frac{\beta}{2}}} \right)^{4d}. \end{align*}
Since the quadratic root number $\epsilon(1/2, \eta)$ satisfies the relation $\epsilon(1/2, \eta)^4 = 1$, we deduce that 
\begin{align*} \epsilon(1/2, \eta)^4 & = \eta(\mathfrak{d}^4) \left( \frac{\tau(\omega)}{{\bf{N}}\mathfrak{D}^{\frac{1}{2}} } \right)^4 
=  \left( \frac{\tau(\omega)}{{\bf{N}}\mathfrak{D}^{\frac{1}{2}} } \right)^4 =1, \end{align*}
and so we obtain the even simpler explicit formula
\begin{align*}\epsilon(1/2, \rho \chi \circ {\bf{N}}) &= \eta(p^{4 \beta} \mathfrak{d}) \cdot 
\chi( {\bf{N}} ( \mathfrak{d}^2 \mathfrak{D}^8) ) \cdot \left( \frac{\tau(\chi^2)}{p^{\frac{\beta}{2}}} \right)^{4d}. \end{align*}
We then substitute this expression into $(\ref{rn1})$ to derive the stated formula for the root number. \end{proof}

\subsubsection{Dirichlet series expansions} 

Let us retain the setup described above, fixing a prime ideal $\mathfrak{p} \subset \mathcal{O}_F$ with underlying rational prime $p$,
and taking $\mathcal{W} = \rho \chi \circ {\bf{N}}$ to be a wide ray class character of $K$ given by the product of a ring class character 
$\rho$ of $K$ of conductor $\mathfrak{p}^{\alpha}$ for some integer $\alpha \geq 0$ 
times a primitive even Dirichlet character $\chi \bmod p^{\beta}$ for some integer $\beta \geq 0$
composed with the norm homomorphism ${\bf{N}}$ on ideals of $K$. 
Let us also write the Dirichlet series expansion of the finite part $L(s, \pi)$
of the $L$-function $\Lambda(s, \pi) = L(s, \pi_{\infty}) L(s, \pi)$ of $\pi$ (first for $\Re(s) > 1$) as 
\begin{align*} L(s, \pi) &= \sum_{\mathfrak{n} \neq \lbrace 0 \rbrace \subset \mathcal{O}_F} \frac{\lambda(\mathfrak{n})}{ {\bf{N}}\mathfrak{n}^s } 
= \sum_{\mathfrak{n} \neq \lbrace 0 \rbrace \subset \mathcal{O}_F} \frac{ \lambda_{\pi}(\mathfrak{n})}{ {\bf{N}}\mathfrak{n}^s }. \end{align*}
We can then write the Dirichlet series expansion of $L(s, \pi \times \mathcal{W})$ (first for $\Re(s) >1$) 
as a $\operatorname{GL}_2({\bf{A}}_F) \times \operatorname{GL}_2({\bf{A}}_F)$ Rankin-Selberg $L$-function over ideals of $\mathcal{O}_F$  by 
\begin{align}\label{DSEF} L(s, \pi \times \rho \chi \circ {\bf{N}}) 
&= \sum_{\mathfrak{m} \neq \lbrace 0 \rbrace \subset \mathcal{O}_F } 
\frac{ \omega \eta (\mathfrak{m}) \chi^2({\bf{N}} \mathfrak{m})}{ {\bf{N}} \mathfrak{m}^{2s}} 
\sum_{\mathfrak{n} \neq \lbrace 0 \rbrace \subset \mathcal{O}_F } 
\frac{ \lambda(\mathfrak{n}) \chi( {\bf{N}} \mathfrak{n})}{ {\bf{N}} \mathfrak{n}^s }  
\left( \sum_{A \in C (\alpha)} r_A(\mathfrak{n}) \rho(A)\right). \end{align} 
Here, we write $C(\alpha)$ to denote the class group of the order 
$\mathcal{O}_{\mathfrak{p}^{\alpha}} := \mathcal{O}_F + \mathfrak{p}^{\alpha} \mathcal{O}_K$, 
and $r_A(\mathfrak{n})$ the number of ideals in the class $A$ whose image under the relative norm homomorphism ${\bf{N}}_{K/F}$ 
equals a given $\mathfrak{n} \subset \mathcal{O}_F$. As well, each of the sums runs over ideals of $\mathcal{O}_F$ which are coprime
to the conductor of the characters which appear, although we omit this natural condition from the notations for simplicity. 

\subsubsection{Partial symmetric square $L$-values and related Dirichlet series} 

Let us also introduce the following Dirichlet series which we encounter in our later calculations, and which is related up to a convergent
product of Euler factors to the following partial symmetric square $L$-function of $\pi$ at $s=1$. We refer to \cite[$\S 1.1$]{CM} and more 
generally \cite{Sha} for some background discussion of the analytic properties of such $L$-functions. Fix $\pi = \otimes_v \pi_v$ a cuspidal 
automorphic representation of $\operatorname{GL}_2({\bf{A}}_F)$ as above, writing $c(\pi) \subset \mathcal{O}_F$ again to denote its conductor. 
Let $\xi = \otimes_v \xi_v = \chi \circ {\bf{N}}$ be a wide ray class Hecke character of $F$ induced via composition with the norm homomorphism 
from a primitive even Dirichlet character $\chi$ of conductor $c(\chi)$. Let $S$ denote the set of places of $F$ where both $\pi$ and $\xi$ are ramified. 
We consider the incomplete $L$-function of the symmetric square of $\pi$ times $\xi$, defined for $s \in {\bf{C}}$ (first with $\Re(s)>1$) by the Euler product 
\begin{align*} L^S(s, \operatorname{Sym}^2 \pi \otimes \xi) &= \prod_{v \notin S} L(s, \operatorname{Sym}^2 \pi_v \otimes \xi_v), \end{align*}
where the local Euler factors $L(s, \operatorname{Sym}^2 \pi_v \otimes \xi_v)$ are defined as follows. If $v$ is a place where $\pi_v$
is unramified, then there exist unramified quasicharacters $\mu_{1,v}$ and $\mu_{2,v}$ of $\operatorname{GL}_2(F_v)$ such that
$\pi_v$ arises from the induced representation of the character $\mu_v = \mu_{1,v} \otimes \mu_{2,v}$ of the torus
$T_v \subset \operatorname{GL}_2(F_v)$ of diagonal matrices. Fixing a uniformizer $\varpi_v$ of $\mathcal{O}_{F_v}$, 
and writing $A_v$ to denote the diagonal matrix 
$\operatorname{diag}(\mu_{1,v}, \mu_{2,v})$, we then have 
\begin{align*} L(s, \operatorname{Sym}^2 \pi_v \otimes \xi) &= \det \left( I - \operatorname{Sym}^2(A_v) \xi_v(\varpi_v) {\bf{N}} v^{-s} \right)^{-1} 
= \prod_{1 \leq i \leq j \leq 2} \left( 1 - \mu_{i, v} \mu_{j, v} \xi_v(\varpi_v) {\bf{N}} v^{-s} \right)^{-1}. \end{align*} 
It is well-known that $L^S(s, \operatorname{Sym}^2 \pi \otimes \xi)$ does not vanish at $\Re(s) = 1$ (see \cite[Theorem 1.1]{Sha}). 
The work of Goldfeld-Hoffstein-Lieman \cite{GHL} on exceptional zeros in fact gives a lower bound for such $L$-values in the classical setting. 
In general, as explained in \cite[$\S 1.1$]{CM}, one can derive individual upper and lower bounds for $L(s, \operatorname{Sym}^2 \pi \otimes \xi)$ 
via the automorphy of $\operatorname{Sym}^2 \pi$ (known thanks to Gelbart-Jacquet \cite{GJ}). 
In particular, the automorphy can be used to deduce individual upper bounds 
$L(1, \operatorname{Sym}^2 \pi \otimes \xi) \ll_{\varepsilon} {\bf{N}} (c(\pi) c(\chi)^2)^{\varepsilon}$, as well as individual lower bounds  
\begin{align}\label{lower} L^S(1, \operatorname{Sym}^2 \pi \otimes \xi) \gg \left( \log ( {\bf{N}} (c(\pi)c(\xi)^2)) \right)^{-C} \end{align}
for some constant $C>0$. Note that the local Euler factors $L(s, \operatorname{Sym}^2 \pi_v \otimes \xi_v)$ at primes $v \mid c(\pi)$ 
can be defined in a more complicated way. We shall omit the superscript $S$ in the discussion below. 

Finally, we note that $L^S(s, \operatorname{Sym}^2 \pi \otimes \xi)$ has Dirichlet series expansion (first for $\Re(s) >1$) 
given up to some convergent finite product of Euler factors by the simplified Dirichlet series defined by   
\begin{align*} L^{S, \star}(s, \operatorname{Sym}^2 \pi \otimes \xi) &:= L^S(2s, \omega \xi^2 )
\sum_{\mathfrak{n} \neq \lbrace 0 \rbrace \subset \mathcal{O}_F \atop (\mathfrak{n}, S)=1} 
\frac{\lambda_{\pi \otimes \xi}(\mathfrak{n}^2)}{{\bf{N}} \mathfrak{n}^s}
= \sum_{\mathfrak{m} \subset \mathcal{O}_F \atop (\mathfrak{m}, S)=1} \frac{\omega \xi^2 (\mathfrak{m})}{{\bf{N}} \mathfrak{m}^{2s}}
\sum_{\mathfrak{n} \neq \lbrace 0 \rbrace \subset \mathcal{O}_F \atop (\mathfrak{n}, S)=1} 
\frac{\lambda(\mathfrak{n}^2) \xi(\mathfrak{n})}{{\bf{N}} \mathfrak{n}^s}. \end{align*} 
We shall often encounter the following corresponding partial Dirichlet series expansion over principal ideals 
$\mathfrak{n} = (a) \subset \mathcal{O}_F$, which we denote throughout (dropping the superscript $S$ for simplicity) by 
\begin{align*} L_{\bf{1}}^{S, \star}(s, \operatorname{Sym}^2 \pi \otimes \xi) &:= L^S_{\bf{1}}(2s, \omega \xi^2 ) 
\sum_{a \neq 0 \in \mathcal{O}_F / \mathcal{O}_F^{\times} \atop (a, S)=1} \frac{ \lambda_{\pi \otimes \xi}(a^2)}{{\bf{N}} a^s}
= \sum_{ b \neq 0 \in \mathcal{O}_F / \mathcal{O}_F^{\times} \atop (b, S)=1} \frac{\omega \xi^2 (b)}{ {\bf{N}} b^{2s}} 
\sum_{a \neq 0 \in \mathcal{O}_F / \mathcal{O}_F^{\times} \atop (a, S)=1} \frac{ \lambda(a^2)\xi(a)}{ {\bf{N}} a^s}. \end{align*}
Here, we write $b = (b)$ and $a = (a)$ to lighten notations, and the sums are really taken over nonzero principal ideals 
$(a), (b) \neq 0 \subset \mathcal{O}_F$. That is, we write the sums here as the corresponding sums over $F$-integers $a, b$
modulo the action of units as $a, b \in \mathcal{O}_F / \mathcal{O}_F^{\times}$. 
We ask the reader to please excuse this shorthand notation, as it simplifies notations greatly for our later calculations. 
Also, although it does not generally admit an Euler product (unless $F$ has class number one), 
this partial Dirichlet series has the same basic analytic properties 
as the sum over classes $L^S(s, \operatorname{Sym}^2 \pi \otimes \xi)$, from which we deduce that it does not vanish at $s=1$.

\subsection{Approximate functional equations}  

Let $\mathcal{W} = \rho \chi \circ {\bf{N}}$ be a wide ray class character of $K$, as above. 
Recall that for $\Re(s) > 1$, we write the Dirichlet series expansion $(\ref{DSEF})$ of the 
finite part $L(s, \pi \times \mathcal{W})$ of the Rankin-Selberg $L$-function 
$\Lambda(s, \pi \times \mathcal{W}) = L_{\infty}(s) L(s, \pi \times \mathcal{W})$ as  
\begin{align*} L(s, \pi \times \mathcal{W}) &= \sum_{\mathfrak{m} \neq \lbrace 0 \rbrace \subset \mathcal{O}_F} 
\frac{ \omega \eta (\mathfrak{m}) \chi^2( {\bf{N}} \mathfrak{m})}{ {\bf{N}} \mathfrak{m}^{2s}} 
\sum_{\mathfrak{n} \neq \lbrace 0 \rbrace \subset \mathcal{O}_F} 
\frac{ \lambda(\mathfrak{n}) \chi( {\bf{N}} \mathfrak{n}) }{ {\bf{N}} \mathfrak{n}^s } \left( \sum_{A \in C (\alpha)} r_A(\mathfrak{n}) \rho(A)\right). \end{align*} 
Recall as well that the completed $L$-function $\Lambda(s, \pi \times \mathcal{W})$ satisfies the functional equation 
\begin{align*} \Lambda(s, \pi \times \mathcal{W}) 
&= \epsilon(1/2, \pi \times \mathcal{W}) \cdot c(\pi \times \mathcal{W})^{\frac{1}{2} - s} \cdot \Lambda(1-s, \widetilde{\pi} \times \overline{\mathcal{W}}), \end{align*}
where $\epsilon(1/2, \pi \times \mathcal{W}) \in {\bf{S}}^1$ as described in $(\ref{RNF})$ above denotes the root number, 
and $c(\pi \times \mathcal{W})$ as described in $(\ref{conductor})$ above denotes the conductor of the $L$-function.

We now derive a suitable presentation for the finite part $L(s, \pi \times \mathcal{W})$ of this $L$-function at $s=1/2$, 
outside of the range of absolute convergence $\Re(s) >1$ via the following standard contour argument. 
We present the details for the convenience of the reader. Let us fix a holomorphic test function $k$ on ${\bf{C}}$ such that $k(s)$ 
is even and bounded in vertical strips. To be more precise, let $g \in \mathcal{C}_c^{\infty}({\bf{R}}_{>0})$ be any smooth and compactly 
supported test function, and let $k(s) = \int_0^{\infty} g(x) x^s \frac{dx}{x}$ denote its Mellin transform. 
We can and do assume that $g$ is chosen so that $k(0)=1$. Now, recall that the generalized Ramanujan conjecture for $\pi$ at the 
real places of $F$ predicts that $\max_j (\mu_{\infty}(j)) = 0$. If $\pi$ arises from a holomorphic Hilbert modular form, then this 
conjecture is known by work of Blasius \cite{Bl} (generalizing Deligne's theorem \cite{De}). 
In general, the conjecture is not yet known, and in this level of generality we can and do assume that $g$ is chosen so that $k(\mu_{\infty}(j)) = 0$ 
for each $1 \leq j \leq 2$. This allows us to avoid having to consider addition residual coming from poles inside the critical strip $0 < \Re(s) < 1$ in 
our subsequent arguments. Fixing such a test function $k(s)$ once and for all, 
we then define the following smooth cutoff functions $V_1(y)$ and $V_2(y)$ on $y \in {\bf{R}}_{>0}$ by 
\begin{align*} V_1(y) &= \int_{\Re(s) = 2} \frac{k(s)}{s} y^{-s} \frac{ds}{2 \pi i} \\
V_2(y) &= \int_{\Re(s) = 2} \frac{k(-s)}{s} \frac{\widetilde{L}_{\infty}(s + 1/2)}{L_{\infty}(-s + 1/2)} y^{-s} \frac{ds}{2 \pi i}. \end{align*} 
Here, we write $\widetilde{L}_{\infty}(s)$ to denote the archimedean component of the contragredient $L$-function 
\begin{align*} \Lambda(s, \widetilde{\pi} \times \overline{\mathcal{W}}) 
= \widetilde{L}_{\infty}(s) L(s, \widetilde{\pi} \times \overline{\mathcal{W}}), \end{align*}
which again does not depend on the choice of wide ray class character $\overline{\mathcal{W}}$ of $K$.

\begin{lemma}\label{AFE} 

Let $\mathcal{W} = \rho \chi \circ {\bf{N}}$ be any wide ray class Hecke character of $K$ of the form described above, 
with $\rho$ a ring class character of $K$ conductor $\mathfrak{p}^{\alpha} \subset \mathcal{O}_F$, 
and $\chi \bmod p^{\beta}$ a primitive even Dirichlet character, 
and assume that $(c(\pi), \mathfrak{D}\mathfrak{p}) = (\mathfrak{p}, \mathfrak{D}) = 1$. 
Let $c(\pi \times \rho \chi \circ {\bf{N}}) = {\bf{N}} (\mathfrak{D}^2 c(\pi_K) p^{4d \max (\alpha, \beta)})$ denote 
the conductor of the $L$-function $L(s, \pi \times \mathcal{W}) = L(s, \pi \times \rho \chi \circ {\bf{N}})$. 
We have for any choice of real parameter $Z >0$ the formula
\begin{align*} &L(1/2, \pi \times \mathcal{W}) = \sum_{\mathfrak{m}  \subset \mathcal{O}_F \atop \mathfrak{m} \neq \lbrace 0 \rbrace} 
\frac{\omega \eta(\mathfrak{m}) \chi^2( {\bf{N}} \mathfrak{m}) }{ {\bf{N}} \mathfrak{m}} 
\sum_{\mathfrak{n} \subset \mathcal{O}_F \atop \mathfrak{n} \neq \lbrace 0 \rbrace} 
\frac{\lambda(\mathfrak{n}) \chi( {\bf{N}} \mathfrak{n}) }{ {\bf{N}} \mathfrak{n}^{\frac{1}{2}}} 
\left( \sum_{A \in C (\alpha)} r_A(\mathfrak{n}) \rho(A)\right)  V_1 \left( {\bf{N}} \mathfrak{m}^2 {\bf{N}} \mathfrak{n} Z \right) \\
&+ \epsilon(1/2, \pi \times \rho \chi \circ {\bf{N}}) 
\sum_{\mathfrak{m} \subset \mathcal{O}_F \atop \mathfrak{m} \neq \lbrace 0 \rbrace} 
\frac{ \overline{\omega} \eta(\mathfrak{m}) \overline{\chi}^2( {\bf{N}} \mathfrak{m}) }{ {\bf{N}}\mathfrak{m}} 
\sum_{\mathfrak{n} \subset \mathcal{O}_F \atop \mathfrak{n} \neq \lbrace 0 \rbrace} 
\frac{ \overline{\lambda(\mathfrak{n}) \chi({\bf{N}} \mathfrak{n})}}{ {\bf{N}} \mathfrak{n}^{\frac{1}{2}}} 
\left( \sum_{A \in C (\alpha)} r_A(\mathfrak{n}) \rho(A)\right)  
V_2 \left( \frac{ {\bf{N}} \mathfrak{m}^2 {\bf{N}} \mathfrak{n} }{ Z c(\pi \times \rho \chi \circ {\bf{N}} )} \right). \end{align*} \end{lemma}

\begin{proof}  The proof is standard (see e.g.~\cite[Lemma 3.2]{LRS} or \cite[$\S 5.2$]{IK}). Note that the ring class character
$\rho$ in his expression is not inverted as a consequence of the fact that such characters are equivariant with 
respect to complex conjugation (cf.~\cite[$\S 1$]{Ro2}), as mentioned already in our discussion of the functional equation $(\ref{symmetricFE})$. \end{proof}

\begin{lemma}\label{cutoff} 

The smooth cutoff functions $V_j(y)$ defined above satisfy the following decay properties: 
\begin{align*} V_1(y) &= \begin{cases} 1 + O_A(y^A) &\text{for any choice of $A \geq 1$ if $0 < y \leq 1$} \\ 
O_{C}(y^{-C}) &\text{for any constant $C > 0$ if $y \geq 1$}. \end{cases} \end{align*} 
and 
\begin{align*} V_2(y) &= \begin{cases} \frac{\widetilde{L}_{\infty}(1/2)}{L_{\infty}(1/2)} 
+ O_{\varepsilon}(y^{\frac{1}{2} - \varepsilon}) &\text{for any small $\varepsilon > 0$ if $0 < y \leq 1$} \\ 
O_{C}(y^{-C}) &\text{for any constant $C > 0$ if $y \geq 1$}. \end{cases} \end{align*} 
\end{lemma}

\begin{proof} This is also standard; see \cite[Lemma 3.1]{LRS} or \cite[Proposition 5.4]{IK}). One shifts the contour defining $V_j(y)$ 
to the right to obtain the behaviour as $y \rightarrow \infty$, and to the left to obtain the behaviour as $y \rightarrow 0$. \end{proof}

\subsection{Shifted convolution sum estimates}  

We have the following estimates for the shifted convolution problem for the $L$-function coefficients of $\operatorname{GL}_2({\bf{A}}_F)$-automorphic forms. 
Although the proof is well-known to experts (even if the exact statement we give does not appear in the literature), 
we provide one for the convenience of the reader in Appendix A below, 
especially as we shall develop several of the ideas of this proof with spectral expansions for our subsequent estimates. 

\begin{theorem}\label{SCS} 

Let $\pi = \otimes \pi_v$ be any non-dihedral cuspidal $\operatorname{GL}_2({\bf{A}}_F)$-automorphic representation. 
Let $W$ be any smooth function of compact support on ${\bf{R}}_{>0}$ 
(or any smooth function which decays rapidly near infinity and moderately near zero), 
whose derivatives satisfy the decay condition $W^{(i)} \ll 1$ for all integers $i \geq 0$. 
Let $q \neq 0 \in \mathcal{O}_{F}$ be any totally positive $F$-integer. Assume that $\pi$ is not dihedral, in other words
that $\pi$ does not arise via automorphic induction from a Hecke character of some quadratic extension of $F$. 
For any choice real number $Y >0$ and any choice of $\varepsilon >0$, we have the uniform upper bound
\begin{align*}\sum_{\gamma \in F^{\times}} \frac{\lambda( \gamma^2 + q )}{ {\bf{N}}(\gamma^2 + q)^{\frac{1}{2}}} 
W \left( \frac{ {\bf{N}} (\gamma^2 + q)}{Y} \right) \ll_{\pi, \varepsilon} Y^{\frac{1}{4}} \cdot {\bf{N}} q^{\delta_0 - \frac{1}{2}} 
\left( \frac{ {\bf{N}} q }{Y} \right)^{\frac{1}{2} - \frac{\theta_0}{2} - \varepsilon}. \end{align*}
Here, the implied constant depends on the choice of weight function $W$.
As well, $0 \leq \theta_0 \leq 1/2$ denotes the best known approximation towards the generalized Ramanujan conjecture for arbitrary 
$\operatorname{GL}_2({\bf{A}}_F)$-automorphic forms, so that $\theta_0 = 7/64$ is admissible by the theorem of Blomer-Brumley \cite{BBRP}. 
On the other hand, $0 \leq \delta_0 \leq 1/4$ denotes the best exponent
in the bound for Fourier coefficients of automorphic forms on the metaplectic cover of $\operatorname{GL}_2({\bf{A}}_F)$, 
which by theorems of Kohnen-Zagier (for $F = {\bf{Q}}$) and Baruch-Mao (for any $F$) is equivalent to the best known approximation 
towards the generalized Lindel\"of hypothesis for $\operatorname{GL}_2({\bf{A}}_F)$-automorphic forms 
in the level aspect. Hence (taking $\theta_0 = 7/64$), $\delta_0 = 103/512$ is admissible by Blomer-Harcos \cite[Corollary 1]{BH10}. \end{theorem}

We also have the following variation, whose proof we also give in Appendix A (see also \cite[Theorem 2]{BH10}):

\begin{theorem}\label{SCS2}

Let $\pi = \otimes_v \pi_{v}$ for be a non-dihedral cuspidal $\operatorname{GL}_2({\bf{A}}_F)$-automorphic representation,
with Hecke eigenvalues or equivalently $L$-function coefficients denoted by $\lambda_{\pi}$.
Fix $W$ a smooth function and compact support 
(or any smooth function of rapid decay near infinity and moderate decay near zero) 
whose derivatives satisfy the decay condition $W^{(i)} \ll 1$ for all integers $i \geq 0$.
Fix an $F$-rational quadratic polynomial $q(x) = r x^2 + s x + t \in \mathcal{O}_F[x]$, and let us assume that 
the discriminant $\Delta := s^2 - 4rt$ is totally negative $\Delta \ll 0$ (hence nonzero).
Fix $Y>0$ a positive real number. 
Then, with notations as in Theorem \ref{SCS}, we have for any $\varepsilon >0$ the uniform upper bound 
\begin{align*} \sum\limits_{\gamma \in F^{\times} } 
\frac{ \lambda_{\pi}(q(\gamma))}{{\bf{N}}(q(\gamma))^{\frac{1}{2}}}
W \left( \frac{ {\bf{N}}q(\gamma) }{Y} \right) \ll_{\pi, \varepsilon} Y^{\frac{1}{4}} \cdot
{\bf{N}}r \cdot {\bf{N}} \Delta^{\delta_0 - \frac{1}{2}}
\left( \frac{ {\bf{N}} \Delta  }{Y} \right)^{\frac{1}{2} - \frac{\theta_0}{2}  - \varepsilon}. \end{align*} \end{theorem}

We shall use these bounds as follows to estimate the averages of $L$-functions we wish to consider.

\subsection{Averages over primitive ring class characters} 

Fix a prime ideal $\mathfrak{p} \subset \mathcal{O}_F$ with underlying rational prime $p$, 
and assume that $(c(\pi), \mathfrak{D} \mathfrak{p}) = (\mathfrak{p}, \mathfrak{D}) = 1$.
Fix a primitive even Dirichlet character $\chi$ of conductor $p^{\beta}$ for some integer $\beta \geq 0$. 
Let us for any integer $\alpha \geq 0$ write $C(\alpha)$ to denote the class group of the 
order $\mathcal{O}_{\mathfrak{p}^{\alpha}} = \mathcal{O}_F + \mathfrak{p}^{\alpha} \mathcal{O}_K$, with $C(\alpha)^{\vee}$ its character group. 
Hence, $C(0)$ denotes the class group of $\mathcal{O}_K$, whose cardinality we denote by $h_K = \#C(0)$. 
Note that by Dedekind's formula (see e.g.~\cite[Theorem 7.24]{Cox}),  
\begin{align*} \#C(\alpha) &= \frac{h_K \cdot {\bf{N}} \mathfrak{p}^{\alpha}}{[\mathcal{O}_K^{\times} : \mathcal{O}_{\mathfrak{p}^{\alpha}}^{\times}]} 
\cdot \left( 1 - \frac{\eta(\mathfrak{p})}{ {\bf{N}}\mathfrak{p}  }\right), 
\quad \eta(\mathfrak{p}) = \eta_{K/F}(\mathfrak{p}) = \begin{cases} ~~1 &\text{if $\mathfrak{p}$ splits in $K$}\\
-1 &\text{if $\mathfrak{p}$ is inert in $K$} \\ ~~ 0 &\text{if $\mathfrak{p}$ ramifies in $K$}, \end{cases} \end{align*} and hence 
\begin{align*} \#C(\alpha) 
&= \begin{cases} \frac{h_K}{[\mathcal{O}_K^{\times}: \mathcal{O}_{\mathfrak{p}^{\alpha}}^{\times}]} \cdot (p^d-1) \cdot p^{d(\alpha-1)} &\text{if $\mathfrak{p}$ splits in $K$} \\
\frac{h_K }{[\mathcal{O}_K^{\times} : \mathcal{O}_{\mathfrak{p}^{\alpha}}^{\times}]} \cdot (p^d+1) \cdot p^{d (\alpha-1)} &\text{if $\mathfrak{p}$ is inert in $K$} \\
\frac{h_K}{[\mathcal{O}_K^{\times} : \mathcal{O}_{\mathfrak{p}^{\alpha}}^{\times}]} \cdot p^{d \alpha} &\text{if $\mathfrak{p}$ ramifies in $K$}. \end{cases} \end{align*}
Moreover, since $(\mathcal{O}_{\mathfrak{p}^{\alpha}}^{\times})_{\alpha \geq 0}$ forms a decreasing sequence of subgroups
of $\mathcal{O}_K^{\times}$ with $\cap_{\alpha \geq 0} \mathcal{O}_{\mathfrak{p}^{\alpha}}^{\times} = \mathcal{O}_F^{\times}$
and $\mathcal{O}_F^{\times}$ has finite index (two) in $\mathcal{O}_K^{\times}$, 
we deduce that $\mathcal{O}_{\mathfrak{p}^{\alpha}} = \mathcal{O}_F^{\times}$ for all $\alpha \gg 0$ sufficiently large (cf.~\cite[Lemma 2.1]{CV}). 
Hence for all sufficiently large $\alpha \gg 0$ in this sense, we have the simpler formulae 
\begin{align}\label{SF} \#C(\alpha) &= \frac{1}{2} \cdot \begin{cases} h_K \cdot (p^d-1) \cdot p^{d (\alpha-1)} &\text{if $\mathfrak{p}$ splits in $K$} \\
h_K \cdot (p^d +1 ) \cdot p^{d (\alpha-1) } &\text{if $\mathfrak{p}$ is inert in $K$} \\
h_K \cdot p^{d \alpha} &\text{if $\mathfrak{p}$ ramifies in $K$}. \end{cases} \end{align}
Recall that for $\alpha \geq 1$, the primitive ring class characters of conductor $\mathfrak{p}^{\alpha}$ 
are those characters of $C(\alpha)^{\vee}$ which do not factor through $C(\alpha-1)^{\vee}$.  
We shall consider the weighted averages over such characters $\rho$ of the corresponding central values $L(1/2, \pi \times \rho \chi \circ {\bf{N}})$,
as well as the following sub-averages. First, in the style of Cornut-Vatsal \cite{CV}, let us consider the profinite limit 
$C(\infty) = \varprojlim_{\alpha \geq 0} C(\alpha)$, writing $C_0 = C(\infty)_{\operatorname{tors}}$ to denote its finite torsion subgroup. 
Given an integer $\alpha \geq 1$ and a character $\rho_0$ of the torsion subgroup $C_0$, we consider the subset $P(\alpha, \rho_0)$ 
of primitive ring class characters of $C(\alpha)$ whose induced character on $C_0$ (determined by the image of $C_0$ on $C(\alpha)$) equals $\rho_0$.  
We shall also consider the weighted sub-averages over primitive ring class characters $\rho$ of exact order $p^x$ of $C(\alpha)$,
where $x = \operatorname{ord}_p(\#C(\alpha))$ is the largest admissible exponent for the ring class characters of conductor $\mathfrak{p}^{\alpha}$ (cf.~\cite[$\S$ 2]{Va}).
Note (see \cite[$\S$3.1]{IK}) that the ring class characters of exponent $p^x$ detect the $p^x$-th powers 
\begin{align*} C(\alpha)^{p^x} &= \left\lbrace A^{p^x}: A \in C(\alpha) \right\rbrace, \end{align*}
in that we have the orthogonality relation 
\begin{align}\label{OO} \sum_{\rho \in C(\alpha)^{\vee} \atop \rho^{p^x}  = {\bf{1}}} \rho(A)
&= \begin{cases} [C(\alpha): C(\alpha)^{p^x}] &\text{if $A \in C(\alpha)^{p^x}$} \\ 0 &\text{otherwise}. \end{cases} \end{align}

\begin{lemma}\label{MI} 

By M\"obius inversion (or simply inclusion-exclusion), we have the following orthogonality relations for 
sums over primitive ring class characters of a given conductor and primitive ring class characters of a given (maximal) $p$-power order.
Let us for any integer $\alpha \geq 1$ write $Z(\alpha) = \ker \left( C(\alpha)  \longrightarrow C(\alpha- 1) \right)$ to denote the kernel 
of the natural surjective morphism $j: C(\alpha) \longrightarrow C(\alpha-1)$, 
with $\#C^{\star}(\alpha) = \#C(\alpha) - \#C(\alpha-1)$ the number of primitive ring class characters of conductor $\mathfrak{p}^{\alpha}$.
Let us also for an integer $0 \leq y \leq \operatorname{ord}_p(\#C(\alpha))$ write $\#C(\alpha, y) = [C(\alpha): C(\alpha)^{p^y}]$ 
to denote the index of the classes in $C(\alpha)$ of a given exponent $p^y$, with 
$\#C^{\star}(\alpha, x) = [C(\alpha): C(\alpha)^{p^x}] - [C(\alpha): C(\alpha)^{p^{x-1}}] = \#C(\alpha, x) - \#C(\alpha, x-1)$ the difference. \\

\begin{itemize}

\item[(i)] The weighted sum over primitive ring class characters in $C(\alpha)^{\vee}$ is given by the orthogonality relation
\begin{align*} \frac{1}{\#C^{\star}(\alpha)}\sum_{\rho \in C(\alpha)^{\vee} \atop \operatorname{primitive}} \rho(A) 
&= \begin{cases} 1 &\text{ if $A = {\bf{1}} \in C(\alpha)$ is the principal class} \\ 
- \frac{\#C(\alpha-1)}{\#C^{\star}(\alpha)} &\text{ if $A\in Z(\alpha)$, $A \neq {\bf{1}} \in C(\alpha)$} \\ 
0 &\text{ otherwise}. \end{cases} \end{align*}
Here, we can identify the sum over classes $A \in Z(\alpha)$ with the principal class in $C(\alpha-1)$.
Moreover, for each sufficiently large $\alpha \gg 1$, we have that $\#C(\alpha-1)/\#C^{\star}(\alpha) = 1/({\bf{N}} \mathfrak{p}  - 1) =1/(p^d-1)$. \\

\item[(ii)] Fix a character $\rho_0$ of $C_0 = C(\infty)_{\operatorname{tors}}$, 
and let us for each integer $\alpha \geq 0$ write $C_0(\alpha)$ to denote the image of $C_0$ in $C(\alpha)$
(so that $C_0 \cong C_0(\alpha)$ for $\alpha$ sufficiently large), with $\overline{C}(\alpha) = C(\alpha)/C_0(\alpha)$. 
The weighed sum over primitive ring class characters in $P(\alpha, \rho_0)$ inducing $\rho_0$ on $C_0$ is given by 
\begin{align*} \frac{1}{\#P(\alpha, \rho_0)} \sum_{\rho \in P(\alpha, \rho_0)} \rho(A) &= 
\rho_0(A) \cdot \begin{cases} 1 & \text{ if $A \in C_0(\alpha)$} \\ - \frac{\#\overline{C}(\alpha-1)}{\# P(\alpha, \rho_0)}
&\text{ if $A \in C_0(\alpha-1) \backslash C_0(\alpha)$} \\ 0 & \text{ otherwise}. \end{cases} \end{align*}
Here, we view $\rho_0$ as a character on the image $C_0(\alpha)$ of $C_0$ in $C(\alpha)$. 
Also, the cardinality $\#P(\alpha, \rho_0)$ is in fact given by the difference 
$\# \overline{C}(\alpha) - \#\overline{C}(\alpha-1)$, and so we can express the formula accordingly. \\

\item[(iii)] The weighted sum over primitive ring class characters in $C(\alpha)^{\vee}$ of exact order 
$p^x$ for $x = \operatorname{ord}_p(\#C(\alpha))$ is given by the orthogonality relation 

\begin{align*} \frac{1}{\# C^{\star}(\alpha, x)} 
\sum\limits_{ {\rho \in C(\alpha)^{\vee} \atop \rho^{p^x}  = {\bf{1}}} \atop \rho^{p^y} \neq {\bf{1}} ~ \forall ~ 0 \leq y < x } \rho(A)
&= \begin{cases} 1 &\text{ if $A \in C(\alpha)^{p^x}$} \\ 
- \frac{\#C(\alpha, {x-1})}{\#C^{\star}(\alpha, x)} &\text{ if $A \in C(\alpha)^{p^{x-1}} \backslash C(\alpha)^{p^x}$} \\
0 &\text{ otherwise}. \end{cases} \end{align*}

\end{itemize}

\end{lemma}

\begin{proof} 

For (i), we first apply M\"obius inversion, then apply standard orthogonality relations
for sums over characters of $C(\alpha)$ and $C(\alpha-1)$ respectively to deduce that 
\begin{align*} \sum\limits_{\rho \in C(\alpha)^{\vee} \atop \operatorname{primitive}} \rho(A) &=
\sum_{\rho \in C(\alpha)^{\vee}} \rho(A) - \sum_{\rho' \in C(\alpha-1)^{\vee}} \rho'(j(A)) \\
&= \begin{cases} \#C(\alpha) &\text{ if $A = {\bf{1}} \in C(\alpha)$} \\ 0 &\text{ otherwise} \end{cases} 
- \begin{cases} \#C(\alpha-1) &\text{ if $j(A) = {\bf{1}} \in C(\alpha-1)$} \\ 0 &\text{ otherwise} \end{cases} \\
&= \begin{cases} \#C(\alpha) - \#C(\alpha-1) &\text{ if $A = {\bf{1}} \in C(\alpha)$} \\
- \#C(\alpha-1) &\text{ if $A \in Z(\alpha) = \ker(j: C(\alpha) \rightarrow C(\alpha-1))$ but $A \neq {\bf{1}} \in C(\alpha)$} \\ 0 &\text{ otherwise} \end{cases}. \end{align*} 
Notice that the classes $A \in C(\alpha)$ in the difference expressions are identified with their images $j(A) \in C(\alpha-1)$,
as the second orthogonality relation applies only in the subquotient $C(\alpha-1)$.
The stated formula (i) then follows after dividing out by the factor $\#C^{\star}(\alpha)$. 

For (ii), we first take for granted the argument of \cite[Lemma 2.8]{CV}, which shows that 
(1) the natural surjective map $C_0  \rightarrow C_0(\alpha)$ is an isomorphism for $\alpha$ sufficiently large,
(2) the surjective map $C(\alpha) \rightarrow \overline{C}(\alpha)$ induces an isomorphism from $Z(\alpha) = \ker(C(\alpha) \rightarrow C(\alpha-1))$
to $\ker(\overline{C}(\alpha) \rightarrow \overline{C}(\alpha-1))$, and (3) the kernel $X(\alpha) = \ker(C(\alpha) \rightarrow \overline{C}(\alpha-1))$
has the direct sum decomposition $X(\alpha) \cong C_0(\alpha) \oplus Z(\alpha)$. We then deduce as in \cite[Lemma 2.8]{CV} that the subset of primitive 
characters $P(\alpha, \rho_0)$ has the following more explicit description: There exists a character $\rho_0'$ on $C_0(\alpha)$ 
inducing $\rho_0$ on $C_0$ and the trivial character ${\bf{1}}$ on $Z(\alpha)$ such that 
\begin{align*} P(\alpha, \rho_0) &= \rho_0' \cdot \left( \overline{C}(\alpha)^{\vee} - \overline{C}(\alpha-1)^{\vee} \right). \end{align*}
We then deduce stated formula from M\"obius inversion, as in (i). That is, we find that 
\begin{align*} \frac{1}{\#P(\alpha, \rho_0)} \sum_{\rho \in P(\alpha, \rho_0)} \rho(A) 
&= \frac{\rho_0'(A)}{(\# \overline{C}(\alpha) - \# \overline{C}(\alpha-1))} 
\cdot \left( \sum_{\overline{\rho} \in \overline{C}(\alpha)^{\vee}} \overline{\rho}(A) 
- \sum_{\overline{\rho}' \in \overline{C}(\alpha-1)^{\vee}} \overline{\rho}'(A)  \right),\end{align*}
which reduces us to the same argument as given for (i).

For (iii), we apply M\"obius inversion in a distinct way to detect the characters $\rho$ of exact order $p^x$ as 
\begin{align*} \sum\limits_{ { \rho \in C(\alpha)^{\vee} \operatorname{primitive} \atop \rho^{p^x} = {\bf{1}} } \atop \rho^{p^y} \neq {\bf{1}} \forall 0 \leq y \leq x- 1 } \rho(A)
&= \sum\limits_{\rho \in C(\alpha)^{\vee} \atop \rho^{p^x} = {\bf{1}}} \rho(A) - \sum\limits_{\rho' \in C(\alpha)^{\vee} \atop  (\rho')^{p^{x-1}} = {\bf{1}}   } \rho'(A).\end{align*}
That is, we argue that the ring class characters of exact order $p^x$ will necessarily be primitive, 
as they cannot factor through $C(\alpha-1)$ by definition of the exponent $x$. 
We then evaluate the difference of sums via the orthogonality relation $(\ref{OO})$ as 
\begin{align*} &\sum\limits_{\rho \in C(\alpha)^{\vee} \atop \rho^{p^x} = {\bf{1}}} \rho(A) - \sum\limits_{\rho'' \in C(\alpha)^{\vee} \atop  (\rho'')^{p^{x-1}} = {\bf{1}}   } \rho''(A) \\
&= \begin{cases} [C(\alpha):C(\alpha)^{p^x}] &\text{ if $A \in C(\alpha)^{p^x}$} \\ 0 &\text{ otherwise} \end{cases}
- \begin{cases} [C(\alpha): C(\alpha)^{p^{x-1}}] &\text{ if $A \in C(\alpha)^{p^{x-1}}$ } \\ 0 &\text{ otherwise} \end{cases} \\ 
&= \begin{cases} [C(\alpha): C(\alpha)^{p^x}] - [C(\alpha):C(\alpha)^{p^{x-1}}] &\text{ if $A \in C(\alpha)^{p^x}$} \\
-[C(\alpha): C(\alpha)^{p^{x-1}}] &\text{ if $A \in C(\alpha)^{p^{x-1}} \backslash C(\alpha)^{p^x}$} \\ 0 &\text{ otherwise}, \end{cases} \end{align*} 
Dividing out by the scaling factor $\#C^{\star}(\alpha, x) = \#C(\alpha, x) - \#C(\alpha, x-1) = [C(\alpha): C(\alpha)^{p^x}] - [C(\alpha): C(\alpha)^{p^{x-1}}]$
then gives the stated relations. \end{proof}

We shall consider the corresponding averages over primitive ring class characters of conductor $\mathfrak{p}^{\alpha}$
\begin{align}\label{primitive} P_{\alpha}(\pi, \chi) 
&:=  \frac{1}{\#C^{\star}(\alpha)} \sum\limits_{\rho \in C(\alpha)^{\vee} \atop \operatorname{primitive}} L(1/2, \pi \times \rho \chi \circ {\bf{N}}), \end{align}
the subaverage over characters inducing a given character $\rho_0$ on the torsion subgroup $C_0 = C(\infty)_{\operatorname{tors}}$, 
\begin{align}\label{tame} P_{\alpha, \rho_0}(\pi, \chi) &= \frac{1}{\#P(\alpha, \rho_0)} \sum_{\rho \in P(\alpha, \rho_0)} L(1/2, \pi \times \rho \chi \circ {\bf{N}}), \end{align}
and also the subaverage over characters of exact order $p^x$ with $x = \operatorname{ord}_p(\#C(\alpha))$, 
\begin{align}\label{galois} G_{\alpha}(\pi, \chi; x) &= \frac{1}{\#C^{\star}(\alpha, x)} 
\sum\limits_{ {\rho \in C(\alpha)^{\vee} \atop \rho^{p^x}  = {\bf{1}}} \atop \rho^{p^y} \neq {\bf{1}} ~ \forall ~ 0 \leq y < x }
L(1/2, \pi \times \rho \chi \circ {\bf{N}}). \end{align}

\subsubsection{Classical descriptions of the averages} 

Given a class $A \in C(\alpha)$, let us fix a positive definite quadratic form $f_A(x, y) = a_A x^2 + b_A xy + c_A y^2$ corresponding to $A$, 
i.e.~under the bijection between the class group $\mathcal{C}(\alpha)$ of positive definite binary quadratic forms of discriminant 
$\operatorname{disc}(\mathcal{O}_{\mathfrak{p}^{\alpha}}) = \mathfrak{D} \mathfrak{p}^{2\alpha}$ and the ideal class group $C(\alpha)$ 
of $\mathcal{O}_{\mathfrak{p}^{\alpha}}$ given by the map sending the 
class of a quadratic form representative $f_A(x, y) = a_Ax^2 + b_A xy + c_A y^2$ in $\mathcal{C}(\alpha)$ 
to the class of the proper integral ideal $\mathfrak{a}_A$ in $C(\alpha)$ with the $\mathcal{O}_F$-basis  
$[a_A, (-b_A + \mathfrak{p}^{\alpha} \sqrt{\mathfrak{D}})/2]$ (cf.~\cite[Theorem 7.7]{Cox}). 
Here, we also write $\mathfrak{p}^{\alpha}$ and $\mathfrak{D}$ to denote fixed $F$-integer representatives 
of the respective ideals $\mathfrak{p}^{\alpha} \subset \mathcal{O}_F$ and $\mathfrak{D} \subset \mathcal{O}_F$, 
and simplify notation in this way henceforth. Let $r_A$ denote the corresponding counting function. 
Hence, for a nonzero integral ideal $\mathfrak{n} \subset \mathcal{O}_F$, $r_A(\mathfrak{n})$ can be defined as the number  
of ideals $\mathfrak{a}$ in the class $A \in C(\alpha)$ of relative norm ${\bf{N}}_{K/F}(\mathfrak{a}) = \mathfrak{a} \overline{\mathfrak{a}} = \mathfrak{n}$.
Writing $w_K$ to denote the number of automorphs of the quadratic form $f_A$, 
or equivalently the number of units in $\mathcal{O}_K^{\times}$ which do not factor through $\mathcal{O}_F^{\times}$ (which is finite), 
this function $r_A(\mathfrak{n})$ can also be parametrized in terms of any representative $f_A(x, y)$ of $\mathcal{C}(\alpha)$ by
\begin{align}\label{ran} r_A(\mathfrak{n}) &= \frac{1}{w_K} \cdot \# \left\lbrace a, b \in \mathcal{O}_F/\mathcal{O}_F^{\times} : f_A(a, b) = \mathfrak{n} \right\rbrace. \end{align}
Here, the set of all $a, b \in \mathcal{O}_F/\mathcal{O}_F^{\times}$ refers to the lattice of integers in $F^2$, up to the action of units $\mathcal{O}_F^{\times}$;
we shall use the same notation later for the corresponding sums.
Note that this parametrization $(\ref{ran})$ is not unique, i.e.~as the choice of representative $f_A(x, y)$ is not unique. 
We shall later often assume that $f_A(x, y) = a_A x^2 + b_A xy + c_A y^2$ is the unique reduced representative
for the class, so that the $F$-integers $a_A$ and $c_A$ are both totally positive, with ${\bf{N}} b_A \leq {\bf{N}}a_A \leq {\bf{N}}c_A$,
and $b_A$ is totally positive if either ${\bf{N}} b_A = {\bf{N}} a_A$ or $a_A = c_A$. 
We also know that $a_A = 1$ in the special case where $A = {\bf{1}} \in C(\alpha)$ is the principal class, 
from which it follows that the corresponding last coefficient $c_A$ for the 
reduced form has norm roughly equal to that of the discriminant $\mathfrak{D} \mathfrak{p}^{2 \alpha}$ 
of the order $\mathcal{O}_{\mathfrak{p}^{\alpha}} \subset \mathcal{O}_K$.

Recall that we write $c(\pi \times \mathcal{W}) = {\bf{N}}(\mathfrak{D}^2 c(\pi_K) c(\mathcal{W})^2)$ to denote the conductor of 
$L(s, \pi \times \mathcal{W})$, where $c(\mathcal{W}) \subset \mathcal{O}_K$ is the conductor of the Hecke character $\mathcal{W}$
(as an ideal of $\mathcal{O}_K$ rather than $\mathcal{O}_F$).
Taking $\mathcal{W} = \rho \chi \circ {\bf{N}}$ with $\rho$ a primitive ring class character of conductor $\mathfrak{p}^{\alpha}$ 
and $\chi$ a primitive even Dirichlet character of conductor $p^{\beta}$ as we do, we have the more explicit formula for the conductor
\begin{align*} c(\pi \times \rho  \chi \circ {\bf{N}}) 
&= {\bf{N}}(\mathfrak{D}^2 c(\pi_K) \left( \operatorname{lcm}(\mathfrak{p}^{\alpha}, p^{\beta} \mathcal{O}_F)\mathcal{O}_K \right)^2)
= {\bf{N}}(\mathfrak{D}^2 c(\pi_K)) \cdot p^{4d \max (\alpha, \beta)}. \end{align*}
Recall as well that we introduce an unbalancing parameter $Z >0$ in Lemma \ref{AFE}. 
Given any such choice of real parameter $Z >0$, then us define the corresponding sums 
\begin{align}\label{DA1} D_{A, 1}(\pi, \chi; Z) 
&= \sum_{\mathfrak{m} \neq \lbrace 0 \rbrace \subset \mathcal{O}_F } 
\frac{ \omega \eta(\mathfrak{m}) \chi^2( {\bf{N}} \mathfrak{m} ) }{ {\bf{N}} \mathfrak{m}} \sum_{\mathfrak{n} \neq \lbrace 0 \rbrace \subset \mathcal{O}_F } 
\frac{\lambda(\mathfrak{n}) \chi( {\bf{N}} \mathfrak{n}) r_A(\mathfrak{n}) }{ {\bf{N}} \mathfrak{n}^{\frac{1}{2}}} 
V_1 \left( Z {\bf{N}} ( \mathfrak{m}^2 \mathfrak{n} ) \right). \end{align}
and 
\begin{align}\label{DA2} D_{A, 2}( \pi, \chi; Z) &= \sum_{\mathfrak{m} \neq \lbrace 0 \rbrace \subset \mathcal{O}_F } 
\frac{ \overline{\omega} \eta(\mathfrak{m}) \overline{\chi}^2( {\bf{N}} \mathfrak{m} ) }{ {\bf{N}} \mathfrak{m}} 
\sum_{\mathfrak{n} \neq \lbrace 0 \rbrace \subset \mathcal{O}_F } 
\frac{ \overline{ \lambda(\mathfrak{n}) \chi( {\bf{N}} \mathfrak{n})} r_A(\mathfrak{n}) }{ {\bf{N}} \mathfrak{n}^{\frac{1}{2}}} 
V_2 \left( \frac{ {\bf{N}} ( \mathfrak{m}^2 \mathfrak{n} )}{ Z {\bf{N}}(\mathfrak{D}^2 c(\pi_K)) p^{4 d \max(\alpha, \beta)}  }  \right). \end{align}
If $A \in Z(\alpha) = \ker(j:C(\alpha) \rightarrow C(\alpha-1))$ but $A \neq {\bf{1}} \in C(\alpha)$ as in Lemma \ref{MI} (i), i.e.~
in which case it is identified with the principal class $j(A) = {\bf{1}} \in C(\alpha-1)$ in taking sum over characters $C(\alpha-1)^{\vee}$
to derive the orthogonality relation, then we write $r_{\bf{1}}^{\star}$ for the corresponding counting function. We also define the sums 
\begin{align}\label{DA1star} D^{\star}_{{\bf{1}}, 1}(\pi, \chi; Z) 
&= \sum_{\mathfrak{m} \neq \lbrace 0 \rbrace \subset \mathcal{O}_F } 
\frac{ \omega \eta(\mathfrak{m}) \chi^2( {\bf{N}} \mathfrak{m} ) }{ {\bf{N}} \mathfrak{m}} \sum_{\mathfrak{n} \neq \lbrace 0 \rbrace \subset \mathcal{O}_F } 
\frac{\lambda(\mathfrak{n}) \chi( {\bf{N}} \mathfrak{n}) r_{\bf{1}}^{\star}(\mathfrak{n}) }{ {\bf{N}} \mathfrak{n}^{\frac{1}{2}}} 
V_1 \left( Z {\bf{N}} ( \mathfrak{m}^2 \mathfrak{n} ) \right). \end{align}
and 
\begin{align}\label{DA2star} D^{\star}_{{\bf{1}}, 2}( \pi, \chi; Z) &= \sum_{\mathfrak{m} \neq \lbrace 0 \rbrace \subset \mathcal{O}_F } 
\frac{ \overline{\omega} \eta(\mathfrak{m}) \overline{\chi}^2( {\bf{N}} \mathfrak{m} ) }{ {\bf{N}} \mathfrak{m}} 
\sum_{\mathfrak{n} \neq \lbrace 0 \rbrace \subset \mathcal{O}_F } 
\frac{ \overline{ \lambda(\mathfrak{n}) \chi( {\bf{N}} \mathfrak{n})} r_{\bf{1}}^{\star}(\mathfrak{n}) }{ {\bf{N}} \mathfrak{n}^{\frac{1}{2}}} 
V_2 \left( \frac{ {\bf{N}} ( \mathfrak{m}^2 \mathfrak{n} )}{ Z {\bf{N}}(\mathfrak{D}^2 c(\pi_K)) p^{4 d \max(\alpha-1, \beta)}  }  \right). \end{align}
Finally, recall that we write $\eta = \eta_{K/F}$ to denote the idele class character of $F$ associated to the quadratic extension $K/F$, 
and $\mathfrak{d} = \mathfrak{d}_F$ the different of $F$. We have the following explicit formula for the average $(\ref{galois})$.

\begin{lemma}\label{RCGA} 

Fix prime ideal $\mathfrak{p} \subset \mathcal{O}_F$ with underlying rational prime $p$ for which 
$(c(\pi), \mathfrak{D} \mathfrak{p}) = (\mathfrak{D}, \mathfrak{p}) = 1$. 
Fix a primitive even Dirichlet character $\chi \bmod p^{\beta}$ for some integer $\beta \geq 0$. 
Fix a sufficiently large integer $\alpha \geq 1$, and let $x = \operatorname{ord}_p(\#C(\alpha))$
be the exponent of $p$ in the order of the class group of $\mathcal{O}_{\mathfrak{p}^{\alpha}}$.
Then, the average $P_{\alpha}(\pi, \chi)$ over primitive ring class characters of conductor $\mathfrak{p}^{\alpha}$ is given for any choice of $Z>0$ by  
\begin{align*} &P_{\alpha}(\pi, \chi) \\
&= \left( 1 - \frac{\#C(\alpha-1)}{\# C^{\star}(\alpha)} \right) \left( D_{{\bf{1}}, 1}(\pi, \chi; Z) + \epsilon \cdot D_{{\bf{1}}, 2}(\pi, \chi; Z) \right) 
-\frac{\#C(\alpha-1)}{\# C^{\star}(\alpha)} \left( D^{\star}_{{\bf{1}}, 1}(\pi, \chi; Z) + \epsilon \cdot D^{\star}_{{\bf{1}}, 2}(\pi, \chi; Z) \right), \end{align*} 
the subaverage $P_{\alpha, \rho_0}(\pi, \chi)$ over primitive ring class characters of conductor $\mathfrak{p}^{\alpha}$ inducing a given character
$\rho_0$ on the torsion subgroup $C_0 = C(\infty)_{\infty}$ for any choice of $Z>0$ by 
\begin{align*} &P_{\alpha, \rho_0}(\pi, \chi) \\ &= \sum_{A \in C_0(\alpha)} \rho_0(A) \left( D_{A, 1}(\pi, \chi; Z) + \epsilon \cdot D_{A, 2}(\pi, \chi; Z) \right)
- \frac{\#\overline{C}(\alpha-1)}{\#P(\alpha, \rho_0)} \sum_{A \in C_0(\alpha-1) \atop A \notin C_0(\alpha)} 
\rho_0(A) \left( D_{A, 1}(\pi, \chi; Z) + \epsilon \cdot D_{A, 2}(\pi, \chi; Z) \right), \end{align*}
and the subaverage $G_{\alpha}(\pi, \chi; x)$ over primitive ring class characters 
of conductor $\mathfrak{p}^{\alpha}$ and exact order $p^x$ by    
\begin{align*} &G_{\alpha}(\pi, \chi; x) \\ &= \sum_{A \in C(\alpha)^{p^x}} \left( D_{A, 1}(\pi, \chi; Z) + \epsilon \cdot D_{A, 2}(\pi, \chi; Z) \right) 
-\frac{ \#C(\alpha, x-1)}{\#C^{\star}(\alpha, x)} \sum_{ A \in C(\alpha)^{p^{x-1}} \atop A \notin C(\alpha)^{p^x} }
\left( D_{A, 1}(\pi, \chi; Z) + \epsilon \cdot D_{A, 2}(\pi, \chi; Z) \right). \end{align*} 
Here,
\begin{align*} \epsilon = \epsilon(1/2, \pi \times \rho \chi \circ {\bf{N}}) 
&= \omega(\operatorname{lcm}(\mathfrak{p}^{\alpha}, p^{\beta} \mathcal{O}_F)) \cdot \eta(p^{4 \beta} \mathfrak{d}c(\pi)) \cdot 
\epsilon(1/2, \pi) \cdot \chi( {\bf{N}}(\mathfrak{d}^2 c(\pi)^2 \mathfrak{D}^8)) 
\cdot \left( \frac{\tau(\chi^2)}{p^{\frac{\beta}{2}}} \right)^{4d} \end{align*}
is the (unique) root number associated to each primitive ring class character $\rho$ of conductor $\mathfrak{p}^{\alpha}$ in the average. \end{lemma}

\begin{proof} 

This is simple to deduce using the formula of Lemma \ref{AFE} to express the central values, switching the order of summation, 
then applying the respective orthogonality relations of Lemma \ref{MI} (i), (ii), and (iii) to evaluate each of the coefficients. 
For the primitive average $P_{\alpha}(\pi, \chi)$, we also identify the sum over all contributions $A \in Z(\alpha)$ with 
the principal class in $C(\alpha-1)$ to derive a more convenient expression. We also use Proposition \ref{RNF} to describe the root number. \end{proof}

To estimate the averages $P_{\alpha}(\pi, \chi)$ and $G_{\alpha}(\pi, \chi; x)$ as described in Lemma \ref{RCGA}, 
we shall open up the counting functions $r_A(\mathfrak{n})$ via $(\ref{ran})$ after fixing a representative $f_A(x, y)$, which leads us to study the sums 
\begin{align}\label{DA1exp} D_{A, 1}(\pi, \chi; Z) &= \frac{1}{w_K} \sum_{ \mathfrak{m}  \subset \mathcal{O}_F \atop \mathfrak{m} \neq \lbrace 0 \rbrace } 
\frac{\eta \omega(\mathfrak{m}) \chi^2({\bf{N}} \mathfrak{m})}{ {\bf{N}} \mathfrak{m}}
\sum_{a, b \in \mathcal{O}_F/\mathcal{O}_F^{\times} \atop f_A(a, b) \neq 0} \frac{\lambda(f_A(a, b)) \chi({\bf{N}} f_A(a, b))}{ {\bf{N}} f_A(a, b)^{\frac{1}{2}}} 
V_1 \left( {\bf{N}}(\mathfrak{m}^2 f_A(a,b)) Z  \right) \end{align} and 
\begin{align}\label{DA2exp} D_{A, 2}(\pi, \chi; Z) &= \frac{1}{w_K}  \sum_{ \mathfrak{m}  \subset \mathcal{O}_F \atop \mathfrak{m} \neq \lbrace 0 \rbrace }
\frac{\eta \overline{\omega} (\mathfrak{m}) \overline{\chi}^2 ( {\bf{N}} \mathfrak{m} ) }{ {\bf{N}} \mathfrak{m}}
\sum_{a, b \in \mathcal{O}_F/\mathcal{O}_F^{\times} \atop f_A(a, b) \neq 0}  \frac{ \overline{ \lambda(f_A(a, b))}  \overline{ \chi ( {\bf{N}} f_A(a, b)) } }{ {\bf{N}} f_A(a, b)^{\frac{1}{2}}} 
V_2 \left( \frac{ {\bf{N}} (\mathfrak{m}^2 f_A(a,b))}{Z {\bf{N}}(\mathfrak{D}^2 c(\pi_K)) p^{4d \max (\alpha, \beta)}  }\right). \end{align} 
The difference sums $D_{{\bf{1}}, j}^{\star}(\pi, \chi; Z)$ for $j=1,2$ are estimated in the same way, and so we omit them from the main discussion,
including them at the end when we assemble various estimates to describe the averages. 

\subsubsection{Estimates} 

To estimate either of the averages $P_{\alpha}(\pi, \chi)$ and $G_{\alpha}(\pi, \chi; x)$, we shall first consider the contributions 
from the $b = 0$ and $a=0$ terms in $(\ref{DA1exp})$ and $(\ref{DA2exp})$ with the following estimate. 
Fix a class $A \in C(\alpha)$, together with a reduced quadratic form class representative $f_A(x, y) = a_A x^2 + b_A xy + c_A y^2$. 
Let us for a given divisor $q$ of $a_A$ consider the shift by right multiplication by 
$ \left( \begin{array}{cc} q & ~\\ ~& 1 \end{array}\right)$ of a vector $\phi \in V_{\pi}$, 
\begin{align*} \phi \in V_{\pi}, \phi(g) &\longmapsto \phi \left( g  \left( \begin{array}{cc} q^{-1} & ~\\ ~& 1 \end{array}\right) \right) \in V_{\pi}. \end{align*} 
Here, we also write $q \in {\bf{A}}_{F, f}^{\times}$ to denote a fixed finite idele representative of the $F$-integer $q \in \mathcal{O}_F$. 
We also write $\lambda^{(q)}(a) = \lambda(q^{-1}a)$ as shorthand to denote the coefficients in the corresponding Dirichlet series expansion,
and consider special values of the corresponding partial congruence symmetric square $L$-function,
defined here (first for $\Re(s) >1$) by the Dirichlet series expansion 
\begin{equation}\begin{aligned}\label{PDE} L^{\star}_{{\bf{1}}, q}(s, \operatorname{Sym}^2 \pi^{(q)} \otimes \chi \circ {\bf{N}}) 
&:= L^{\star}_{{\bf{1}}}(2s, \omega \chi^2 \circ {\bf{N}}) 
\sum_{a \neq 0 \in \mathcal{O}_F\mathcal{O}_F^{\times} \atop a \equiv 0 \bmod q \mathcal{O}_F} \frac{\lambda^{(q)}(a^2)\chi({\bf{N}}a^2) }{{\bf{N}}a^s} \\
&= L^{\star}_{{\bf{1}}}(2s, \omega \chi^2 \circ {\bf{N}}) 
\sum_{a \neq 0 \in \mathcal{O}_F/\mathcal{O}_F^{\times} \atop a \equiv 0 \bmod q \mathcal{O}_F} 
\frac{ \lambda ( a^2 q^{-1} )\chi( {\bf{N}}a^2) }{ {\bf{N}} a^s}, \end{aligned}\end{equation}
i.e.~where the symbol $\pi^{(q)}$ is also shorthand notation. Note that if $q=1$, then this is simply
\begin{align*}L^{\star}_{{\bf{1}}, 1}(s, \operatorname{Sym}^2 \pi^{(q)} \otimes \chi \circ {\bf{N}}) 
= L^{\star}_{{\bf{1}}}(s, \operatorname{Sym}^2 \pi^{(q)} \otimes \chi \circ {\bf{N}}). \end{align*}
If $a_A=1$, as will often be the case later, this is the only divisor we consider. 

\begin{proposition}\label{RCres} 

Assume that the representation $\pi \otimes \xi$ is non-dihedral\footnote{This is assumed for simplicity; a variation of this and subsequent bounds
can be derived in the dihedral case, but we do not pursue it here directly.}, i.e.~that $\pi$ is not the induced representation of some Hecke character
of a quadratic extension of $F$ in the special case where $\xi$ is trivial. Let us keep the setup of Lemma \ref{RCGA}, fixing an integer $\alpha \geq 0$, 
as well as a primitive even Dirichlet character $\chi \bmod p^{\beta}$ for some integer $\beta \geq 0$. 
Given a class $A \in C(\alpha)$, we fix $f_A(x, y) = a_A x^2 + b_A x y + c_A y^2$ to be the reduced quadratic form class representative. 
We again write $0 \leq \theta_0 \leq 1/2$ to denote the best approximation to the generalized Ramanujan conjecture
for $\operatorname{GL}_2({\bf{A}}_F)$-automorphic forms in the level aspect.
Let us for simplicity write $c(\omega \eta \chi^2 \circ {\bf{N}}) = {\bf{N}}(\mathfrak{D}c(\pi) p^{\beta} \mathcal{O}_F)$
to denote the absolute norm of the conductor of the Hecke $L$-function $L(s, \omega \eta \chi^2 \circ {\bf{N}})$. 
Writing $\mu$ to denote the M\"obius function on ideals in $\mathcal{O}_F$,
let us also define for any $F$-integer $a_A$ the residual quantity 

\begin{equation}\begin{aligned}\label{residue}  
R(\pi, \chi, a_A) &= \frac{1}{w_K} \cdot L(1, \omega \eta \chi^2 \circ {\bf{N}}) \cdot \sum\limits_{ q  \mid a_A  } 
\left( \frac{\mu(q) \lambda^{(q)}(a_A) \omega(q) \chi ({\bf{N}}a_A ) }{ {\bf{N}}a_A^{\frac{1}{2}} } \right) 
\cdot \frac{  L^{\star}_{{\bf{1}}, q}(1, \operatorname{Sym}^2 \pi^{(q)} \otimes \chi \circ {\bf{N}})}{ L^{\star}_{{\bf{1}}}(2, \omega \chi^2 \circ {\bf{N}}) }. \end{aligned}\end{equation}
Note that the sum $(\ref{residue})$ does not vanish when $a_A=1$ (whence there is only one divisor $q=a_A =1$)
as a consequence of the nonvanishing of the symmetric square $L$-function $L(s, \operatorname{Sym}^2 \pi \otimes \xi)$ at $s=1$. 
Let us also write $0 \leq \delta_3 \leq 1/4$ to denote the best subconvexity estimate for the twisted symmetric square $L$-function
$L(s, \operatorname{Sym}^2 \pi \otimes \chi \circ {\bf{N}})$ towards the generalized Lindel\"of hypothesis for $\operatorname{GL}_3({\bf{A}}_F)$-automorphic $L$-functions
in the level aspect, i.e.~so that $L(1/2, \operatorname{Sym}^2 \pi \otimes \chi \circ {\bf{N}}) \ll_{\varepsilon}  {\bf{N}}(c(\pi)p^{3 \beta} \mathcal{O}_F)^{\delta_3 + \varepsilon}$
for any choice of $\varepsilon >0$.\\

\noindent (i) The $b = 0$ contributions $D_{A, 1}(\pi, \chi; Z)\vert_{b=0}$ in the expansion $(\ref{DA1exp})$ of $D_{A, 1}(\pi, \chi; Z)$ 
can be estimated for any choice of parameter $Z = Y^{-1}$ with $Y > {\bf{N}} a_A$ as follows: We have for any $\varepsilon >0$ that  
\begin{align*} D_{A, 1}(\pi, \chi; Y^{-1})\vert_{b=0} &= R(\pi, \chi, a_A) 
+ O_{\varepsilon} \left( c(\omega \eta \chi^2 \circ {\bf{N}})^{\frac{1}{4} - \frac{(1-2\theta_0)}{16} + \varepsilon} \cdot 
 {\bf{N}}(c(\pi)p^{3 \beta} \mathcal{O}_F)^{\delta_3 + \varepsilon} \cdot \left( \frac{ {\bf{N}} a_A}{Y} \right)^{\frac{1}{4}}  \right). \end{align*} 

\noindent (ii) The $b=0$ contributions $D_{A, 2}(\pi, \chi; Z)\vert_{b=0}$ in the expansion $(\ref{DA2exp})$ of $D_{A, 2}(\pi, \chi; Z)$ can be estimated 
for any choice of $Z > 0$ for which $ Z \cdot {\bf{N}}(\mathfrak{D}^2 c(\pi_K)) \cdot p^{4d \max (\alpha, \beta)} > {\bf{N}} a_A$ as follows: For any small $\epsilon >0$, 
\begin{align*} &D_{A, 2}(\pi, \chi; Z)\vert_{b=0} \\ &= \frac{\widetilde{L}_{\infty}(\frac{1}{2})}{L_{\infty}(\frac{1}{2})} \cdot 
R(\widetilde{\pi}, \overline{\chi}, a_A) + O_{\varepsilon} \left( c(\omega \eta \chi^2 \circ {\bf{N}})^{\frac{1}{4} - \frac{(1-2\theta_0)}{16} + \varepsilon} \cdot 
{\bf{N}}(c(\pi)p^{3 \beta} \mathcal{O}_F)^{\delta_3 + \varepsilon} \cdot 
\left( \frac{ {\bf{N}} a_A}{Z  {\bf{N}}(\mathfrak{D}^2 c(\pi_K)) p^{4d \max (\alpha, \beta)}} \right)^{\frac{1}{4}}  \right). \\ \end{align*} 

If on the other hand $Z >0$ is chosen so that $0 < Z \cdot {\bf{N}}(a_A \mathfrak{D}^2 c(\pi_K)) \cdot p^{4d \max (\alpha, \beta)} < 1$, then we have 
\begin{align*} D_{A, 2}(\pi, \chi; Z)\vert_{b=0}  
&=  O_{B} \left( \left( Z {\bf{N}}(a_A^{-1} \mathfrak{D}^2 c(\pi_K) p^{4d \max (\alpha, \beta)})  \right)^{B} \right) \end{align*} and more generally 
\begin{align*} D_{A, 2}(\pi, \chi; Z)  &=  O_{B} \left( \left( Z {\bf{N}}(a_A^{-1}\mathfrak{D}^2 c(\pi_K) p^{4d \max (\alpha, \beta)})  \right)^{B} \right) \end{align*}
for any choice(s) of constant(s) $B \geq 1$. \end{proposition}

\begin{proof} 

For (i), we expand out the sum and open up the definition of the cutoff function $V_1$ to find that 
\begin{align*} D_{A, 1}(\pi, \chi; Z)\vert_{b=0} &= \frac{1}{w_K} \sum_{ \mathfrak{m} \neq \lbrace 0 \rbrace \subset \mathcal{O}_F } 
\frac{ \eta \omega(\mathfrak{m}) \chi^2( {\bf{N}} \mathfrak{m} ) }{ {\bf{N}} \mathfrak{m} } 
\sum_{ a \neq 0 \in \mathcal{O}_F/\mathcal{O}_F^{\times}}  \frac{ \lambda(a_A a^2) \chi( {\bf{N}} (a_A a^2)) }{ {\bf{N}}(a_A a^2)^{\frac{1}{2}}} 
V_1 \left(  {\bf{N}}( \mathfrak{m}^2 a_A a^2) Z \right). \end{align*} 
Let us first consider the Hecke relation, which for each $F$-integer $a \in \mathcal{O}_F$ in the sum takes the form 
\begin{align*}\lambda(a_A a^2) &= \sum_{q \mid \gcd (a_A, a^2)} \mu \left(q \right) \omega(q) 
\lambda \left( \frac{a_A}{q} \right) \lambda \left( \frac{a}{q} \right). \end{align*}
Hence, we find that for any $\mathfrak{m} \subset \mathcal{O}_F$ in the latter expression for $D_{A, 1}(\pi, \chi; Z)\vert{b=0}$, 
\begin{equation*}\begin{aligned} 
&\sum_{ a \neq 0 \in \mathcal{O}_F/\mathcal{O}_F^{\times}}  \frac{ \lambda(a_A a^2) \chi( {\bf{N}} (a_A a^2)) }{ {\bf{N}}(a_A a^2)^{\frac{1}{2}}} 
V_1 \left(  {\bf{N}}( \mathfrak{m}^2 a_A a^2) Z \right) \\ 
&= \sum_{ a \neq 0 \in \mathcal{O}_F/\mathcal{O}_F^{\times}}  \sum_{q \mid \gcd (a_A, a)}
\frac{  \mu \left( q \right) \omega (q)  \lambda \left( \frac{a_A}{q} \right) \lambda \left( \frac{a^2}{q} \right) \chi( {\bf{N}} (a_A a^2)) }{ {\bf{N}}(a_A a^2)^{\frac{1}{2}}} 
V_1 \left(  {\bf{N}}( \mathfrak{m}^2 a_A a^2) Z \right) \\ &= 
\sum_{q \mid a_A} \mu(q) \omega (q) \frac{\lambda \left( \frac{a_A}{q} \right) \chi({\bf{N}} a_A)  }{ {\bf{N}}a_A^{\frac{1}{2}} }
\sum_{a' \neq 0 \in \mathcal{O}_F/ \mathcal{O}_F^{\times} } \frac{ \lambda \left( \frac{(qa')^2}{q} \right) \chi( {\bf{N}}(q^2 a'^2 )) }{ {\bf{N}}((q a')^2)^{\frac{1}{2}}} 
V_1 \left(  {\bf{N}}( \mathfrak{m}^2 a_A (q a')^2) Z \right) \\ 
&= \sum_{q \mid a_A} \mu(q) \omega (q) \frac{\lambda^{(q)} \left( a_A \right) \chi({\bf{N}} a_A)  }{ {\bf{N}}a_A^{\frac{1}{2}} }
\sum_{a' \neq 0 \in \mathcal{O}_F/ \mathcal{O}_F^{\times} } \frac{ \lambda^{(q)} \left( q^2 a'^2 \right) \chi( {\bf{N}} (q^2 a'^2)) }{ {\bf{N}}(q a'^2)^{\frac{1}{2}}} 
V_1 \left(  {\bf{N}}( \mathfrak{m}^2 a_A (q a')^2) Z \right) \\
&=  \sum_{q \mid a_A} \mu(q) \omega (q) \frac{\lambda^{(q)} \left( a_A \right) \chi({\bf{N}} a_A)  }{ {\bf{N}} a_A^{\frac{1}{2}} }
\sum_{a \neq 0 \in \mathcal{O}_F/\mathcal{O}_F^{\times} \atop a \equiv 0 \bmod q \mathcal{O}_F} \frac{ \lambda^{(q)} \left( a^2 \right) \chi( {\bf{N}}a^2 ) }{ {\bf{N}}a} 
V_1 \left(  {\bf{N}}( \mathfrak{m}^2 a_A a^2) Z \right). \end{aligned}\end{equation*}
Let us now consider any of the inner sums corresponding to a given divisor $q \mid a_A$ in this latter expression, 
whose contribution to $D_{A, 1}(\pi, \chi; Z) \vert_{b=0}$ is then seen to be given explicitly by 
\begin{align*} & \mu(q) \omega (q) \frac{\lambda^{(q)} \left( a_A \right) \chi({\bf{N}} a_A)  }{ {\bf{N}} a_A^{\frac{1}{2}} }
\int_{\Re(s)=2} \frac{k(s)}{s} \sum_{\mathfrak{m} \neq \lbrace 0 \rbrace \subset \mathcal{O}_F} 
\frac{\eta \omega(\mathfrak{m}) \chi^2( {\bf{N}} \mathfrak{m})}{ {\bf{N}} \mathfrak{m}^{1 + 2s}} 
\sum_{a \neq 0 \in \mathcal{O}_F/\mathcal{O}_F^{\times} \atop a \equiv 0 \bmod q \mathcal{O}_F} 
\frac{ \lambda^{(q)}(a^2) \chi( {\bf{N}} a^2) }{ {\bf{N}}(a^2)^{\frac{1}{2} + s} } (a_A Z)^{-s} \frac{ds}{2 \pi i} \\ 
&= \mu(q) \omega (q) \frac{\lambda^{(q)} \left( a_A \right) \chi({\bf{N}} a_A)  }{ {\bf{N}} a_A^{\frac{1}{2}} }
\int_{\Re(s)=2} \frac{k(s)}{s} \cdot L(2s+1, \omega\eta \chi^2 \circ {\bf{N}} ) \cdot 
\frac{ L_{{\bf{1}}, q}^{\star}(2s + 1, \operatorname{Sym}^2 \pi^{(q)} \otimes \chi \circ  {\bf{N}} ) }{L_{ {\bf{1}}}^{\star}(4s+ 2, \omega \chi^2 \circ {\bf{N}} ) } 
(a_A Z)^{-s} \frac{ds}{2 \pi i}. \end{align*} 
Shifting the line of integration to $\Re(s)=-1/4$, we then cross a simple pole at $s=0$ of residue 
\begin{align*} \mu(q) \omega (q) \frac{\lambda^{(q)} \left( a_A \right) \chi({\bf{N}} a_A)  }{ {\bf{N}} a_A^{\frac{1}{2}} } \cdot L(1, \omega \eta \chi^2 \circ {\bf{N}}) 
\cdot \frac{ L_{{\bf{1}}, q}^{\star}(1, \operatorname{Sym}^2 \pi^{(q)} \otimes \chi \circ {\bf{N}})}{ L_{\bf{1}}^{\star}(2, \omega \chi^2 \circ {\bf{N}})  }. \end{align*}
The remaining integral is seen easily to be bounded by 
\begin{align*} \ll_{\varepsilon}  c(\omega \eta \chi^2 \circ {\bf{N}})^{\frac{1}{4} - \frac{(1-2\theta_0)}{16} + \varepsilon} \cdot 
 {\bf{N}}(c(\pi)p^{3 \beta} \mathcal{O}_F)^{\delta_3 + \varepsilon} \cdot ({\bf{N}} a_A Z)^{\frac{1}{4}}\end{align*}
using Stirling's approximation formula with a suitable subconvexity bound to estimate the Hecke $L$-series. 
Here, we use the Burgess-like subconvexity bound 
$L(1/2, \omega \eta \chi^2 \circ {\bf{N}}) \ll_{\varepsilon} c(\omega \eta \chi^2 \circ {\bf{N}})^{\frac{1}{4} - \frac{(1-2\theta_0)}{16} + \varepsilon}$
shown in \cite{Wu}, as well as the best existing subconvexity bound towards the generalized Lindel\"of hypothesis in the level aspect 
for the twisted symmetric square $L$-function 
$L(s, \operatorname{Sym}^2 \pi \otimes \chi \circ {\bf{N}}) \ll_{\varepsilon}  {\bf{N}}(c(\pi)p^{3 \beta} \mathcal{O}_F)^{\delta_3 + \varepsilon}$,
e.g.~viewed as a $\operatorname{GL}_3({\bf{A}}_F)$-automorphic $L$-function via the Gelbart-Jacquet lift (cf.~also \cite[Lemma 4.1]{CM}).

For (ii), we proceed in the same way, first noting that for any $\mathfrak{m} \subset \mathcal{O}_F$ we have the Hecke decomposition 
\begin{equation*}\begin{aligned} &D_{A, 2}(\pi, \chi; Z)\vert_{b=0} 
= \frac{1}{w_K}\sum_{\mathfrak{m} \subset \mathcal{O}_F} 
\frac{\overline{\omega} \eta(\mathfrak{m}) \overline{\chi}^2({\bf{N}} \mathfrak{m})}{{\bf{N}} \mathfrak{m}}
\sum_{ a \neq 0 \in \mathcal{O}_F / \mathcal{O}_F^{\times}}  \frac{ \overline{ \lambda(a_A a^2) \chi( {\bf{N}} (a_A a^2))} }{ {\bf{N}}(a_A a^2)^{\frac{1}{2}}} 
V_2 \left( \frac{  {\bf{N}} ( \mathfrak{m}^2 a_A a^2) }{Z C} \right) \\ 
&= \frac{1}{w_K} \sum_{\mathfrak{m} \subset \mathcal{O}_F} \frac{\overline{\omega}\eta(\mathfrak{m}) \overline{\chi}^2({\bf{N}} \mathfrak{m})}{{\bf{N}} \mathfrak{m}}
\sum_{q \mid a_A} \mu(q) \overline{\omega(q)} \cdot \frac{\overline{ \lambda^{(q)} \left( a_A \right) \chi({\bf{N}} a_A)} }{ {\bf{N}} a_A^{\frac{1}{2}} }
\sum_{a \neq 0 \in \mathcal{O}_F/ \mathcal{O}_F^{\times} \atop a \equiv 0 \bmod q \mathcal{O}_F} \frac{ \overline{ \lambda^{(q)} \left( a^2 \right) \chi( {\bf{N}}a^2 )} }{ {\bf{N}}a} 
V_2 \left(  \frac{ {\bf{N}}( \mathfrak{m}^2 a_A a^2)}{ Z C} \right). \end{aligned}\end{equation*}
Here, we write $C = {\bf{N}}(\mathfrak{D}^2 c(\pi_K)) p^{4d \max (\alpha, \beta)}$ for simplicity to denote the conductor. 
Let us now consider any of the inner $q$-sums in this latter expansion, opening up the function $V_2$ to find 
\begin{align*} &\mu(q) \overline{\omega(q)} \cdot \frac{\overline{ \lambda^{(q)} \left( a_A \right) \chi({\bf{N}} a_A)} }{ {\bf{N}} a_A^{\frac{1}{2}} }
\sum_{ \mathfrak{m} \neq \lbrace 0 \rbrace \subset \mathcal{O}_F } 
\frac{ \eta \overline{\omega}(\mathfrak{m}) \overline{\chi}^2( {\bf{N}} \mathfrak{m} ) }{ {\bf{N}} \mathfrak{m} } 
\sum_{ a \neq 0 \in \mathcal{O}_F/\mathcal{O}_F^{\times} \atop a \equiv 0 \bmod q \mathcal{O}_F} 
\frac{ \overline{ \lambda^{(q)} \left( a^2 \right) \chi( {\bf{N}}a^2 )} }{ {\bf{N}}a}
V_2 \left(  \frac{ {\bf{N}}( \mathfrak{m}^2 a_A a^2)}{ Z C} \right) \\ 
&= \mu(q) \overline{\omega(q)} \cdot \frac{\overline{ \lambda^{(q)} \left( a_A \right) \chi({\bf{N}} a_A)} }{ {\bf{N}} a_A^{\frac{1}{2}} } 
\int_{\Re(s)=2} \frac{k(-s)}{s} \frac{\widetilde{L}_{\infty}(s + \frac{1}{2})}{L_{\infty}(-s + \frac{1}{2})} 
\sum_{\mathfrak{m} \neq \lbrace 0 \rbrace \subset \mathcal{O}_F} 
\frac{\eta \overline{\omega}(\mathfrak{m}) \overline{\chi}^2( {\bf{N}} \mathfrak{m})}{ {\bf{N}} \mathfrak{m}^{1 + 2s}} \\
&\times \sum_{a \neq 0 \in \mathcal{O}_F/\mathcal{O}_F^{\times} \atop a \equiv 0 \bmod q \mathcal{O}_F }  
\frac{ \overline{ \lambda^{(q)}(a^2) \chi( {\bf{N}} a^2)} }{ {\bf{N}} a^{1 + 2s} } 
\left( {\bf{N}}a_A^{-1} ZC \right)^{s} \frac{ds}{2 \pi i} \\ 
&= \mu(q) \overline{\omega(q)} \cdot \frac{\overline{ \lambda^{(q)} \left( a_A \right) \chi({\bf{N}} a_A)} }{ {\bf{N}} a_A^{\frac{1}{2}} }
\int_{\Re(s)=2} \frac{k(-s)}{s} \frac{\widetilde{L}_{\infty}(s + \frac{1}{4})}{L_{\infty}(-s + \frac{1}{2})} \\ &\times 
L(2s+1, \overline{\omega} \eta \overline{\chi}^2 \circ {\bf{N}} ) \cdot 
\frac{ L_{{\bf{1}},1}^{\star}(2s + 1, \operatorname{Sym}^2 \widetilde{\pi}^{(q)} \otimes \overline{\chi} \circ  {\bf{N}} ) }
{L_{ {\bf{1}}}^{\star}(4s + 2, \overline{\omega} \overline{\chi}^2 \circ {\bf{N}} ) } 
\left( \frac{ {\bf{N}}a_A}{ZC} \right)^{-s} \frac{ds}{2 \pi i}. \\ \end{align*}
Now if ${\bf{N}} a_A > ZC$, then we shift the contour leftward
to $\Re(s) = -1/8$ again to derive the stated estimate (again using the bounds described above for the contributions of the $L$-functions in the contour), 
noting that the same argument applies to estimate the entire sum $D_{A, 2}(\pi, \chi; Z)$. 
Otherwise, we shift the contour to the left, and use the bound of Lemma \ref{cutoff} for $V_2(y)$ to derive the stated estimate. \end{proof}

It remains to estimate the contribution from $b \neq 0$ terms in the region of moderate decay for 
cutoff functions $V_j$ in each each of the respective sums $D_{A, j}(\pi, \chi; Z)$. 
For all of the subsequent discussion following Proposition \ref{RCres} above, 
we shall choose the unbalancing parameter to be within the interval $0<Z<1$, 
and often simply $Z = Y^{-1}$ for $Y = C^{\frac{1}{2}}$ the square root of the conductor 
$C = {\bf{N}}(\mathfrak{D}^2 c(\pi_K) ) p^{4d \max(\alpha, \beta)}$
corresponding to a balanced approximate functional equation. By the decay properties of $V_j$, 
it will then do to bound the truncated finite double sums defined for an arbitrary small $\varepsilon >0$ by  
\begin{equation}\begin{aligned}\label{ts1} 
&D_{A, 1}^{\dagger}(\pi, \chi; Z) = \frac{1}{w_K} \sum_{ \mathfrak{m}  \subset \mathcal{O}_F \atop \mathfrak{m} \neq \lbrace 0 \rbrace } 
\frac{\eta \omega(\mathfrak{m}) \chi^2({\bf{N}} \mathfrak{m})}{ {\bf{N}} \mathfrak{m}} 
\sum\limits_{ {a, b \in \mathcal{O}_F / \mathcal{O}_F^{\times}  \atop b \neq 0} \atop {\bf{N}}(\mathfrak{m}^2 f_A(a,b))  \leq Z^{-1 + \varepsilon}   } 
\frac{\lambda_{\chi}(f_A(a,b)) }{ {\bf{N}}(f_A(a,b))^{\frac{1}{2}}} 
 V_1 \left( {\bf{N}}(\mathfrak{m}^2(f_A(a,b))) Z  \right) \end{aligned}\end{equation} and 
\begin{equation}\begin{aligned}\label{ts2} 
&D_{A, 2}^{\dagger}(\pi, \chi; Z) = \frac{1}{w_K}  \sum_{ \mathfrak{m}  \subset \mathcal{O}_F \atop \mathfrak{m} \neq \lbrace 0 \rbrace }
\frac{\eta \overline{\omega} (\mathfrak{m}) \overline{\chi}^2 ( {\bf{N}} \mathfrak{m} ) }{ {\bf{N}} \mathfrak{m}}
\sum\limits_{ {a, b \in \mathcal{O}_F / \mathcal{O}_F^{\times} \atop b \neq 0} \atop {\bf{N}} (\mathfrak{m}^2 (f_A(a,b))) \leq (CZ)^{1+ \varepsilon} }  
\frac{ \overline{ \lambda_{\chi}(f_A(a,b))}  }{ {\bf{N}}(f_A(a,b))^{\frac{1}{2}}} 
V_2 \left( \frac{ {\bf{N}} (\mathfrak{m}^2 (f_A(a,b)))}{Z C }\right). \end{aligned}\end{equation} 
Again, we write $\lambda_{\chi}(\mathfrak{n}) := \lambda(\mathfrak{n}) \chi( {\bf{N}} \mathfrak{n})$ for $\mathfrak{n} \subset \mathcal{O}_F$.
Note that these sums are only defined for certain choices of unbalancing parameter $Z >0$, and in particular only need to be considered 
in the event that there are contributions $a \in \mathcal{O}_F$ and 
$b \neq 0 \in \mathcal{O}_F$ in the region of moderate decay for the corresponding function $V_j$. 
 
Let us for simplicity write $\xi = \otimes_v \xi_v$ to denote the idele class character of $F$ 
determined by composition with the norm $\chi \circ {\bf{N}}$ with our chosen Dirichlet character $\chi \bmod p^{\beta}$. 

\begin{theorem}\label{RCOD} 

Let $\pi$ be any cuspidal $\operatorname{GL}_2({\bf{A}}_F)$-automorphic form of level $c(\pi)$ and central character $\omega$. 
Let $K/F$ a totally imaginary quadratic extension of absolute discriminant $D_K = {\bf{N}} \mathfrak{D}$, 
and $\mathfrak{p} \subset \mathcal{O}_F$ a fixed prime ideal with underlying rational prime $p$. 
Assume that $(c(\pi), \mathfrak{D} \mathfrak{p}) = (\mathfrak{p}, \mathfrak{D}) = 1$. 
Recall we fix integers $\alpha \geq 0$ and $\beta \geq 0$ corresponding to 
$\mathcal{W} = \rho \chi \circ {\bf{N}}$ a Hecke character of $K$ with primitive ring class component 
$\rho$ of conductor $\mathfrak{p}^{\alpha}$ and cyclotomic component $\xi = \chi \circ {\bf{N}}$ 
induced from a primitive even Dirichlet character $\chi \bmod p^{\beta}$.
We have the following estimates for the sums $D_{A, j}^{\dagger}(\pi, \chi; Z)$ and $D_{A, j}(\pi, \chi; Z)$ 
for $j=1,2$, for any class $A \in C(\alpha)$ in the class group of the order 
$\mathcal{O}_{\mathfrak{p}^{\alpha}} = \mathcal{O}_F + \mathfrak{p}^{\alpha} \mathcal{O}_K$. 
Again, we shall fix a quadratic form class representative $f_A(x, y) = a_A x^2 + b_A xy +c_A y^2$ for each class $A \in C(\alpha)$. \\
  
\noindent (i) Suppose $\pi \otimes \xi$ is non-dihedral 
(i.e.~not induced from a Hecke character of a quadratic extension of $F$) if $\xi \neq {\bf{1}}$ is nontrivial. 
Fix a class $A \in C(\alpha)$, and let $f_A(x, y) = a_A x^2 + b_A xy + c_A y^2$ be the reduced class form representative. 
Assume that $a_A=1$ and $b_A=0$,
as is the case when $A$ is principal and $D_K \equiv 0 \bmod 4$. 
We have for any choice of real numbers $Y \geq 1$ and $\varepsilon >0$ the upper bounds 
\begin{align*} D_{A, 1}^{\dagger}(\pi, \chi; Y^{-1}) &\ll_{\pi, \chi, \varepsilon} 
Y^{\frac{1}{4} + \delta_0 + \varepsilon} \cdot {\bf{N}} c_A^{- \frac{1}{2}} \end{align*} and 
\begin{align*} D_{A, 2}^{\dagger}(\pi, \chi; Y^{-1}) &\ll_{ \pi, \chi, \varepsilon} 
\left( \frac{ {\bf{N}}(\mathfrak{D}^2 c(\pi_K)) p^{4d \max (\alpha, \beta)}  }{Y} \right)^{\frac{1}{4} + \delta_0 + \varepsilon} \cdot {\bf{N}} c_A^{- \frac{1}{2}}. \\ \end{align*}
Here again, $0 < \delta_0 < 1/4$ denotes the best approximation to the generalized Lindel\"of hypothesis for 
$\operatorname{GL}_2({\bf{A}}_F)$-automorphic forms in the level aspect, and hence we can take $\delta_0 = 103/512$ 
by \cite[Corollary 1]{BH10}, using the approximation $\theta_0 = 7/64$ to the generalized Ramanujan conjecture
for $\operatorname{GL}_2({\bf{A}}_F)$ given in \cite{BBRP}. We can therefore take $\theta_0 = 7/64$ and $\delta_0 = 103/512$ in 
these statements to obtain unconditional estimates with exponents $1/4 - (1-2\theta_0)/16 = 206/1024$ and $1/4 + \delta_0 = 231/512$.
In particular, if $A = {\bf{1}} \in C(\alpha)$ is the principal class, then  
$D_{\bf{1}}(\pi, \chi) := D_{{\bf{1}}, 1}(\pi, \chi; Y^{-1}) + \epsilon \cdot D_{{\bf{1}}, 2}(\pi, \chi; Y^{-1})$
converges with $\alpha \rightarrow \infty$ to the constant 
\begin{align}\label{nvres}\frac{1}{w_K} \left( L(1, \omega \eta \chi^2 \circ {\bf{N}}) 
\cdot \frac{ L_{\bf{1}}^{\star}(1, \operatorname{Sym}^2 \pi \otimes \chi \circ {\bf{N}}) }{ L_{\bf{1}}^{\star}(2, \omega \chi^2 \circ {\bf{N}})} 
+ \epsilon \cdot \frac{\widetilde{L}_{\infty}(\frac{1}{2})}{L_{\infty}(\frac{1}{2})} \cdot L(1, \overline{\omega} \eta \overline{\chi}^2 \circ {\bf{N}}) 
\cdot \frac{ L_{\bf{1}}^{\star}(1, \operatorname{Sym}^2 \widetilde{\pi} \otimes \overline{\chi} \circ {\bf{N}}) }
{ L_{\bf{1}}^{\star}(2, \overline{\omega} \overline{\chi}^2 \circ {\bf{N}})} \right).\end{align} 
Here, the sum $(\ref{nvres})$ is seen by inspection to be nonvanishing (using positivity of the $L$-values) 
in the special case where $\pi \cong \widetilde{\pi}$ is self-dual and $\chi = {\bf{1}}$ is the principal Dirichlet character.
In general, so long as $\pi \cong \widetilde{\pi}$ is self-dual, then we can also show that this sum of 
residual terms $(\ref{nvres})$ with $\chi \neq {\bf{1}}$ nontrivial is nonvanishing. \\ 

\noindent (ii) If in the setup of (i) above we do not impose any condition on the absolute discriminant $D_K$
or the coefficients $a_A$, $b_A$, and $c_A$ of the reduced quadratic form class representative $f_A(x,y) = a_A x^2 + b_A xy + c_A y^2$,
then we can derive for any choices of real numbers $Y \geq 1$ and $\varepsilon >0$ the stronger estimates 
\begin{align*} D_{A, 1}^{\dagger}(\pi, \chi; Y^{-1}) &\ll_{\pi, \chi, \varepsilon} 
 {\bf{N}} a_A \cdot Y^{\frac{1}{4} + \delta_0} \cdot {\bf{N}}(\mathfrak{p}^{2 \alpha} \mathfrak{D})^{\delta_0 - \frac{\theta_0}{2} - \varepsilon} 
\cdot {\bf{N}}c_A^{- \frac{1}{2} - \delta_0 + \frac{\theta_0}{2} + \varepsilon} \end{align*} and 
\begin{align*} D_{A, 2}^{\dagger}(\pi, \chi; Y^{-1}) &\ll_{ \pi \otimes \xi, \varepsilon} 
 {\bf{N}} a_A \cdot Y^{\frac{1}{4} + \delta_0} \cdot {\bf{N}}(\mathfrak{p}^{2 \alpha} \mathfrak{D})^{\delta_0 - \frac{\theta_0}{2} - \varepsilon} 
\cdot {\bf{N}}c_A^{- \frac{1}{2} - \delta_0 + \frac{\theta_0}{2} + \varepsilon}. \end{align*} 
Let us remark that while these bounds appear on the surface to be uniform in $a_A$, 
they are weaker for ${\bf{N}} a_A >1$ as the constraints on the quadratic form $f_A(x, y)$
then force the quantity ${\bf{N}} c_A$ to be strictly smaller, i.e.~so that we detect less cancellation in the corresponding shifted convolution sums. \end{theorem} 

\begin{proof} 

To show (i), let $V$ be any smooth function of compact support on $[1/2, 1]$ satisfying $V^{(i)} \ll 1$ for all $i \geq 0$. 
Given $R \geq 1$ any real number and $c_A \in \mathcal{O}_F$ any nonzero $F$-integer, we have that 
\begin{equation*}\begin{aligned}
\bigg| &\sum_{\mathfrak{m} \subset \mathcal{O}_F} \frac{\eta \omega (\mathfrak{m}) \chi^2({\bf{N}} \mathfrak{m})}{ {\bf{N}}\mathfrak{m} } 
\sum\limits_{a, b \in \mathcal{O}_F/\mathcal{O}_F^{\times} \atop b \neq 0} \frac{\lambda_{\chi}(a^2 + c_A b^2)}{{\bf{N}}(a^2 + c_A b^2)^{\frac{1}{2}} }
V \left( \frac{ {\bf{N}}\mathfrak{m}^2(a^2 + c_A b^2) }{R} \right) \bigg| \\
&\leq \sum_{ \mathfrak{m} \subset \mathcal{O}_F } \frac{1}{ {\bf{N}} \mathfrak{m} } 
\sum_{b \neq 0 \in \mathcal{O}_F/ \mathcal{O}_F^{\times} } \bigg| 
\sum_{a \in \mathcal{O}_F/\mathcal{O}_F^{\times}} \frac{\lambda_{\chi}(a^2 + c_A b^2)}{{\bf{N}}(a^2 + c_A b^2)^{\frac{1}{2}}}  
V \left( \frac{ {\bf{N}} \mathfrak{m}^2(a^2 + c_A b^2)}{ R } \right) \bigg|, \end{aligned}\end{equation*}
which after applying Theorem \ref{SCS} to the inner sum is bounded above for any choice of $\varepsilon >0$ by 
\begin{equation*}\begin{aligned} 
&\ll \sum_{b \in \mathcal{O}_F/\mathcal{O}_F^{\times} \atop {\bf{N}} b \leq (R/{\bf{N}}c_A)^{\frac{1}{2}}  } R^{- \frac{1}{4} + \frac{\theta_0}{2} + \varepsilon}
{\bf{N}}(c_A b^2)^{\delta_0 - \frac{\theta_0}{2}} \\
&\ll ({\bf{N}}c_A)^{\delta_0 - \frac{\theta_0}{2}} R^{-\frac{1}{4} + \frac{\theta_0}{2} + \varepsilon}
\left( \frac{R}{{\bf{N}}c_A} \right)^{\frac{1}{2} + \delta_0 - \frac{\theta_0}{2}}
= R^{\frac{1}{4} + \delta_0 + \varepsilon} ({\bf{N}} c_A)^{- \frac{1}{2}}. \end{aligned}\end{equation*}
In particular, using a smooth partition of unity and the rapid decay of the cutoff functions 
$V_j$ ($j=1,2$) in $(\ref{ts1})$ and $(\ref{ts2})$, we deduce the claimed bounds. 

Let us now consider the sum of two residual terms $(\ref{nvres})$. 
Note that if the primitive Dirichlet character $\chi = {\bf{1}}$ is principal and $\pi \cong \widetilde{\pi}$ is self-dual, 
then $\widetilde{L}_{\infty}(\frac{1}{2})/L_{\infty}(\frac{1}{2})=1$. The sum $(\ref{nvres})$ is then the same as 
\begin{align*}\frac{2}{w_K} \cdot L(1, \omega \eta) \cdot \frac{L_{\bf{1}}^{\star}(1, \operatorname{Sym}^2 \pi)}{ L_{\bf{1}}^{\star}(2, \omega) }. \end{align*}
We deduce single residual term is nonvanishing by a well-known argument (see e.g.~\cite[Lemma]{CM}) which establishes a lower bound
for the contribution of the (partial) symmetric square $L$-function. In general, $\widetilde{L}_{\infty}(\frac{1}{2})/L_{\infty}(\frac{1}{2}) \neq 1$, 
and if we assume otherwise that $(\ref{nvres})$ vanishes identically, then we would have to have  
\begin{align}\label{RAA} L(1, \omega \eta \chi^2 \circ {\bf{N}}) 
\cdot \frac{ L_{\bf{1}}^{\star}(1, \operatorname{Sym}^2 \pi \otimes \chi \circ {\bf{N}}) }{ L_{\bf{1}}^{\star}(2, \omega \chi^2 \circ {\bf{N}})}  
&= \left( -\epsilon \cdot \frac{\widetilde{L}_{\infty}(\frac{1}{2})}{L_{\infty}(\frac{1}{2})} \right) \cdot L(1, \overline{\omega} \eta \overline{\chi}^2 \circ {\bf{N}}) 
\cdot \frac{ L_{\bf{1}}^{\star}(1, \operatorname{Sym}^2 \widetilde{\pi} \otimes \overline{\chi} \circ {\bf{N}}) }
{ L_{\bf{1}}^{\star}(2, \overline{\omega} \overline{\chi}^2 \circ {\bf{N}})}, \end{align} 
i.e.~where the $-\epsilon \cdot \widetilde{L}_{\infty}(\frac{1}{2})/L_{\infty}(\frac{1}{2})$ 
term is a constant independent of the symmetric square $L$-values in question. 
To rule out this possibility $(\ref{RAA})$, we argue as follows. Let us for for any Dirichlet character $\chi  \bmod p^{\beta}$ write
\begin{align*} \mathfrak{L}_{\chi} (1) &= L(1, \omega \eta \chi^2 \circ {\bf{N}}) 
\cdot \frac{ L_{\bf{1}}^{\star}(1, \operatorname{Sym}^2 \pi \otimes \chi \circ {\bf{N}}) }{ L_{\bf{1}}^{\star}(2, \omega \chi^2 \circ {\bf{N}})},  \quad 
\mathfrak{L}_{\overline{\chi}}(1) = L(1, \overline{\omega} \eta \overline{\chi}^2 \circ {\bf{N}}) \cdot 
\frac{ L_{\bf{1}}^{\star}(1, \operatorname{Sym}^2 \widetilde{\pi} \otimes \overline{\chi} \circ {\bf{N}}) }
{ L_{\bf{1}}^{\star}(2, \overline{\omega} \overline{\chi}^2 \circ {\bf{N}})}, \end{align*} 
and $\epsilon' = \epsilon \cdot \widetilde{L}_{\infty}(\frac{1}{2})/L_{\infty}(\frac{1}{2})$ to 
simplify notations\footnote{Note that since we assume $\pi \cong \widetilde{\pi}$ is self-dual, 
we again have that $\widetilde{L}_{\infty}(1/2) = L_{\infty}(1/2)$, and hence can ignore this quotient of archimedean factors.
However, we include these harmless extra factors to reveal the structure of the argument in the general setting, e.g.~as it may be useful elsewhere.}. 
Observe that $(\ref{RAA})$ implies we have the relation
\begin{align}\label{RAA2} \frac{L_{\infty}(\frac{1}{2})}{\widetilde{L}_{\infty}(\frac{1}{2})}
\cdot \frac{ \mathfrak{L}_{\chi}(1)}{ \mathfrak{L}_{\overline{\chi}}(1) } &= - \epsilon(\chi), \end{align} 
where
\begin{align*} \epsilon(\chi) &:= \omega(\operatorname{lcm}(\mathfrak{p}^{\alpha}, p^{\beta} \mathcal{O}_F)) \cdot \eta(p^{4 \beta} \mathfrak{d}c(\pi)) \cdot 
\epsilon(1/2, \pi) \cdot \chi( {\bf{N}}(\mathfrak{d}^2 c(\pi)^2 \mathfrak{D}^8)) 
\cdot \left( \frac{\tau(\chi^2)}{p^{\frac{\beta}{2}}} \right)^{4d} \in \overline{\bf{Q}} \end{align*}
denotes the root number $\epsilon(1/2, \pi \times \rho \chi \circ {\bf{N}})$.
Since this root number $\epsilon(\chi)$ determines an algebraic number, 
and moreover factors through\footnote{In general, it factors through the compositum ${\bf{Q}}(\pi, \chi) = {\bf{Q}}(\pi) {\bf{Q}}(\pi)$
of the Hecke field ${\bf{Q}}(\pi)$ of $\pi$ and the cyclotomic field ${\bf{Q}}(\chi)$. The same argument is then
made via automorphisms of ${\bf{Q}}(\pi, \chi)$ fixing the Hecke field ${\bf{Q}}(\pi)$.} the 
cyclotomic field ${\bf{Q}}(\chi)$ obtained by adjoining the values of $\chi$ (since $\pi \cong \widetilde{\pi}$), 
we can act on the algebraic values in $(\ref{RAA2})$. This gives us for all $\sigma \in \operatorname{Aut}({\bf{Q}}(\chi)/{\bf{Q}})$ the relation  
\begin{align*}\frac{L_{\infty}(\frac{1}{2})}{\widetilde{L}_{\infty}(\frac{1}{2})}
\cdot \frac{ \mathfrak{L}_{\chi^{\sigma}}(1)}{ \mathfrak{L}_{\overline{\chi}^{\sigma}}(1) } 
&= - \epsilon(\chi^{\sigma}) \quad \forall ~~\sigma \in  \operatorname{Aut}({\bf{Q}}(\chi)/{\bf{Q}}). \end{align*} 
In particular, the relation $(\ref{RAA})$ would have to hold for all of the $\varphi(p^{\beta}) - \varphi(p^{\beta-1})$ 
many characters in this orbit corresponding to Dirichlet characters $\chi$ of exact order $p^{\beta}$, 
which is to say for {\it{each}} primitive Dirichlet character $\chi \bmod p^{\beta}$. But, this relation clearly fails
for the (primitive) principal character $\chi = {\bf{1}} \bmod p^{\beta}$, giving us the desired contradiction. 
That is, we deduce that the sum of residual terms $(\ref{nvres})$ for the primitive average is nonvanishing in this way. 

To show (ii), we argue in the same way as for (i), with Theorem \ref{SCS2} replacing Theorem \ref{SCS}.
To be more precise, we consider for each nonzero $F$-integer $b \in \mathcal{O}_F$ the quadratic polynomial 
\begin{align*} q_{A, b}(x) := f_A(x, b) = a_A x^2 + b_A b x + c_A b^2 \end{align*}
with discriminant $\Delta_b = (b_A b)^2 - 4 a_A c_A b^2 = b^2(b_A^2 - 4 a_A c_A) = b^2 \Delta$, 
i.e.~where we write $\Delta = \operatorname{disc}(f_A) = \mathfrak{p}^{2 \alpha} \mathfrak{D}$. 
We then have for any smooth function $V$ with support on $[1/2, 1]$ satisfying $V^{(i)} \ll 1$ for all $i \geq 1$
and any real parameter $R \geq 1$ that 
\begin{equation*}\begin{aligned} &\left\vert \sum\limits_{\mathfrak{m} \subset \mathcal{O}_F} 
\frac{  \eta \omega(\mathfrak{m}) \chi^2({\bf{N}} \mathfrak{m})}{ {\bf{N}}\mathfrak{m} } 
\sum\limits_{a, b \in \mathcal{O}_F/ \mathcal{O}_F^{\times} \atop b \neq 0} 
\frac{\lambda_{\chi}(q_{A, b}(a))}{ {\bf{N}} q_{A, b}(a)^{\frac{1}{2}}}
V \left( \frac{ {\bf{N}}\mathfrak{m}^2 {\bf{N}}q_{A, b}(a) }{R} \right) \right\vert \\
&\ll  \sum\limits_{\mathfrak{m} \subset \mathcal{O}_F} \frac{1}{ {\bf{N}}\mathfrak{m} } 
\sum\limits_{b \in \mathcal{O}_F/\mathcal{O}_F^{\times}}
\left\vert \sum\limits_{a \in \mathcal{O}_F/\mathcal{O}_F^{\times} }
\frac{\lambda_{\chi}(q_{A, b}(a))}{ {\bf{N}} q_{A, b}(a)^{\frac{1}{2}}}
V \left( \frac{ {\bf{N}}\mathfrak{m}^2 {\bf{N}}q_{A, b}(a) }{R} \right) \right\vert. \end{aligned} \end{equation*}
Applying Theorem \ref{SCS2} to each of the inner $a$-sums, we then obtain for any $\varepsilon >0$ the bound
\begin{equation}\begin{aligned}\label{SCS2bounds}
&\ll_{\pi, \varepsilon} \sum\limits_{ b \neq 0 \in \mathcal{O}_F/\mathcal{O}_F^{\times} \atop {\bf{N}}b \leq \left( \frac{R}{ {\bf{N}}c_A } \right)^{\frac{1}{2}}  }
{\bf{N}} a_A \cdot R^{- \frac{1}{4} + \frac{\theta_0}{2} + \varepsilon} \cdot {\bf{N}}\Delta_b^{\delta_0 - \frac{\theta_0}{2} - \varepsilon} \\
&\ll {\bf{N}} a_A \cdot R^{- \frac{1}{4} + \frac{\theta_0}{2} + \varepsilon} \cdot {\bf{N}} \Delta^{\delta_0 - \frac{\theta_0}{2} - \varepsilon}
\sum\limits_{ {\bf{N}}b \leq \left( \frac{R}{ {\bf{N}}c_A } \right)^{\frac{1}{2}} } {\bf{N}} b^{2 \left( \delta_0 - \frac{\theta_0}{2} - \varepsilon \right)} \\
&\ll {\bf{N}} a_A \cdot R^{- \frac{1}{4} + \frac{\theta_0}{2} + \varepsilon} \cdot {\bf{N}} \Delta^{\delta_0 - \frac{\theta_0}{2} - \varepsilon}
\cdot \left( \frac{ R  }{ {\bf{N}}c_A } \right)^{\frac{1}{2} + \delta_0 - \frac{\theta_0}{2} - \varepsilon} 
= {\bf{N}} a_A \cdot R^{ \frac{1}{4} + \delta_0} \cdot {\bf{N}}(\mathfrak{p}^{2 \alpha} \mathfrak{D})^{\delta_0 - \frac{\theta_0}{2} - \varepsilon} 
\cdot {\bf{N}}c_A^{- \frac{1}{2} - \delta_0 + \frac{\theta_0}{2} + \varepsilon}. \end{aligned}\end{equation}
The claimed bound again follows after taking a standard partition of unity and dyadic decomposition for the 
corresponding off-diagonal sum, taking a sum over $\log Y$ many ranges $R \geq 1$ of these bounds $(\ref{SCS2bounds})$. \end{proof} 

\begin{corollary}\label{RCGAnv} 

Suppose the generic root number $\epsilon$ is not $-1$ in the special case 
where $\pi \otimes \xi = \pi$ is self-dual (hence with $\xi = \chi \circ {\bf{N}}$ trivial). We have the following estimates: \\

\noindent (i) Assume $\pi \otimes \xi$ is non-dihedral (i.e. $\pi$ is non-dihedral if $\xi$ is trivial). 
Taking $Y = {\bf{N}}(\mathfrak{D}c(\pi_K)^{\frac{1}{2}}) p^{2d \max(\alpha, \beta)}$ to be the square root of the conductor, 
we have have the following estimate for the average $P_{\alpha}(\pi, \chi)$ 
over primitive ring class characters of conductor $\mathfrak{p}^{\alpha}$:
\begin{equation}\begin{aligned}\label{primitive1} &P_{\alpha}(\pi, \chi) = 
\left(1 -2 \cdot \frac{\#C(\alpha-1)}{\#C^{\star}(\alpha)} \right) \\ &\times \frac{1}{w_K} 
\left( L(1, \omega \eta \chi^2 \circ {\bf{N}}) \cdot \frac{ L_{\bf{1}}^{\star}(1, \operatorname{Sym}^2 \pi \otimes \chi \circ {\bf{N}}) }{ L_{\bf{1}}^{\star}(2, \omega \chi^2 \circ {\bf{N}})} 
+ \epsilon \cdot \frac{\widetilde{L}_{\infty}(\frac{1}{2})}{L_{\infty}(\frac{1}{2})} \cdot L(1, \overline{\omega} \eta \overline{\chi}^2 \circ {\bf{N}}) 
\cdot \frac{ L_{\bf{1}}^{\star}(1, \operatorname{Sym}^2 \widetilde{\pi} \otimes \overline{\chi} \circ {\bf{N}}) }{ L_{\bf{1}}^{\star}(2, \overline{\omega} \overline{\chi}^2 \circ {\bf{N}})} \right) \\
&+ O_{\pi, \chi, \varepsilon} \left( Y^{\frac{1}{4} + \delta_0 + \varepsilon} {\bf{N}}(\mathfrak{D} \mathfrak{p}^{2 \alpha})^{-\frac{1}{2}} \right). \end{aligned}\end{equation}
Hence (by the discussion above for $(\ref{nvres})$), the primitive average $P_{\alpha}(\pi, \chi)$ converges to a nonzero constant as $\alpha \rightarrow \infty$
so long as we assume $\pi \cong \widetilde{\pi}$ is self-dual. Consequently,
for each sufficiently large $\alpha \geq 1$, there exists a primitive ring class character $\rho$ of conductor $\mathfrak{p}^{\alpha}$ for which
$L(1/2, \pi \times \rho \chi \circ {\bf{N}}) \neq 0$. \\

\noindent (ii) Assume again that $\pi \otimes \xi$ is non-dihedral. Then, the subaverage $P_{\alpha, \rho}(\pi, \chi)$ over 
primitive characters $\rho \in C(\alpha)$ with restriction to the torsion subgroup $C_0 = C(\infty)_{\operatorname{tors}}$
given by some $\rho_0$ can be estimated as follows. Taking $Y = {\bf{N}}(\mathfrak{D}c(\pi_K)^{\frac{1}{2}}) p^{2d \max(\alpha, \beta)}$ 
again to be the square root of the conductor, and using the residues notation defined in $(\ref{residue})$, we have 
we have for any $\varepsilon >0$ that 
\begin{align*} &P_{\alpha, \rho_0}(\pi, \chi) \\ = 
&\sum_{A \in C_0(\alpha)} \rho_0(A) \left( R(\pi, \chi, a_A) 
+ \epsilon \cdot \frac{\widetilde{L}_{\infty}(\frac{1}{2})}{ L_{\infty}(\frac{1}{2}) } \cdot R(\widetilde{\pi}, \overline{\chi}, a_A) 
+ O_{\pi, \chi, \varepsilon} \left( {\bf{N}} a_A \cdot Y^{\frac{1}{4} + \delta_0} \cdot 
\frac{ {\bf{N}}(\mathfrak{p}^{2 \alpha} \mathfrak{D})^{\delta_0 - \frac{\theta_0}{2} - \varepsilon}}{ {\bf{N}}(c_A)^{ \frac{1}{2} + \delta_0 - \frac{\theta_0}{2} - \varepsilon} } \right) \right) \\
&- \frac{\# \overline{C}(\alpha-1)}{\# P(\alpha, \rho_0)} \sum_{A \in C_0(\alpha-1) \atop A \notin C_0(\alpha)} \rho_0(A) \left( R(\pi, \chi, a_A) 
+ \epsilon \cdot \frac{\widetilde{L}_{\infty}(\frac{1}{2})}{ L_{\infty}(\frac{1}{2}) } \cdot R(\widetilde{\pi}, \overline{\chi}, a_A) 
+  O_{\pi, \chi, \varepsilon} \left( {\bf{N}} a_A \cdot Y^{ \frac{1}{4} + \delta_0} \cdot 
\frac{ {\bf{N}}(\mathfrak{p}^{2 \alpha} \mathfrak{D})^{\delta_0 - \frac{\theta_0}{2} - \varepsilon}}{ {\bf{N}}(c_A)^{ \frac{1}{2} + \delta_0 - \frac{\theta_0}{2} - \varepsilon} } \right) \right). \end{align*}
In particular, we argue that the average converges with $\alpha \rightarrow \infty$ to the constant 
\begin{equation}\begin{aligned}\label{tameconstant}
&\sum_{A \in C_0(\alpha)} \rho_0(A) \left( R(\pi, \chi, a_A) 
+ \epsilon \cdot \frac{\widetilde{L}_{\infty}(\frac{1}{2})}{ L_{\infty}(\frac{1}{2}) } \cdot R(\widetilde{\pi}, \overline{\chi}, a_A) \right) \\
&- \frac{\# \overline{C}(\alpha-1)}{\# P(\alpha, \rho_0)} \sum_{A \in C_0(\alpha-1) \atop A \notin C_0(\alpha)} \rho_0(A) \left( R(\pi, \chi, a_A) 
+ \epsilon \cdot \frac{\widetilde{L}_{\infty}(\frac{1}{2})}{ L_{\infty}(\frac{1}{2}) } \cdot R(\widetilde{\pi}, \overline{\chi}, a_A) \right), \end{aligned}\end{equation}
and moreover that this constant does not vanish if $\pi \cong \widetilde{\pi}$ is assumed to be self-dual. \\

\noindent (iii) Assume again that $\pi \otimes \xi$ is non-dihedral. 
Taking $Y = {\bf{N}}(\mathfrak{D}c(\pi_K)^{\frac{1}{2}}) p^{2d \max(\alpha, \beta)}$ again to be the square root of the conductor,
the Galois subaverage $G_{\alpha}(\pi, \chi; x)$ can be estimated in a similar way for any $\varepsilon >0$ as 
\begin{align*} &G_{\alpha}(\pi, \chi; x) = \sum_{A \in C(\alpha)^{p^x}} \left( R(\pi, \chi, a_A) 
+ \epsilon \cdot \frac{\widetilde{L}_{\infty}(\frac{1}{2})}{ L_{\infty}(\frac{1}{2}) } \cdot R(\widetilde{\pi}, \overline{\chi}, a_A) 
 O_{\pi, \chi, \varepsilon} \left( {\bf{N}} a_A \cdot Y^{ \frac{1}{4} + \delta_0} \cdot 
\frac{ {\bf{N}}(\mathfrak{p}^{2 \alpha} \mathfrak{D})^{\delta_0 - \frac{\theta_0}{2} - \varepsilon}}{ {\bf{N}}(c_A)^{ \frac{1}{2} + \delta_0 - \frac{\theta_0}{2} - \varepsilon} }  \right) \right) \\
&- \frac{\# \overline{C}(\alpha, x-1)}{\# C^{\star}(\alpha, x)} \sum_{A \in C(\alpha)^{p^{x-1}} \atop A \notin C(\alpha)^{p^x}} \left( R(\pi, \chi, a_A) 
+ \epsilon \cdot \frac{\widetilde{L}_{\infty}(\frac{1}{2})}{ L_{\infty}(\frac{1}{2}) } \cdot R(\widetilde{\pi}, \overline{\chi}, a_A) 
 O_{\pi, \chi, \varepsilon} \left( {\bf{N}} a_A \cdot Y^{\frac{1}{4} + \delta_0} \cdot 
\frac{ {\bf{N}}(\mathfrak{p}^{2 \alpha} \mathfrak{D})^{\delta_0 - \frac{\theta_0}{2} - \varepsilon}}{ {\bf{N}}(c_A)^{ \frac{1}{2} + \delta_0 - \frac{\theta_0}{2} - \varepsilon} } \right) \right). \end{align*}
In particular, if ${\bf{N}} c_A \gg {\bf{N}} a_A$ is sufficiently large relative to the discriminant ${\bf{N}}(\mathfrak{p}^{2 \alpha} \mathfrak{D})$ for each class $A$ in the sum, 
then this Galois average converges with the ring class exponent $\alpha \rightarrow \infty$ to the constant 
\begin{equation*}\begin{aligned}
&\sum_{A \in C(\alpha)^{p^x}} \left( R(\pi, \chi, a_A) 
+ \epsilon \cdot \frac{\widetilde{L}_{\infty}(\frac{1}{2})}{ L_{\infty}(\frac{1}{2}) } \cdot R(\widetilde{\pi}, \overline{\chi}, a_A) \right) \\
&- \frac{\# \overline{C}(\alpha, x-1)}{\# C^{\star}(\alpha, x)} 
\sum_{A \in C(\alpha)^{p^{x-1}} \atop A \notin C(\alpha)^{p^x}} \left( R(\pi, \chi, a_A) 
+ \epsilon \cdot \frac{\widetilde{L}_{\infty}(\frac{1}{2})}{ L_{\infty}(\frac{1}{2}) } \cdot R(\widetilde{\pi}, \overline{\chi}, a_A) \right), \end{aligned}\end{equation*}
and this constant does not vanish if $\pi \cong \widetilde{\pi}$ is assumed to be self-dual. 

\end{corollary}

\begin{proof} 

Taking for granted Lemma \ref{RCGA} and the residual estimates from Proposition \ref{RCres}, 
which also apply in a natural way to the sums $D_{{\bf{1}}, j}^{\star}(\pi, \chi; Z)$, 
the first claim (i) follows directly from Theorem \ref{RCOD} (i) if we assume $D_K \equiv 0 \bmod 4$, 
and more generally from Theorem \ref{RCOD} if we do not make this assumption.

To deduce the stated estimate for (ii), we put together the estimates of Proposition \ref{RCres} and Theorem \ref{RCOD} (ii)
in our average formula. We then argue that for sufficiently large ring class exponent $\alpha \gg 1$, the first coefficients $a_A$
of each contributing class $A \in C_0(\alpha) \cong C_0(\alpha-1) \cong C_0$ will be bounded independently of $\alpha$. 
Since we take the chosen quadratic form representative $f_A(x, y) = a_A x^2 + b_A xy + c_A y^2$ to be reduced, 
hence with ${\bf{N}} b_A \leq {\bf{N}} a_A \leq {\bf{N}} c_A$, we deduce that the middle coefficient $b_A$ is also 
bounded independently of the ring class exponent $\alpha$. Here, we can also use the unconditional approximation
$\delta_0 = 103/512$ of Blomr-Harcos \cite{BH10} (via the approximation $\theta_0 = 7/64$ of Blomer-Brumley \cite{BBRP}).
In this way, we deduce that the average converges with $\alpha \rightarrow \infty$
to the constant term $(\ref{tameconstant})$. To derive the claimed nonvanishing of this term, let us write the residue as 
\begin{align*} R(\pi, \chi, a_A) &= \sum_{q \mid a_A} \mathcal{L}_{\chi, q}(1), \quad \mathcal{L}_{\chi, q}(1) = 
\left( \frac{\mu(q) \lambda^{(q)}(a_A) \omega(q) \chi ({\bf{N}}a_A ) }{ {\bf{N}}a_A^{\frac{1}{2}} } \right) 
\cdot \frac{  L^{\star}_{{\bf{1}}, q}(1, \operatorname{Sym}^2 \pi^{(q)} \otimes \chi \circ {\bf{N}})}{ L^{\star}_{{\bf{1}}}(2, \omega \chi^2 \circ {\bf{N}}) }. \end{align*}
We argue by inspection of the Dirichlet series defining each $L^{\star}_{{\bf{1}}, q}(1, \operatorname{Sym}^2 \pi^{(q)} \otimes \chi \circ {\bf{N}})$ in relation to the full series 
$L^{\star}_{{\bf{1}}, 1}(1, \operatorname{Sym}^2 \pi^{(q)} \otimes \chi \circ {\bf{N}}) = L^{\star}_{{\bf{1}}}(1, \operatorname{Sym}^2 \pi^{(q)} \otimes \chi \circ {\bf{N}})$
that each of the summands $\mathcal{L}_{\chi}(1)$ is nonvanishing. A more intrinsic way to see this is that the standard contour arguments 
used to derive the corresponding statement for $L(1, \operatorname{Sym^2} \pi \otimes \chi)$, 
as given for instance in \cite[Lemma 4.2, cf.~Lemma 4.1]{CM}, can be applied directly to each of these partial Dirichlet series. 
We can also deduce by the argument given in Theorem \ref{RCOD} (i) that for each divisor $q \mid a_A$, the corresponding sum of residual terms
\begin{align*} \mathfrak{L}_{\chi, q}(1) &:= \mathcal{L}_{\chi, q}(1) 
+ \epsilon \cdot \frac{\widetilde{L}_{\infty}(\frac{1}{2})}{L_{\infty}(\frac{1}{2})} \cdot \mathcal{L}_{\overline{\chi}, q}
=  \mathcal{L}_{\chi, q}(1) + \epsilon \cdot \mathcal{L}_{\overline{\chi}, q}(1) \end{align*} does not vanish. It is then
easy to check that the sum $(\ref{tameconstant})$ is nonvanishing. This is simple to deduce what $\rho_0 = {\bf{1}}$
is the trivial character. Otherwise, using orthogonality, we pick up the first nonvanishing term in the first sum from $q \nmid \gcd(a_A)$
not dividing the mutual greatest common divisor of the coefficients $a_A$. 

The third claim (iii) is deduced in the same way as for (ii) from the average formula. 
\end{proof}

Let us conclude with a few remarks on the limitations of the method we use here to derive bounds, which is drawn out in Appendix A below.
In short, to derive bounds for the off-diagonal sums $D_{A, j}^{\dagger}(\pi, \chi; Z)$ we consider  
via spectral decompositions of shifted convolution sums, we need to derive integral presentations in terms
of Fourier coefficients of some distinct automorphic form. This imposes some constraints on the coefficients of the
reduced binary quadratic form class representative $f_A(x, y) = a_A x^2 + b_A xy  + c_A y^2$. 
In particular, it requires the first coefficient $a_A$ to be small relative to the last coefficient $c_A$. 
It is for this reason that we do not derive an unconditional nonvanishing estimate for the Galois subaverage $G_{\alpha}(\pi, \chi, x)$ 
directly\footnote{although we can use rationality theorems to deduce this from the previous estimate (ii)} -- 
we do not know the relative sizes of the coefficients $a_A$ and $c_A$ for the classes that contribute, 
and this remains and interesting open problem to consider. 
In general, without constraints on the coefficients of $f_A(x,y)$, we can prove the following result, 
in fact a special adaptation of the proof of Theorem \ref{SCS} for the sums we consider above.
However, since the choice of archimedean local vector in the Kirillov model does not seem to be admissible without introducing a smooth partition of unity
and dyadic decomposition, it cannot be used in a direct way to bound the sums we consider here suitably\footnote{This is because we need to consider
each length $R$ as in the proof of Theorem \ref{RCOD} (i), so have to include the contributions of small lengths $R$ of size close to one in the sum. 
At the same time, although we do not describe it here (see rather \cite{VO12}), it seems a slightly better bound can be derived in the classical setup over $F={\bf{Q}}$,
essentially as we can detect more cancellation thanks to a more explicit knowledge of the constant coefficients of the Eisenstein series 
appearing in the spectral decomposition.} -- a fact which is consistent with the cutoff functions $V_j$ having poles near zero. 

\begin{proposition}\label{binary} 

Let $\pi$ as above be any cuspidal automophic representation of $\operatorname{GL}_2({\bf{A}}_F)$ 
with unitary central character $\omega$. Let $\chi$ be any primitive Dirichlet character of conductor $p^{\beta}$ with 
$\xi = \chi \circ {\bf{N}}$ the corresponding idele class character of $F$, 
and $A \in C(\alpha)$ and ring class of conductor $\mathfrak{p}^{\alpha}$ with associated
reduced quadratic form class group representative $f_A(x, y) = a_A x^2 + b_A xy + c_A y^2$.
Again, we write $C =  {\bf{N}}(\mathfrak{D} ^2 c(\pi_K)) p^{4d \max(\alpha, \beta)} $ for simplicity to denote the conductor of the $L$-functions in the average. 
Let us take $V \in \mathcal{C}^{\infty}({\bf{R}}_{>0})$ to be any smooth function of rapid decay at infinity with bounded derivatives $V^{(i)} \ll 1$ for all $i \geq 0$.
More specifically, we assume that the function $\exp(2 \pi y) V(y)$ of $y \in {\bf{R}}_{>0}$ is square integrable, 
which will always be the case e.g.~if $V$ is compactly supported. Then for any choice of real number $Y >1$, the sum 
\begin{align*} \frac{1}{w_K} \sum_{\mathfrak{m} \subset \mathcal{O}_F } 
\frac{\overline{\omega} \eta(\mathfrak{m}) \overline{\chi}^2({\bf{N}} \mathfrak{m}) }{ {\bf{N}}\mathfrak{m}^2  }
\sum_{a, b \in \mathcal{O}_F / \mathcal{O}_F^{\times} \atop f_A(a, b) \neq 0} 
\frac{ \overline{ \lambda_{\pi \otimes \xi}(f_A(a, b))} }{ {\bf{N}}f_A(a, b)^{\frac{1}{2}}} 
V \left( \frac{ {\bf{N}} f_A(a, b)}{ Y } \right) \end{align*}
as well as the corresponding contragredient sum can be bounded above in modulus by the quantity 
\begin{align*} \ll_{\pi, \chi, \varepsilon} C^{\frac{1}{2} + \varepsilon} \cdot (D_K p^{d \beta}) \cdot Y^{-\frac{1}{2}}. \end{align*}

\end{proposition}

\begin{proof} Ignoring $\mathfrak{m}$-sums for simplicity, this can be deduced from the following argument, 
writing $V = V_j$ for $j=1, 2$ to denote the relevant cutoff function. 
We have by orthogonality of characters $\rho \in C(\alpha)^{\vee}$ that 
\begin{align*} \sum_{a, b \in \mathcal{O}_F/ \mathcal{O}_F^{\times} } \frac{\lambda_{\chi}(f_A(a,b))}{{\bf{N}}f_A(a, b)^{\frac{1}{2}}  }
V\left( \frac{ {\bf{N}}f_A(a, b)}{Y} \right) &= \frac{1}{\# C(\alpha)} \sum_{\rho \in C(\alpha)^{\vee}} \sum_{A \in C(\alpha)} \rho(A)
\sum_{\mathfrak{n} \subset \mathcal{O}_F} \frac{\lambda_{\chi}(\mathfrak{n}) r_A(\mathfrak{n})}{{\bf{N}} \mathfrak{n}^{\frac{1}{2}}} 
V \left( \frac{ {\bf{N}} \mathfrak{n}  }{Y} \right). \end{align*}
Estimating the $\rho$-sum trivially, this latter sum is bounded above in modulus by 
\begin{align*} \sum_{A \in C(\alpha)} \rho(A)
\sum_{\mathfrak{n} \subset \mathcal{O}_F} \frac{\lambda_{\chi}(\mathfrak{n}) r_A(\mathfrak{n})}{{\bf{N}} \mathfrak{n}^{\frac{1}{2}}} 
V \left( \frac{{\bf{N}}\mathfrak{n} }{Y} \right), \end{align*}
which after Mellin inversion is the same as 
\begin{align*} \int_{\Re(s) = 2} L(s+1/2, \pi \times \rho) \widehat{V}(s) \frac{ds}{2 \pi i}. \end{align*}
Shifting the contour leftward to $\Re(s) = -1/2$ and estimating the contribution of $L(s+1/2, \pi \times \rho)$
trivially by $C^{\frac{1}{2}}$ on the line $\Re(s) = 0$ recovers essentially the same bound. Let us also remark
that a similar bound can be obtained by replacing the metaplectic theta series $\theta_Q$ in the proof of Theorem \ref{SCS}
below with the binary theta series $\theta_{f_A}$ associated to the fixed quadratic form representative $f_A(x,y)$, i.e.~then decomposing 
the constant coefficient in the Fourier-Whittaker expansion of $\Phi_A = \phi \overline{\theta}_{f_A}$ for some suitable choice of pure tensor 
$\phi = \otimes_v \phi_v \in V_{\pi}$ spectrally in terms of the constant coefficient of $\operatorname{GL}_2({\bf{A}}_F)$ 
Eisenstein series\footnote{We omit the details of this alternative argument for brevity.}. \end{proof}
 
\subsection{Galois conjugate values} 

Assume now that $\pi_{\infty}$ is a holomorphic discrete series of weight $k = (k_j)_{j=1}^d$ with each $k_j \geq 2$, 
so that $\pi$ arises from a holomorphic cuspidal Hilbert modular eigenform of ``arithmetic weight $k \geq 2$",
The Hecke eigenvalues $\lambda(\mathfrak{n}) = \lambda_{\pi}(\mathfrak{n})$ are then known by a theorem of Shimura \cite{Sh} 
to be algebraic numbers. Writing $\langle \pi, \pi \rangle$ to denote the Petersson norm of $\pi$, 
another theorem of Shimura shows \cite{Sh} shows essentially that the values 
\begin{align*} \mathcal{L}(1/2, \pi \times \mathcal{W}) = \frac{L(1/2, \pi \times \mathcal{W})}{8 \pi^2 \langle \pi, \pi \rangle} \end{align*}
are algebraic numbers, and moreover acted upon in a natural way by automorphisms $\sigma \in \operatorname{Gal}(\overline{ {\bf{Q}} }/  {\bf{Q}})$. 
More precisely, writing ${\bf{Q}}(\pi)$ to denote the finite extension of ${\bf{Q}}$ obtained by 
adjoining the eigenvalues of $\pi$, and ${\bf{Q}}(\pi, \mathcal{W})$ the finite extension of ${\bf{Q}}(\pi)$ obtained by adjoining the 
values of $\mathcal{W}$, the values $\mathcal{L}(1/2, \pi \times \mathcal{W})$ lie in ${\bf{Q}}(\pi, \mathcal{W})$. 
These algebraic values are Galois conjugate in the sense that $\sigma \in \operatorname{Aut}({\bf{C}})$ 
acts on them via the rule $\sigma \left( \mathcal{L}(1/2, \pi \times \mathcal{W}) \right) = \mathcal{L}(1/2, \pi^{\sigma} \times \mathcal{W}^{\sigma})$. 
Here, $\pi^{\sigma}$ denotes the representation of $\operatorname{GL}_2({\bf{A}}_F)$ 
obtained from $\pi$ by applying $\sigma$ to its eigenvalues, and $\mathcal{W}^{\sigma}$ the character
defined on nonzero ideals $\mathfrak{a} \subset \mathcal{O}_K$ by the rule 
$\mathfrak{a} \mapsto \mathcal{W}(\mathfrak{a})^{\sigma}$. Restricting to embeddings $\sigma$ of ${\bf{Q}}(\pi, \mathcal{W})$ into ${\bf{C}}$ 
which fix ${\bf{Q}}(\pi)$, we obtain a Galois conjugate family of values $\mathcal{L}(1/2, \pi \times \mathcal{W})$, 
where the action fixes $\pi$ but varies over Galois conjugate characters $\mathcal{W}$. When ${\bf{Q}}(\pi)$ is linearly disjoint over
${\bf{Q}}$ to the cyclotomic extension ${\bf{Q}}(\mathcal{W})$ obtained by adjoining the values of $\mathcal{W}$, 
then the (well-defined) weighted average 
\begin{align*} G_{[\mathcal{W}]}(\pi) &:= \frac{1}{[{\bf{Q}}(\pi, \mathcal{W}): {\bf{Q}}(\pi)]} \sum_{ \sigma: {\bf{Q}}(\pi, \mathcal{W}) \rightarrow 
{\bf{C}} \atop \sigma( {\bf{Q}}(\pi)) = {\bf{Q}}(\pi)} L(1/2, \pi \times \mathcal{W}^{\sigma})\end{align*}
over all complex embeddings $\sigma: {\bf{Q}}(\pi, \mathcal{W}) \rightarrow {\bf{C}}$ which fix ${\bf{Q}}(\pi)$ consists of Galois conjugate 
values. It is easy to see from this that the sum defining $G_{[\mathcal{W}]}(\pi)$ vanishes if and only if each of the summands vanishes. 
This allows us to deduce the following direct consequences of Corollary \ref{RCGAnv} above. 

\begin{theorem}\label{GAnv}

Assume $\pi$ is a holomorphic discrete series of weight $(k_j)_{j=1}^d$ with each $k_j \geq 2$ and conductor 
$c(\pi) \subset \mathcal{O}_F$. Fix a prime ideal $\mathfrak{p} \subset \mathcal{O}_F$ with underlying rational prime $p$, 
together with a totally imaginary quadratic extension $K/F$ of relative discriminant $\mathfrak{D} \subset \mathcal{O}_F$. 
Assume that $(c(\pi), \mathfrak{D}\mathfrak{p}) = (\mathfrak{p}, \mathfrak{D}) = 1$, 
and that ${\bf{Q}}(\pi)$ is linearly disjoint over ${\bf{Q}}$ to the cyclotomic tower 
obtained by adjoining all $p$-power roots of unity ${\bf{Q}}(\zeta_{p^{\infty}}) = \bigcup_{n \geq 1} {\bf{Q}}(\zeta_{p^n})$. 
Let us also assume that $\pi \cong \widetilde{\pi}$ is self-dual. \\

Fix a primitive even Dirichlet character $\chi \bmod p^{\beta}$ for some integer $\beta \geq 0$. 
In the event that $\beta =0$ (hence $\chi$ trivial), let us also assume that the generic root number $\epsilon(1/2, \pi \times \rho)$ for $\rho$ 
ranging over primitive ring class characters characters of conductor $\mathfrak{p}^{\alpha}$ with $\alpha \gg 1$ sufficiently large 
is not equal to $-1$. Then for each sufficiently large integer $\beta \geq 1$, there exists a primitive ring class character $\rho$ 
of conductor $\mathfrak{p}^{\alpha}$ for which the Galois average $G_{[\rho \chi \circ {\bf{N}}]}(\pi)$ does not vanish, 
and so $L(1/2, \pi \times \mathcal{W}) = L(1/2, \pi \times \rho \chi \circ {\bf{N}})$ does not vanish for $\mathcal{W} = \rho \chi \circ {\bf{N}}$ 
ranging over such Hecke characters taking values in roots of unity of exact order $\operatorname{lcm}(p^{\beta}, \operatorname{ord}(\rho))$,
i.e.~where $\operatorname{ord}(\rho) \mid (\#C(\alpha) - \#C (\alpha-1))$ denotes the exact order of the character $\rho$. \end{theorem}

\begin{proof} 

The claim follows from Corollary \ref{RCGAnv}, i.e.~after using Shimura's algebraicity theorem suitably.  \end{proof}

\section{$p$-adic $L$-functions}

Let us assume from now on that $\pi$ is a holomorphic discrete series at each of weight $(k_j)_{j=1}^d$ with $k_j \geq 2$. 
We explain in this setting how to derive stronger nonvanishing results from the existence of a suitable $p$-adic $L$-function. 
This can be viewed as a more efficient way of using the algebraicity theorem of Shimura \cite{Sh} (as done in Theorem \ref{GAnv} above) 
together with congruences to derive stronger results from the nonvanishing of a single character twist (as supplied by Corollary \ref{RCGAnv} above). 
To ensure the existence of such a $p$-adic $L$-function, we shall assume for simplicity that $\pi$ is $\mathfrak{p}$-ordinary at our fixed prime 
$\mathfrak{p} \subset F$, i.e.~that the image of the Hecke eigenvalue $\lambda(\mathfrak{p}) = \lambda_{\pi}(\mathfrak{p})$ under our fixed embedding 
$\overline{\bf{Q}} \rightarrow \overline{\bf{Q}}_p$ is a $p$-adic unit. 

\subsection{Some background} Let us first establish some background notions required to prove our main result.

\subsubsection{Iwasawa algebras}

Let $\mathcal{O}$ be a finite extension of ${\bf{Z}}_p$, $\mathcal{G}$ a profinite group, and 
$\mathcal{O}[[\mathcal{G}]] = \varprojlim_{\mathcal{U} \subset \mathcal{G}} \mathcal{O}[\mathcal{G}/\mathcal{U}]$
its $\mathcal{O}$-Iwasawa algebra. The limit here runs over open normal subgroups $\mathcal{U} \subset \mathcal{G}$.
Note that if $\mathcal{G}$ is abelian and finitely generated, then the elements $\mathcal{L}$ of $\mathcal{O}[[\mathcal{G}]]$ 
can be viewed as $\mathcal{O}$-valued measures $d \mathcal{L}$ on $\mathcal{G}$. 

\subsubsection{Choice of profinite group $\mathcal{G}$}

Let us henceforth consider the profinite group 
\begin{align*} \mathcal{G} &= \varprojlim_{\alpha, \beta} C(\mathcal{O}_{\mathfrak{p}^{\alpha}}) \times ({\bf{Z}} / p^{\beta} {\bf{Z}})^{\times}.\end{align*}
Note that composing with the Artin reciprocity map $\operatorname{rec}_K$ gives us for each integer $\alpha \geq 0$ an identification 
\begin{align*} \operatorname{rec}_K: C(\mathcal{O}_{\mathfrak{p}^{\alpha}}) \cong \varOmega_{\alpha} :
= \operatorname{Gal}(K[\mathfrak{p}^{\alpha}]/K),\end{align*}
where $K[\mathfrak{p}^{\alpha}]$ denotes the ring class extension of conductor $\mathfrak{p}^{\alpha}$ of $K$. 
The torsion subgroup $\Omega_0 = \varOmega_{\operatorname{tors}}$ of 
$\varOmega := \varprojlim_{\alpha} \varOmega_{\alpha}$ is a finite group (see e.g.~\cite[$\S 2$]{CV}), 
and the quotient $\Omega$ of $\varOmega$ by $\Omega_0$ is isomorphic as a topological group to ${\bf{Z}}_p^{\delta}$, 
where $\delta = \delta_{\mathfrak{p}} = [F_{\mathfrak{p}}: {\bf{Q}}_p]$ is the residue degree of $\mathfrak{p}$. 
On the other hand, let us for each integer $\beta \geq 0$ write $\varGamma_{\beta}$ to denote the Galois group 
$\operatorname{Gal}(K(\zeta_{p^{\beta}})/K)$, where $K(\zeta_{p^{\beta}})$ is the extension obtained from $K$ by 
adjoining a primitive $p^{\beta}$-th root of unity $\zeta_{p^{\beta}}$. The corresponding limit
$\varGamma = \varprojlim_{\beta} \varGamma_{\beta}$ is isomorphic as a topological group to ${\bf{Z}}_p^{\times}$,
and hence its torsion subgroup $\Gamma_0 = \varGamma_{\operatorname{tors}}$ is also finite. 
Let us write $\Gamma$ to denote the quotient of $\varGamma$ by $\Gamma_0$, so that $\Gamma$ is isomorphic as topological group to ${\bf{Z}}_p$. 
Writing $K_{\infty}^{(\mathfrak{p})} = \bigcup_{\alpha, \beta \geq 0} K[\mathfrak{p}^{\alpha}] K(\zeta_{p^{\beta}})$ 
to denote the compositum of the extensions $K[\mathfrak{p}^{\alpha}]$ and $K(\zeta_{p^{\beta}})$, the reciprocity map gives an identification
\begin{align*} \operatorname{rec}_K: \mathcal{G} &\longrightarrow \operatorname{Gal}(K_{\infty}^{(\mathfrak{p})}/K) = \varOmega \times \varGamma 
= \varprojlim_{\alpha, \beta \geq 0} \varOmega_{\alpha} \times \varGamma_{\beta}. \end{align*}  

Let us now write $G \approx {\bf{Z}}_p^{\delta + 1}$ to denote the quotient of $\mathcal{G}$ modulo its finite torsion subgroup 
$G_0 = \mathcal{G}_{\operatorname{tors}}$. We can then describe our choice of $\mathcal{G}$ in terms of the following tower of Galois extensions: \\

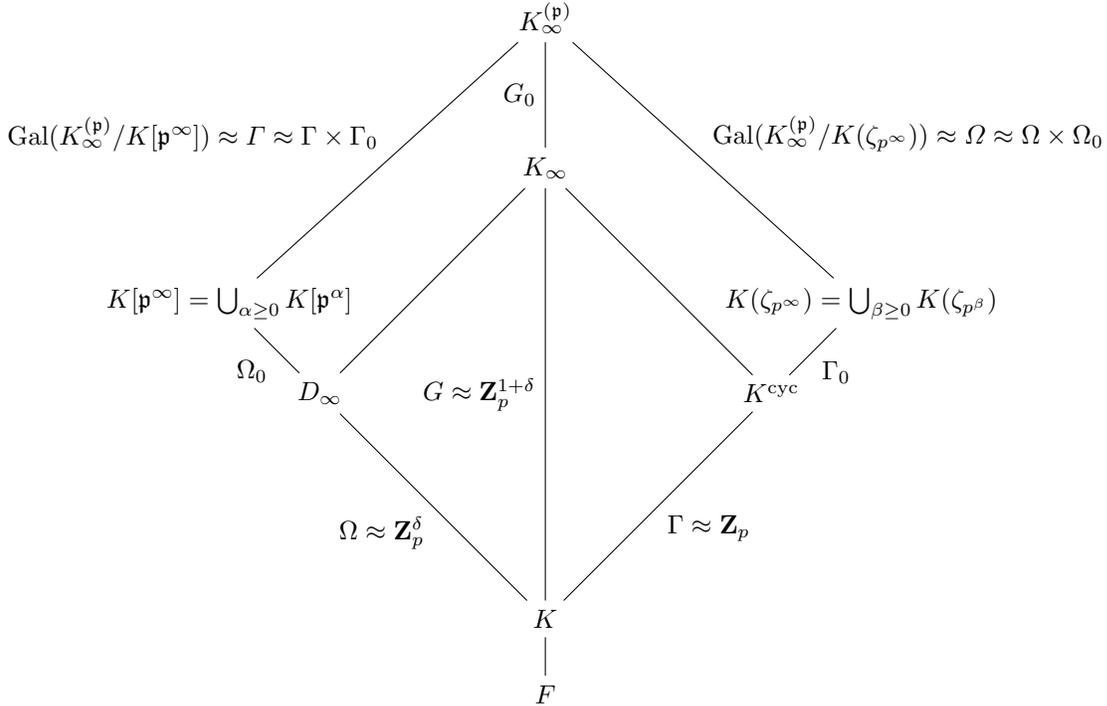
\begin{center} \scalebox{0.77}{\begin{tikzpicture}[scale=0.7, node distance = 2.2cm, auto]

      \node (B) {$F$};
      \node (K) [above of = B, node distance = 1 cm]{$K$};
      \node (D) [above of=K, left of=K, node distance = 3 cm] {$D_{\infty}$};
      \node (C) [above of=K, right of=K, node distance = 3 cm] {$K^{\operatorname{cyc}}$};
      \node (R) [above of=K, node distance = 6 cm] {$K_{\infty}$};
      \node (A) [above of=D, left of=D, node distance = 1.2 cm]{$K[\mathfrak{p}^{\infty}] = \bigcup_{\alpha \geq 0} K[\mathfrak{p}^{\alpha}]$};
      \node (Z) [above of=C, right of=C, node distance= 1.2 cm]{$K(\zeta_{p^{\infty}}) = \bigcup_{\beta \geq 0} K(\zeta_{p^{\beta}})$};
      \node (P) [above of=R]{$K_{\infty}^{(\mathfrak{p})}$};
    
      \draw[-] (B) to node {} (K); 
      \draw[-] (K) to node {$\Omega \approx {\bf{Z}}_p^{\delta}$} (D);
      \draw[-] (D) to node {$\Omega_0$} (A);
      \draw[-] (A) to node {$\operatorname{Gal}(K_{\infty}^{(\mathfrak{p})}/K[\mathfrak{p}^{\infty}]) \approx \varGamma \approx \Gamma \times \Gamma_0$} (P);
      \draw[-] (K) to node [swap]{$\Gamma \approx {\bf{Z}}_p$} (C);
      \draw[-] (C) to node [swap]{$\Gamma_0$} (Z);
      \draw[-] (Z)  to node [swap]{$\operatorname{Gal}(K_{\infty}^{(\mathfrak{p})}/K(\zeta_{p^{\infty}})) \approx \varOmega \approx \Omega \times \Omega_0$} (P);
      \draw[-] (D) to node {} (R);
      \draw[-] (C) to node {} (R);
      \draw[-] (R) to node {$G_0$} (P);
      \draw[-] (K) to node {$G \approx {\bf{Z}}_p^{1 + \delta}$} (R);
      
      \end{tikzpicture}}\end{center} 
    
We consider the corresponding $\mathcal{O}$-Iwasawa algebra $\mathcal{O}[[\mathcal{G}]] \approx \mathcal{O}[G_0][[G]]$. 
Here, we have an injection
\begin{align}\label{icgr} \mathcal{O}[[\mathcal{G}]] &\longrightarrow \bigoplus_{\mathcal{W}_0 \in G_0^{\vee}} \mathcal{O} [[G]], 
~~~~~  \lambda \longmapsto (\mathcal{W}_0(\lambda))_{\mathcal{W}_0 \in G_0^{\vee}}, \end{align} 
where the sum runs over characters $\mathcal{W}_0 = \rho_0 \psi_0 = \rho_0 (\chi_0 \circ {\bf{N}})$ of the torsion subgroup $G_0$, 
and each $\mathcal{W}_0(\lambda)$ denotes the specialization of the given element $\lambda \in \mathcal{O}[\mathcal{G}]]$ to the 
character $\mathcal{W}_0$ of $G_0$, but not to any character of $G \approx \Omega \times \Gamma \approx {\bf{Z}}_p^{\delta + 1}$. 
Thus, each $\mathcal{W}_0(\lambda)$ denotes a genuine element of the Iwasawa algebra $\mathcal{O}[[G]]$ rather than simply a value 
in $\mathcal{O}$. Note that when $G_0$ has order prime to $p$, this injection $(\ref{icgr})$ is also a surjection. 

\subsubsection{Relation to formal power series}

Taking $G = \operatorname{Gal}(K_{\infty}/K) \approx \Omega \times \Gamma \cong {\bf{Z}}_p^{\delta + 1}$ as above, we view the corresponding 
$\mathcal{O}$-Iwasawa algebra $\mathcal{O}[[G]]$ as a multivariable power series ring in the following standard way. 
Let $r \geq 2$ denote the integer defined by $r = \delta + 1 = [F_{\mathfrak{p}}: {\bf{Q}}_p]+1$. 
Fixing a system of topological generators $\gamma_1, \ldots \gamma_{\delta}$ of $\Omega$ and $\gamma_r$ of $\Gamma$,
we have an isomorphism to the formal power series ring $\mathcal{O}[[T_1, \ldots, T_r]]$ in $r$ indeterminates $T_1, \ldots, T_r$ given by 
\begin{align}\label{nci} \mathcal{O}[[G]] \approx \mathcal{O}[[\Omega \times \Gamma]] 
&\longrightarrow \mathcal{O}[[T_1, \ldots, T_r]], ~~~~~(\gamma_1, \ldots, \gamma_r) \longmapsto (T_1 +1, \ldots, T_r +1). \end{align}
 
\subsubsection{Weierstrass preparation theorem}

Let $R$ be any complete local ring (e.g.~$R = \mathcal{O}[[T]]$) with maximal ideal $\mathfrak{m}_R$, and fix a uniformizer $\varpi_R$ of $R$. 
Consider the formal power series ring $R[[T]]$ in the indeterminate $T$. Recall that a polynomial $g(T)$ in $R[T]$ is said to be {\it{distinguished (or Weierstrass)}} 
if it takes the form \begin{align*} g(T) &= T^n + b_{n-1} T^{n-1} + \ldots + b_0 \end{align*} for some integer $n \geq 1$, with each coefficient $b_i$ lying in the maximal 
ideal $\mathfrak{m}_R$.

\begin{proposition}[Weierstrass preparation theorem]\label{WPT} 

Let $h(T) = \sum_{j \geq 0} a_j T^j$ be an element of the formal power series ring $R[[T]]$. If $h(T)$ is not identically zero, 
then it can be expressed uniquely as the product 
\begin{align*} h(T) &= u(T) g(T) \varpi_R^{\mu_R}\end{align*} of some unit $u(T)$ in $R[[T]]$ times some distinguished polynomial 
$g(T)$ in $R[T]$ times some integer power $\mu_R \geq 0$ of the fixed uniformizer $\varpi_R$ of $R$.\end{proposition}

\begin{proof} The result is standard, see e.g. \cite[Ch. IV, Theorem 9.2]{La}. \end{proof}

Given a nonzero element $h(T) = u(T)g(T) \varpi_R^{\mu_R}$ of a formal power series  ring $R[[T]]$ as above, the degree of the distinguished polynomial 
$g(T)$ is known as the {\it{Weierstrass degree of $h(T)$,}} and the positive integer $\mu = \mu_R$ as the {\it{$\mu$-invariant}}. Note that this Weierstrass
degree can also be characterized as the least integer $j \geq 0$ for which the coefficient $a_j$ in the power series expansion 
$h(T) = \sum_{j \geq 0} a_j T^j \in R[[T]]$ is a unit in $R$. 

\subsubsection{Multivariable $p$-adic $L$-functions} 

Fix $\mathcal{O}$ a finite extension of ${\bf{Z}}_p$ containing the Hecke eigenvalues of $\pi$. 
Recall that for $\mathcal{W}$ a Hecke character of $K$, a well-known theorem of Shimura \cite{Sh} shows that the values 
\begin{align}\label{val} \mathcal{L}(1/2, \pi \times \mathcal{W}) &= \frac{L(1/2, \pi \times \mathcal{W})}{ 8 \pi^2 \langle \pi, \pi \rangle} \end{align} 
are algebraic, and moreover that they lie in the finite extension of ${\bf{Q}}$ defined by the compositum of the Hecke field ${\bf{Q}}(\pi)$ with the 
cyclotomic extension ${\bf{Q}}(\mathcal{W})$ of ${\bf{Q}}$ obtained by adjoining the values of the character $\mathcal{W}$.
Via our fixed embedding $\overline{\bf{Q}} \rightarrow \overline{\bf{Q}}_p$, we view the values $(\ref{val})$ as element in $\overline{\bf{Q}}_p$. 
We have the following construction, stated here in a simplified form\footnote{i.e.~without spelling out the interpolation formula} that will suffice for our subsequent arguments. 

\begin{theorem}\label{mvplfn} 

Suppose $\pi$ is associated to a holomorphic $\mathfrak{p}$-ordinary Hilbert modular form of arithmetic weight $k \geq 2$. 
There exists a measure $\mathcal{L}_{\mathfrak{p}} = \mathcal{L}_{\mathfrak{p}}(\pi) \in \mathcal{O}[[\mathcal{G}]]$ 
such that for any nontrivial finite-order character $\mathcal{W}$ of $\mathcal{G}$, we have the interpolation formula
\begin{align}\label{interpolation} \mathcal{W}(\mathcal{L}_{\mathfrak{p}}) 
&=\eta(\pi, \mathcal{W}) \cdot \mathcal{L}(1/2, \pi \times \overline{\mathcal{W}}) \in \overline{\bf{Q}}_p. \end{align} 
Here, $\eta(\pi, \mathcal{W})$ is some algebraic number which does not vanish so long as $\mathfrak{p}$ does not 
divide the conductor of $\pi$, and  $\mathcal{W}(\mathcal{L}_{\mathfrak{p}}) = \int_{\mathcal{G}} \mathcal{W}(\sigma) d \mathcal{L}_{\mathfrak{p}}(\sigma)$ 
denotes the specialization of the measure $\mathcal{L}_{\mathfrak{p}}$ to the character $\mathcal{W}$. \end{theorem}

\begin{proof} See e.g.~\cite{Har}, or the constructions of \cite[$\S 2$]{Pa}, \cite{Da}, \cite{Hi}, and \cite[$\S 5.1$]{PR88}. 
Each of these constructions shows the existence of such an element; again we suppress the exact form of $\eta(\pi, \mathcal{W})$ for simplicity. \end{proof}

\subsection{Power series expansions}

Let $\mathcal{L}_{\mathfrak{p}} \in \mathcal{O}[[\mathcal{G}]]$ be the $p$-adic $L$-function of Theorem \ref{mvplfn} above.
Fix a character $\mathcal{W}_0$ of the finite torsion subgroup $G_0$, and consider the corresponding partially-specialized $p$-adic $L$-function
$\mathcal{W}_0(\mathcal{L}_{\mathfrak{p}}) \in \mathcal{O}[[G]]$. 
Let us then write $\mathcal{L}_{\mathfrak{p}}(\mathcal{W}_0; T_1, \cdots, T_r) \in \mathcal{O}[[T_1, \ldots, T_r]]$ to
denote the image of $\mathcal{W}_0(\mathcal{L}_{\mathfrak{p}})$ under the non-canonical isomorphism $(\ref{nci})$. Recall that we label the 
indeterminates here so that $T_1, \ldots, T_{\delta}$ denote the anticyclotomic variables corresponding to a fixed system $\gamma_1, \ldots, \gamma_{\delta}$ 
of generators of the anticyclotomic Galois group $\Omega \approx {\bf{Z}}_p^{\delta}$, and $T_r = T_{\delta}$ denotes the cyclotomic variable
corresponding to a fixed generator $\gamma_r = \gamma_{\delta +1}$ of the cyclotomic Galois group $\Gamma \approx {\bf{Z}}_p$. 
We now consider expansions of the $p$-adic $L$-function $\mathcal{L}_{\mathfrak{p}}(\mathcal{W}_0; T_1, \ldots, T_r)$ is each $T_j$.
Let $\mathfrak{P}$ denote the maximal ideal of the (complete local) ring $\mathcal{O}$.

\subsubsection{Expansion in the cyclotomic variable} 

Let us now expand the $p$-adic $L$-function $\mathcal{L}_{\mathfrak{p}}(\mathcal{W}_0; T_1, \ldots, T_r)$ in the cyclotomic variable $T_r$. 
Hence, writing $k \in \lbrace 0, 1 \rbrace$ to denote the integer for which the anticyclotomic root number $\epsilon(1/2, \pi \times \rho) = (-1)^k$ 
for all almost all characters $\rho$ of $\varOmega$ (cf~\cite[$\S 1$]{CV}), we have  
\begin{align}\label{Cexp} \mathcal{L}_{\mathfrak{p}}(\mathcal{W}_0; T_1, \cdots, T_r) &= \sum_{j \geq k} a_j(T_1, \ldots, T_{\delta}) T_r^j 
\in \mathcal{O}[[T_1, \cdots, T_{\delta}]][[T_r]]. \end{align}

\begin{proposition}[Least nonvanishing criterion via the cyclotomic variable]\label{lnvcC} 

Let $\mathcal{W}_0$ be any character of the finite torsion subgroup $G_0 = \mathcal{G}_{\operatorname{tors}}$.
Assume that $\operatorname{ord}_{\mathfrak{P}}(\mathcal{L}_{\mathfrak{p}}(\mathcal{W}_0; 0, \ldots, 0, T_r)) = 0$. 
Let us also assume that for some character $\mathcal{W} = \mathcal{W}_0 \mathcal{W}_w = \mathcal{W}_0 \rho_w \psi_w$ factoring through 
$\mathcal{G} \approx G_0 \times \Omega \times \Gamma$, we know that $L(1/2, \pi \times \mathcal{W}) \neq 0$, or equivalently that
$\mathcal{W}(\mathcal{L}_{\mathfrak{p}}) 
= \mathcal{L}_{\mathfrak{p}}(\mathcal{W}_0; \rho_w(\gamma_1)-1, \ldots, \rho_w(\gamma_{\delta})-1, \psi_w(\gamma_r)-1) \neq 0$. 
Then, there exists a minimal exponent $\beta_0 \geq 0$ such that for all characters $\psi_w$ of the cyclotomic Galois group $\Gamma$ of exact 
order $p^{\beta}$ with $\beta \geq \beta_0$, the central value $L(1/2, \pi \times \mathcal{W}_0 \rho_w \psi_w)$ does not vanish for any character $\rho_w$ 
of the anticyclotomic Galois group $\Omega$. \end{proposition}

\begin{proof} See \cite[Part II, Proposition 3.1 (i)]{VO12}. Using Proposition \ref{WPT}, 
we deduce from the nonvanishing hypothesis that $\mathcal{L}_{\mathfrak{p}}(\mathcal{W}_0, T_1, \ldots, T_r)$ has a finite Weierstrass degree $w(T_r)$ in $T_r$.
Since $\operatorname{ord}_{\mathfrak{P}}(\mathcal{L}_{\mathfrak{p}}(\mathcal{W}_0; \cdots, 0, T_r)) = 0$, this latter fact has the following consequence 
for the power series expansion $(\ref{Cexp})$: There exists a least integer $j_0 \geq k$ such that $a_{j_0}(T_1, \cdots T_{\delta})$ is a unit in 
$\mathcal{O}[[T_1, \cdots, T_{\delta}]]$. Since units never specialize to zero, we deduce that there exists a least integer $\beta_0 = \beta_0(w(T_r))$
(depending on the Weierstrass degree $w(T_r)$) such that for all characters $\psi_w$ of $\Gamma$ of exact order $p^{\beta}$ with 
$\beta \geq \beta_0$, we have $\mathcal{L}_{\mathfrak{p}}(\mathcal{W}_0; \rho_w(\gamma_1)-1, \cdots, \rho_w(\gamma_{\delta})-1, \psi_w(\gamma_r)-1) \neq 0$ 
for all characters $\rho_w$ of $\Omega$. This latter assertion is equivalent to the stated claim. \end{proof} 

\begin{corollary}\label{cycmu} 

Assume $\pi \cong \widetilde{\pi}$ is $\mathfrak{p}$-ordinary, and that $(c(\pi), \mathfrak{D}\mathfrak{p}) = (\mathfrak{p}, \mathfrak{D}) = 1$. 
Fix a character $\mathcal{W}_0$ of the torsion subgroup $G_0 = \mathcal{G}_{\operatorname{tors}}$, and assume that the cyclotomic analytic 
$\mu$-invariant of the corresponding $p$-adic $L$-function $\mathcal{W}_0(\mathcal{L}_{\mathfrak{p}}) \in \mathcal{O}[[G]]$ vanishes. 
Then, there exists a minimal exponent $\beta_0 \geq 0$ such that for all characters $\psi_w$ of the cyclotomic Galois group $\Gamma$ 
of exact order $p^{\beta}$ with $\beta \geq \beta_0$, the central value $L(1/2, \pi \times \mathcal{W}_0 \rho_w \psi_w)$ does not vanish 
for any character $\rho_w$ of the anticyclotomic Galois group $\Omega$.

\end{corollary} 

\begin{proof} Taking both $\alpha \gg \beta$ to be sufficiently large, 
we can use the results of Corollary \ref{RCGAnv} (ii) as input for the argument Proposition \ref{lnvcC} described above. 
To be clear, we can assume without loss of generality that the $\mathcal{W}_0 = \rho_0 \psi_0 = \rho_0 \chi_0 \circ {\bf{N}}$
is purely cyclotomic $\mathcal{W}_0 = \rho_0$, i.e.~after replacing the $\chi$ in our main theorem with $\chi \chi_0$ as we may. 
The desired input then follows from Corollary \ref{RCGAnv} (ii) for the weighted average over primitive characters $P(\alpha, \rho_0)$. \end{proof}

\subsubsection{Specialization in the anticyclotomic variables}

Let us now fix an index $1 \leq l \leq \delta$, and consider the expansion of 
$\mathcal{L}_{\mathfrak{p}}(\mathcal{W}_0; T_1, \ldots, T_r)$ in each anticyclotomic variable $T_l$: 
\begin{align} \mathcal{L}_{\mathfrak{p}}(\mathcal{W}_0; T_1, \ldots, T_r) 
&= \sum_{i \geq 0} b_i(T_1, \ldots, T_{l-1}, T_{l+1}, \cdots, T_{\delta}) T_l^i
\in \mathcal{O}[[T_1, \ldots, T_{l-1}, T_{l+1}, \cdots, T_{\delta}]] [[T_r]].  \end{align}

\begin{proposition}[Least nonvanishing criterion via the anticyclotomic variable]\label{lnvcA}

Let $\mathcal{W}_0$ be any character of the torsion subgroup $G_0 = \mathcal{G}_{\operatorname{tors}}$.
Assume that $\operatorname{ord}_{\mathfrak{P}}(\mathcal{L}_{\mathfrak{p}}(\mathcal{W}_0; 0, \ldots, 0, T_l, 0, \ldots, 0)) = 0$ 
for each index $1 \leq l \leq \delta$. Let us also assume that for some character 
$\mathcal{W} = \mathcal{W}_0 \mathcal{W}_w = \mathcal{W}_0 \rho_w \psi_w$ 
factoring through $\mathcal{G}$, we know that $L(1/2, \pi \times \mathcal{W}) \neq 0$, 
or equivalently that $\mathcal{W}(\mathcal{L}_{\mathfrak{p}}) 
= \mathcal{L}_{\mathfrak{p}}(\mathcal{W}_0; \rho_w(\gamma_1)-1, \ldots, \rho_w(\gamma_{\delta}-1), \psi_w(\gamma_r)-1) \neq 0$. 
Then, there exists a minimal exponent $\alpha_0 \geq 0$ such that for all characters $\rho_w$ of the anticyclotomic 
Galois group $\Omega$ of exact order $p^{\alpha}$ with $\alpha \geq \alpha_0$, the central value 
$L(1/2, \pi \times \mathcal{W}_0 \rho_w \psi_w)$ does not vanish for any character $\psi_w$ of the cyclotomic Galois group $\Gamma$. 

\end{proposition}

\begin{proof} We apply the same argument as given for Proposition \ref{lnvcC} in each anticylotomic indeterminate $T_l$. 
In this way, we deduce that for each index $1 \leq l \leq \delta$, there exists a minimal exponent $\alpha_0(l)$ such that 
$\mathcal{L}_{\mathfrak{p}}(\mathcal{W}_0; \rho_w(\gamma_1)-1, \cdots, \rho_w(\gamma_{\delta})-1, \psi_w(\gamma_r)-1) \neq 0$ for all 
characters $\rho_w$ of $\Omega$ of exact order $p^{\alpha}$ for $\alpha \geq \alpha_0(l)$ and all characters $\psi_w$ 
of $\Gamma$. Taking $\alpha_0 = \max_{1 \leq l \leq \delta} \alpha_0(l)$ then proves the claim. \end{proof}

\begin{corollary}\label{AC} 

Assume $\pi \cong \widetilde{\pi}$ is $\mathfrak{p}$-ordinary, 
and that $(c(\pi), \mathfrak{D}\mathfrak{p}) = (\mathfrak{p}, \mathfrak{D}) = 1$. Fix a character $\mathcal{W}_0$ of the torsion subgroup 
$G_0 = \mathcal{G}_{\operatorname{tors}}$, and assume that the anticyclotomic analytic $\mu$-invariant of the corresponding 
$p$-adic $L$-function $\mathcal{W}_0(\mathcal{L}_{\mathfrak{p}}) \in \mathcal{O}[[G]]$ vanishes. Then, there exists a minimal exponent 
$\alpha_0 \geq 0$ such that for all characters $\rho_w$ of the anticyclotomic Galois group $\Omega$ of exact order $p^{\alpha}$ with $\alpha \geq \alpha_0$, 
the central value $L(1/2, \pi \times \mathcal{W}_0 \rho_w \psi_w)$ does not vanish for any character $\psi_w$ of the Galois group $\Gamma$. \end{corollary}

\begin{proof} Again, we can use the result of Corollary \ref{RCGAnv} as input to deduce the claim. \end{proof}

Finally, we can deduce the following unconditional result in this direction thanks to \cite{HC}.

\begin{theorem}\label{ACmu} 

Assume that the cuspidal automorphic representation $\pi \cong \widetilde{\pi}$ as described above is $\mathfrak{p}$-ordinary, 
that $(c(\pi), \mathfrak{D}\mathfrak{p}) = (\mathfrak{p}, \mathfrak{D}) = 1$, and also that the residual Galois representation 
associated to $\pi$ by constructions of Carayol \cite{Ca}, Taylor \cite{Ta}, and Wiles \cite{Wi} is absolutely irreducible. 
Fix a character $\mathcal{W}_0$ of the torsion subgroup $G_0 = \mathcal{G}_{\operatorname{tors}}$. 
There exists a minimal exponent  $\alpha_0 \geq 0$ such that for all characters $\rho_w$ of the anticyclotomic Galois group 
$\Omega \approx {\bf{Z}}_p^{\delta}$ of exact order $p^{\alpha}$ with $\alpha \geq \alpha_0$, the central value 
$L(1/2, \pi \times \mathcal{W}_0 \rho_w \psi_w)$ does not vanish for any character $\psi_w$ of the cyclotomic Galois group $\Gamma \approx {\bf{Z}}_p$. \end{theorem}

\begin{proof} We use (as input for Corollary \ref{AC}) the result of Chida-Hsieh \cite[Theorem C]{HC}, which implies that 
\begin{align*} \operatorname{ord}_{\mathfrak{P}}\left( \mathcal{L}_{\mathfrak{p}}(\mathcal{W}_0; T_1, \ldots, T_{\delta}, 0) \right) &= 0 \end{align*}
under the stated hypotheses on the conductor $c(\pi)$ and the Galois representation associated to $\pi$. \end{proof}

\appendix

\section*{Appendices}
\renewcommand{\thesubsection}{\Alph{subsection}}

\section{Shifted convolution sums over totally real fields}

We now give a proof of Theorem \ref{SCS} from the body of the text. 
The idea is to use the surjectivity of the archimedean Kirillov map for $\pi$ to realize the shifted
convolution sum on the left hand side as the Fourier-Whittaker coefficient at $q$ of a certain genuine automorphic form $\Phi$ on the two-fold
metaplectic cover $\overline{G}({\bf{A}}_F)$ of $\operatorname{GL}_2({\bf{A}}_F)$ (see \cite{Ge}). Decomposing $\Phi$ spectrally, 
and using its convergence in the Sobolev norm topology, the stated bound can be derived from existing bounds for the Fourier-Whittaker 
coefficients of each form appearing in the decomposition. Let us now now explain this idea in more detail as follows, noting that the special case of 
$F = {\bf{Q}}$ is worked out in \cite[Theorem 1]{TT} (cf.~also \cite{BH10}). 

\subsubsection*{Fourier-Whittaker expansions}

Let $\psi = \otimes \psi_v$ denote the standard additive character on ${\bf{A}}_F/F$. Hence, $\psi$ is trivial on $F$, agrees with the function 
$x = (x_j)_{j=1}^d \mapsto \exp(2 \pi i(x_1 + \cdots x_d))$ on the archimedean component $F_{\infty} = F \otimes_{\bf{Q}} {\bf{R}}$ of ${\bf{A}}_F$, 
and each finite place $v$ is trivial on the local inverse different $\mathfrak{d}_{F, v}^{-1}$ but nontrivial on $v^{-1} \mathfrak{d}_{F, v}^{-1}$. 
Let $\phi \in V_{\pi}$ be any vector in the representation space of $\pi$. We have for $x \in {\bf{A}}_F$ any generic adele and $y \in {\bf{A}}_F^{\times}$ 
any generic idele the Fourier-Whittaker expansion 
\begin{align}\label{fwse1} \phi \left( \left(\begin{array} {cc} y &  x \\ ~ & 1 \end{array}\right) \right) &= \sum_{\gamma \in F^{\times}}
W_{\phi} \left( \left(\begin{array} {cc} \gamma &  ~ \\ ~ & 1 \end{array}\right) 
\left(\begin{array} {cc} y &  ~ \\ ~ & 1 \end{array}\right) \right) \psi(- \gamma x). \end{align}
Here, for any $g \in \operatorname{GL}_2({\bf{A}}_F)$, we write 
\begin{align*} W_{\phi}(g) &:= \int_{ {\bf{A}}_F/F } \phi \left( \left(\begin{array} {cc} 1 &  x \\ ~ & 1 \end{array}\right) g \right) \psi(-x) dx \end{align*}
to denote the Whittaker function of $\phi$. Note that if we decompose the idele $y \in {\bf{A}}_F^{\times}$ into its corresponding 
nonarchimedean and archimedean components as $y = y_f y_{\infty}$, with $y_f \in {\bf{A}}_{F, f}^{\times}$ and 
$y_{\infty} \in F_{\infty}^{\times} \cong ({\bf{R}}^{\times})^d$, then we can also decompose the (specialized) Whittaker 
function $W_{\phi}$ its corresponding nonarchimedean component 
\begin{align*} \rho_{\phi}(y_f) = \rho_{\phi} \left( \left(\begin{array} {cc} y_f &  ~ \\ ~ & 1 \end{array}\right)  \right) 
&:= W_{\phi} \left( \left(\begin{array} {cc} y_f &  ~ \\ ~ & 1 \end{array}\right) \right) \end{align*}  
and archimedean component 
\begin{align*} W_{\phi}(y_{\infty}) &:= W_{\phi} \left( \left(\begin{array} {cc} y_{\infty} &  ~ \\ ~ & 1 \end{array}\right) \right), \end{align*}
so that $(\ref{fwse1})$ is the same as 
\begin{align}\label{fwse2} \phi \left( \left(\begin{array} {cc} y &  x \\ ~ & 1 \end{array}\right) \right) 
&= \sum_{\gamma \in F^{\times}} \rho_{\phi}(\gamma y_f) W_{\phi} \left( \gamma y_{\infty} \right) \psi(- \gamma x). \end{align}
Note as well that if $\phi \in V_{\pi}$ is a new vector, and we write $\vert \cdot \vert$ to denote the idele norm, 
then the coefficients $\rho_{\phi}(\gamma y_f)$ are related to the $L$-function coefficients $\lambda(\gamma y_f) = \lambda_{\pi}(\gamma y_f)$ 
in the sense that the Fourier-Whittaker expansion $(\ref{fwse1})$ (or $(\ref{fwse2})$) is equivalent to  
\begin{align}\label{fwse3} \phi \left( \left(\begin{array} {cc} y &  x \\ ~ & 1 \end{array}\right) \right) 
&= \sum_{\gamma \in F^{\times}} \frac{\lambda(\gamma y_f)}{ \vert \gamma y_f \vert^{\frac{1}{2}}} W_{\phi}(\gamma y_{\infty}) \psi(\gamma x). \end{align}

In what follows, we shall always take $\phi \in V_{\pi}$ to be a pure tensor $\phi = \otimes_v \phi_v$ whose nonarchimedean components are each 
essential Whittaker vectors. The corresponding Whittaker coefficients $\rho_{\phi}$ are then related to the $L$-function coefficients 
$\lambda$ as in $(\ref{fwse3})$. That is, the local vectors $\phi_v$ are then related directly via Mellin transformation to the corresponding local Euler 
factors of $L(s, \pi)$. We shall also make a precise choice of the archimedean local vectors $\phi_{\infty} = \otimes_{v \mid \infty} \phi_v$ as follows. 
Namely, we shall use the surjectivity of the archimedean Kirillov map $\phi \mapsto W_{\phi} $, which as explained in \cite[$\S$ 2.5 (37)]{BH10} 
(for instance) induces an isometry between the representation space $V_{\pi}$ and the Whittaker model $\mathcal{W}(\pi)$ of $\pi$. In particular:

\begin{proposition}\label{choice} 

Let $W \in L^2(F_{\infty}^{\times}) \cong L^2(({\bf{R}}^{\times})^d)$ be any smooth\footnote{Strictly speaking, 
we should impose the condition that $W$ be compactly supported to match the statements of results in the literature.
In practice however, this condition is really only imposed to ensure the square summability of the function.} function on 
$F_{\infty}^{\times} \cong ({\bf{R}}^{\times})^d$. Let $\pi$ be any cuspidal automorphic representation of $\operatorname{GL}_2({\bf{A}}_F)$.
Then, there exists a vector $\phi \in V_{\phi}$ whose corresponding (archimedean) Whittaker function $W_{\phi}$ satisfies  
$W_{\phi}(y_{\infty}) = W(y_{\infty})$ as function(s) of $y_{\infty} \in F_{\infty}^{\times}$. \end{proposition}

\begin{proof} The result is relatively well-known; see \cite[(37), Lemma 3]{BH10} and \cite[Proposition 2.1]{TT} (for instance). \end{proof}

We use this result to choose our archimedean vector $\phi_{\infty} = \otimes_{v \mid \infty} \phi_v$ in such a way that the corresponding archimedean 
Whittaker coefficient $W_{\phi}(y_{\infty})$ in the expansion $(\ref{fwse3})$ matches the chosen test function that appears in the statement of Theorem \ref{SCS}. 
That is, we use Proposition \ref{choice} to derive an integral presentation for the shifted convolution sum appearing in Theorem \ref{SCS} as follows. 
Let us now consider the $F$-rational quadratic form defined on $\gamma \in F$ by $Q(\gamma) = \gamma^2$. Let $\theta_Q$ denote the corresponding 
half-integral weight theta series, viewed as an automorphic form on the metaplectic cover $\overline{G}({\bf{A}}_F)$ of $\operatorname{GL}_2({\bf{A}}_F)$. 
This theta series has the following expansion (at archimedean components): 
For $x_{\infty} \in F_{\infty} \cong {\bf{R}}^d$ and $y_{\infty} \in F_{\infty}^{\times} \cong ({\bf{R}}^d)^{\times}$, 
\begin{align*} \theta_Q \left( \left(\begin{array} {cc} y_{\infty} &  x_{\infty} \\ ~ & 1 \end{array}\right) \right) 
&= \vert y_{\infty} \vert^{\frac{1}{4}} \sum_{\gamma \in F} \psi(Q(\gamma) (x_{\infty} + iy_{\infty}))
= \vert y_{\infty} \vert^{\frac{1}{4}} \sum_{ q \in \mathcal{O}_F} \psi(Q(q)(x_{\infty}+ iy_{\infty})). \end{align*} 
Let us also write $\overline{\theta}_Q = T_{-1}\theta_Q$ to denote the image of $\theta_Q$ under the Hecke 
operator $T_{-1}$ corresponding to the classical Hecke operator sending $z \in \mathfrak{H} $ to $-\overline{z}$. 
The corresponding form $\overline{\theta}_Q$ then has the expansion 
\begin{align}\label{fwse4} \overline{\theta}_Q \left( \left(\begin{array} {cc} y_{\infty} &  x_{\infty} \\ ~ & 1 \end{array}\right) \right) 
&= \vert y_{\infty} \vert^{\frac{1}{4}} \sum_{\gamma \in F} \psi(- Q(\gamma) (x_{\infty} - iy_{\infty}))
= \vert y_{\infty} \vert^{\frac{1}{4}} \sum_{\gamma \in \mathcal{O}_F} \psi(- Q(\gamma)(x_{\infty}- iy_{\infty})). \end{align}
Fixing $\phi = \otimes_v \phi_v \in V_{\pi}$ a pure tensor as described above, 
we consider the product $\Phi = \phi \overline{\theta}_Q$. Note that this $\Phi$ is a genuine automorphic form 
on the metaplectic group $\overline{G}({\bf{A}}_F)$. Abstractly, it has the Fourier expansion 
\begin{align*} \Phi \left(  \left(\begin{array} {cc} y_{\infty} &  x_{\infty} \\ ~ & 1 \end{array}\right) \right) 
&= \sum_{\gamma \in F} W_{\Phi} \left(   \left(\begin{array} {cc} \gamma &  ~ \\ ~ & 1 \end{array}\right) 
\left(\begin{array} {cc} y_{\infty} &  ~ \\ ~ & 1 \end{array}\right) \right) \psi(\gamma x_{\infty}) 
= \sum_{ q \in \mathcal{O}_F} W_{\Phi} \left(   \left(\begin{array} {cc} q &  ~ \\ ~ & 1 \end{array}\right) 
\left(\begin{array} {cc} y_{\infty} &  ~ \\ ~ & 1 \end{array}\right) \right) \psi( q x_{\infty}), \end{align*}
where for each $F$-integer $\mathfrak{q}$, 
\begin{align*} W_{\Phi} \left(   \left(\begin{array} {cc}  q &  ~ \\ ~ & 1 \end{array}\right) 
\left(\begin{array} {cc} y_{\infty} &  ~ \\ ~ & 1 \end{array}\right) \right) 
&= \int_{ I \cong [0,1]^d \subset F_{\infty} } \Phi \left(  \left(\begin{array} {cc} 1 & x_{\infty} \\ ~ & 1 \end{array}\right) 
\left(\begin{array} {cc} y_{\infty} &  ~ \\ ~ & 1 \end{array}\right) \right) \psi(- q x_{\infty}) dx_{\infty}. \end{align*}

\begin{proposition}\label{fcpresentation} 

Let $W \in L^2(F_{\infty}^{\times}) \cong L^2(({\bf{R}}^{\times})^d)$ be any smooth function on 
$F_{\infty}^{\times} \cong ({\bf{R}}^{\times})^d$. Fix a nonzero $F$-integer $q \in \mathcal{O}_F$, 
as well as an archimedean idele $Y_{\infty} \in F_{\infty}^{\times}$ with idele norm $\vert Y_{\infty} \vert \gg \vert q \vert$.
Let $\pi$ be any cuspidal automorphic representation of $\operatorname{GL}_2({\bf{A}}_F)$.
Fix $\phi = \otimes_v \phi_v \in V_{\pi}$ a pure tensor whose nonarchimedean local components $\phi_v$ are each essential Whittaker vectors,
and whose archimedean local component $\phi_{\infty} = \otimes_{v \mid \infty} \phi_v$ is chosen in such a way that 
\begin{align*} W_{\phi}(y_{\infty}) &:= W_{\phi} \left( \left(\begin{array} {cc} y_{\infty} &  ~ \\ ~ & 1 \end{array}\right) \right)
= \psi(- i y_{\infty}) \psi \left( \frac{i q}{Y_{\infty}} \right) W(y_{\infty}) \end{align*}
as a function of $y_{\infty} \in F_{\infty}^{\times} \cong ({\bf{R}}^{\times})^d$. Then, the coefficient at $q$ in the expansion of 
\begin{align*} \Phi \left( \left(\begin{array} {cc} \frac{1}{Y_{\infty}} & ~ \\ ~ & 1 \end{array}\right) \right) 
= \phi \overline{\theta}_Q \left( \left(\begin{array} {cc} \frac{1}{Y_{\infty}} & ~ \\ ~ & 1 \end{array}\right) \right) \end{align*} 
is given by 
\begin{align*} \int_{ I \cong [0,1]^d \subset F_{\infty} } 
\Phi \left(   \left(\begin{array} {cc} \frac{1}{Y_{\infty}} & x_{\infty} \\ ~ & 1 \end{array}\right)   \right) \psi(- q x_{\infty}) dx_{\infty}
&= \vert Y_{\infty} \vert^{- \frac{1}{4}} \sum_{ \gamma \in F^{\times}} 
\frac{ \lambda(Q(\gamma) + q) }{ \vert Q(\gamma) + q \vert^{\frac{1}{2}}} 
W \left(  \frac{ Q(\gamma) + q }{Y_{\infty}}  \right). \end{align*}
Equivalently, we have the integral presentation 
\begin{align}\label{or} \vert Y_{\infty} \vert^{\frac{1}{4}} \int_{ I \cong [0,1]^d \subset F_{\infty} } \phi \overline{\theta}_Q \left( 
\left(\begin{array} {cc} \frac{1}{Y_{\infty}} & x_{\infty} \\ ~ & 1 \end{array}\right) \right) \psi(- q x_{\infty} ) dx_{\infty} 
&= \sum_{ \gamma \in F^{\times} } \frac{\lambda( \gamma^2 + q)}{\vert \gamma^2 + q \vert^{\frac{1}{2}} } 
W \left( \frac{ \gamma^2 + q }{ Y_{\infty} } \right). \end{align} 
Here, each of the sums is supported only on nonzero $F$-integers. \end{proposition} 

\begin{proof} Cf.~\cite[$\S 6.1$]{TT}. We use the expansions $(\ref{fwse3})$ and $(\ref{fwse4})$ to compute 
\begin{align*} & \int_{ I \cong [0,1]^d \subset F_{\infty} } \phi \left(   \left(\begin{array} {cc} \frac{1}{Y_{\infty}} & x_{\infty} \\ ~ & 1 \end{array}\right)   \right) 
\overline{\theta}_Q \left(   \left(\begin{array} {cc} \frac{1}{Y_{\infty}} & x_{\infty} \\ ~ & 1 \end{array}\right)   \right) \psi(- q x_{\infty}) dx_{\infty} 
\\
&= \int_{ I \cong [0,1]^d \subset F_{\infty} } 
\sum_{\gamma_1 \in F^{\times}} \frac{\lambda(\gamma_1)}{\vert \gamma_1 \vert^{\frac{1}{2}}} 
W_{\phi} \left( \frac{\gamma_1}{Y_{\infty}} \right) \psi(\gamma_1 x_{\infty}) 
\cdot \vert Y_{\infty} \vert^{-\frac{1}{4}} \sum_{\gamma_2 \in F} \psi \left( \frac{Ê i Q(\gamma_2)}{Y_{\infty}} \right) 
\psi(- Q(\gamma_2) x_{\infty}) \psi(- q x_{\infty}) dx_{\infty}  \\
&= \vert Y_{\infty} \vert^{-\frac{1}{4}} \sum_{\gamma_1 \in F^{\times}} \frac{\lambda(\gamma_1)}{\vert \gamma_1 \vert^{\frac{1}{2}}} 
W_{\phi} \left( \frac{\gamma_1}{Y_{\infty}} \right) \sum_{\gamma_2 \in F} \psi \left(  \frac{Êi Q(\gamma_2)}{Y_{\infty}} \right) 
\int_{ I \cong [0,1]^d \subset F_{\infty} } \psi(\gamma_1 x_{\infty} - Q(\gamma_2) x_{\infty} - q x_{\infty}) dx_{\infty}. \end{align*} 
To compute the integral, recall that the characters on the compact abelian group $I \cong [0,1]^d \cong ({\bf{R}}/{\bf{Z}})^d$ 
can be parametrized by $\psi(\gamma x_{\infty})$ for any fixed $x_{\infty} \in I$ with $\gamma \in F$ varying. 
We can then use the orthogonality of these characters to deduce that the integral vanishes unless $\gamma_1 - Q(\gamma_2) - q = 0$, so
\begin{align*} \int_{ I \cong [0,1]^d \subset F_{\infty} } 
\Phi \left(   \left(\begin{array} {cc} \frac{1}{Y_{\infty}} & x_{\infty} \\ ~ & 1 \end{array}\right)   \right) \psi(- q x_{\infty}) dx_{\infty} 
&= \vert Y_{\infty} \vert^{-\frac{1}{4}} \sum_{ \gamma \in F^{\times}} \frac{\lambda(Q(\gamma) + q)}{\vert Q(\gamma) + q \vert^{\frac{1}{2}}}
W_{\phi} \left( \frac{Q(\gamma) + q}{Y_{\infty}} \right) \psi \left( \frac{i Q(\gamma)}{Y_{\infty}} \right). \end{align*}
Now, by our choice of archimedean local vector $\phi_{\infty} = \otimes_{v \mid \infty} \phi_v$, this latter identity is the same as 
\begin{align*} \int_{ I \cong [0,1]^d \subset F_{\infty} } 
\Phi \left(   \left(\begin{array} {cc} \frac{1}{Y_{\infty}} & x_{\infty} \\ ~ & 1 \end{array}\right)   \right) \psi(- q x_{\infty}) dx_{\infty} 
&= \vert Y_{\infty} \vert^{-\frac{1}{4}} \sum_{ \gamma \in F^{\times}} \frac{\lambda(Q(\gamma) + q)}{\vert Q(\gamma) + q \vert^{\frac{1}{2}}}
W \left( \frac{Q(\gamma) + q}{Y_{\infty}} \right) \end{align*} or equivalently 
\begin{align*} \vert Y_{\infty} \vert^{\frac{1}{4}}  \int_{ I \cong [0,1]^d \subset F_{\infty} } \phi \overline{\theta}_Q 
\left(  \left(\begin{array} {cc} \frac{1}{Y_{\infty}} & x_{\infty} \\ ~ & 1 \end{array}\right)   \right) \psi(- q x_{\infty}) dx_{\infty} 
&= \sum_{ \gamma \in F^{\times}} \frac{\lambda( \gamma^2 + q)}{{\bf{N}}( \gamma^2 + q)^{\frac{1}{2}}} W \left( \frac{\gamma^2 + q}{Y_{\infty}} \right). \end{align*}
Note that these expansions are supported only on nonzero $F$-integers. The claimed presentation $(\ref{or})$ follows. \end{proof}

\subsubsection*{Upper bounds for Whittaker functions}

We shall use the following general bounds for classical Whittaker functions, as derived via contour integral arguments in \cite[$\S$ 7]{TT} 
for the case of $F = {\bf{Q}}$. These bounds can be applied componentwise to derive bounds in the more general setting we consider here. 
Let us for arbitrary complex numbers $\kappa, \nu \in {\bf{C}}$ consider the classical Whittaker function $W_{\kappa, \nu}(y)$ defined on a 
positive real variable $y \in {\bf{R}}_{>0}$ as in \cite[$\S$ 7.1]{TT}. To be clear about the choice of normalization of $W_{\kappa, \nu}$, 
we note that the Mellin transform of $e^{\frac{y}{2}} W_{\kappa, \nu}(y)$ at $s \in {\bf{C}}$ 
with $\Re(s) > \frac{1}{2} \pm \nu$ is known via direct calculation to equal 
\begin{align*} \int_0^{\infty} e^{\frac{y}{2}} W_{\kappa, \nu}(y) y^s \frac{dy}{y} 
&= \frac{\Gamma \left( \frac{1}{2} + s + \nu \right) \Gamma \left( \frac{1}{2} + s - \nu \right)}{\Gamma \left( 1 + s - \kappa \right)}. \end{align*}

\begin{proposition}\label{whittaker}

There exists a constant $A > 0$ for which we have the following uniform bounds in $y \in {\bf{R}}_{>0}$ with $y \rightarrow 0$, 
i.e.~for any choice of real parameter $y >0$ in the interval $0 < y < 1$:

\begin{itemize}

\item[(i)] If $\kappa, \nu \in {\bf{R}}$, then we have for any choice of $\varepsilon >0$ the upper bound 
\begin{align*} \frac{ W_{\kappa, \nu}(y)}{ \Gamma \left( \frac{1}{2} + \kappa + i \nu \right)  }  
&\ll_{\varepsilon} \left( \vert \kappa \vert + \vert \nu \vert  + 1 \right)^A y^{\frac{1}{2} - \varepsilon}. \end{align*}

\item[(ii)] If $\kappa, \nu \in {\bf{R}}$ with $0 < \nu < \frac{1}{2}$, then we have for any choice of $\varepsilon >0$ the upper bound 
\begin{align*} \frac{W_{\kappa, \nu}(y)}{\Gamma \left( \frac{1}{2} + \kappa \right)} 
&\ll_{\varepsilon} \left( \vert \kappa \vert + 1 \right)^A y^{\frac{1}{2} - \nu - \varepsilon}. \end{align*}

\item[(iii)] If $\kappa, \nu \in {\bf{R}}$ with $\kappa - \nu - \frac{1}{2} \in {\bf{Z}}_{\geq 0}$ 
and $\nu > - \frac{1}{2} + \varepsilon$, then we have the upper bound 
\begin{align*} \frac{W_{\kappa, \nu}(y)}{ \left\vert \Gamma \left( \frac{1}{2} + \kappa - \nu  \right) 
\Gamma \left( \frac{1}{2} + \kappa + \nu \right)  \right\vert^{\frac{1}{2}}  }
&\ll_{\varepsilon} \left( \vert \kappa \vert + \vert \nu \vert + 1 \right)^A y^{\frac{1}{2} - \varepsilon}. \end{align*}

\end{itemize} 

Here, $\vert \cdot \vert$ denotes the complex absolute value. 
Each of these bounds is uniform in the choice of weight $\kappa$ and spectral parameter $\nu$, 
with the implied constant depending only on the corresponding choice of $\varepsilon >0$.

\end{proposition}

\begin{proof} See \cite[Proposition 3.1]{TT}; the bounds are derived via contour integral presentations in \cite[$\S 7$]{TT}. \end{proof}

Let us now return to the setting we consider throughout this work, with $F$ a totally real number field of degree $d = [F:{\bf{Q}}]$.
Given $d$-tuples $\kappa = (\kappa_j)_{j=1}^d \in F_{\infty} \cong {\bf{R}}^d$ and $\nu = (\nu_j)_{j=1}^d \in F_{\infty} \cong {\bf{R}}^d$, 
we consider the Whittaker function $W_{\kappa, \nu}(y_{\infty})$ defined on 
$y_{\infty} = (y_{\infty, j})_{j=1}^d \in F_{\infty}^{\times} \cong ({\bf{R}}^{\times})^d $ in the natural way via the product 
\begin{align*} W_{\kappa, \nu}(y_{\infty}) &= \prod_{j=1}^d W_{\kappa_j, \nu_j}(\vert y_{\infty, j} \vert). \end{align*}
We can then derive the following immediate consequence from Proposition \ref{whittaker} for the setting 
of totally real fields we consider here, i.e.~where the notations are now altered accordingly to reflect this.

\begin{corollary}\label{totallyrealwhittaker}

There exists a constant $A > 0$ for which we have the following uniform bounds in the archimedean idele variable
$y_{\infty} = (y_{\infty, j})_{j=1}^d \in F_{\infty}^{\times} \cong ({\bf{R}}^{\times})^d$ with $0 < \vert y_{\infty, j} \vert < 1$ for each index $1 \leq j \leq d$.

\begin{itemize}

\item[(i)] For all $\kappa = (\kappa_j)_{j=1}^d, \nu = (\nu_j)_{j=1}^d \in F_{\infty} \cong {\bf{R}}^d$, we have for any $\varepsilon >0$ the bound 
\begin{align*} \frac{ W_{\kappa, \nu}(y_{\infty})}{ \prod_{j=1}^d \Gamma \left( \frac{1}{2} + \kappa_j + i \nu_j \right)  }  
&\ll_{\varepsilon} \vert y_{\infty} \vert^{ \frac{1}{2} - \varepsilon} \prod_{j=1}^d \left( \vert \kappa_j \vert + \vert \nu_j \vert  + 1 \right)^A, \end{align*}
where $\vert y_{\infty} \vert = \prod_{j=1}^d \vert y_{\infty, j} \vert$ denotes the idele norm of $y_{\infty} = (y_{\infty, j})_{j=1}^d$. \\

\item[(ii)] For all $\kappa = (\kappa_j)_{j=1}^d, \nu = (\nu_j)_{j=1}^d \in F_{\infty} \cong {\bf{R}}^d$ 
with $0 < \nu_j  < \frac{1}{2}$ for each index $1 \leq j \leq d$, we have for any choice $\varepsilon >0$ the bound 
\begin{align*} \frac{W_{\kappa, \nu}(y)}{ \prod_{j=1}^d \Gamma \left( \frac{1}{2} + \kappa_j \right)} 
&\ll_{\varepsilon} \prod_{j=1}^d \left( \vert \kappa_j \vert + 1 \right)^A \vert y_j \vert^{\frac{1}{2} - \nu_j - \varepsilon}.  \end{align*}

\item[(iii)] If $\kappa = (\kappa_j)_{j=1}^d, \nu = (\nu_j)_{j=1}^d \in F_{\infty} \cong {\bf{R}}^d$ 
with $\kappa_j - \nu_j - \frac{1}{2} \in {\bf{Z}}_{\geq 0}$ and $\nu_j > - \frac{1}{2} + \varepsilon$ for each index $1 \leq j \leq d$, we have the bound 
\begin{align*} \frac{W_{\kappa, \nu}(y)}{ \prod_{j=1}^d \left\vert \Gamma \left( \frac{1}{2} + \kappa_j - \nu_j  \right) 
\Gamma \left( \frac{1}{2} + \kappa_j + \nu_j \right)  \right\vert^{\frac{1}{2}}  }
&\ll_{\varepsilon}  \vert y_{\infty} \vert^{\frac{1}{2} - \varepsilon}  \prod_{j=1}^d 
\left( \vert \kappa_j \vert + \vert \nu_j \vert + 1 \right)^A. \end{align*}
Here again, $\vert y_{\infty} \vert = \prod_{j=1}^d \vert y_{\infty, j} \vert$ denotes the idele norm of $y_{\infty} = (y_{\infty, j})_{j=1}^d$. \end{itemize}

Again, these bounds are uniform in the choice of weight $\kappa = (\kappa_j)_{j=1}^d$ and spectral parameter $\nu = (\nu_j)_{j=1}^d$.

\end{corollary}

\subsubsection*{Spectral decomposition of genuine metaplectic forms}

We now consider the genuine automorphic form $\Phi = \phi \overline{\theta}_Q$ on $\overline{G}({\bf{A}}_F)$ 
appearing in $(\ref{or})$. We shall decompose $\Phi$ spectrally to prove Theorem \ref{SCS}. 
Viewing $\operatorname{GL}_2$ as an algebraic group, we let $\overline{G}$ denote its two-fold metaplectic cover, 
as constructed via cocycles in \cite{Ge}. The adelic points $\overline{G}({\bf{A}}_F)$ fit into the exact sequence
\begin{align*} 1 \longrightarrow C_2 \longrightarrow \overline{G}({\bf{A}}_F) \longrightarrow \operatorname{GL}_2({\bf{A}}_F) \longrightarrow 1, \end{align*}
where $C_2 = \lbrace \pm 1 \rbrace$ denotes the group of square roots of unity. We recall that an automorphic form on $\overline{G}({\bf{A}}_F)$ which 
transforms nontrivially under $C_2$ is said to be genuine, in which case it corresponds to a classical Hilbert modular form of half-integral weight. We write
$L^2(\operatorname{GL}_2(F) \backslash \overline{G}({\bf{A}}_F), \omega)$ to denote the space of such genuine automorphic forms 
(of central character $\omega$), although the notation is perhaps ambiguous, i.e.~as this space does not include non-genuine forms 
arising as liftings of $\operatorname{GL}_2({\bf{A}}_F)$-automorphic forms. 
This space of genuine automorphic forms on $\overline{G}({\bf{A}}_F)$ decomposes into a Hilbert direct sum 
\begin{align*} L^2(\operatorname{GL}_2(F) \backslash \overline{G}({\bf{A}}_F), \omega) 
&= L^2_{\operatorname{disc}} \left( \operatorname{GL}_2(F) \backslash \overline{G}({\bf{A}}_F), \omega \right) \oplus
L^2_{\operatorname{cont}} \left( \operatorname{GL}_2(F) \backslash \overline{G}({\bf{A}}_F), \omega \right) \end{align*}
of a discrete spectrum $L^2_{\operatorname{disc}}\left( \operatorname{GL}_2(F) \backslash \overline{G}({\bf{A}}_F), \omega \right)$
plus a continuous spectrum $L^2_{\operatorname{cont}} \left( \operatorname{GL}_2(F) \backslash \overline{G}({\bf{A}}_F), \omega \right)$
spanned by analytic continuations of Eisenstein series. The discrete spectrum decomposes further into a direct sum    
\begin{align*} L^2_{\operatorname{disc}}(\operatorname{GL}_2(F) \backslash \overline{G}({\bf{A}}_F), \omega) 
&= L^2_{\operatorname{cusp}} \left( \operatorname{GL}_2(F) \backslash \overline{G}({\bf{A}}_F), \omega \right) \oplus
L^2_{\operatorname{res}} \left( \operatorname{GL}_2(F) \backslash \overline{G}({\bf{A}}_F), \omega \right) \end{align*} 
of cuspidal forms $L^2_{\operatorname{cusp}} \left( \operatorname{GL}_2(F) \backslash \overline{G}({\bf{A}}_F), \omega \right)$
defined by the usual vanishing condition over unipotent integrals plus residual forms 
$L^2_{\operatorname{res}} \left( \operatorname{GL}_2(F) \backslash \overline{G}({\bf{A}}_F), \omega \right)$
which arise as residues of Eisenstein series. We note that the latter space is spanned by theta series, 
and more specifically translates of the metaplectic theta series $\theta_Q$ introduced above (see \cite[$\S$ 6]{Ge}). It is then apparent
(cf.~\cite[$\S$ 6]{TT}) that we can find a basis $\mathcal{B}$ of $L^2(\operatorname{GL}_2(F) \backslash \overline{G}({\bf{A}}_F), \omega)$ consisting of: \\

\begin{itemize}

\item An orthonormal basis $\lbrace f_i \rbrace_i$ consisting of cuspidal forms 
$f_i \in L^2_{\operatorname{cusp}} \left( \operatorname{GL}_2(F) \backslash \overline{G}({\bf{A}}_F), \omega \right)$
of respective weights $\kappa_i = (\kappa_{i, j})_{j=1}^d$ and spectral parameters $\nu_i = (\nu_{i, j})_{j=1}^d$. \\

\item An orthonormal basis $\lbrace \vartheta_{\xi} \rbrace_{\xi}$ consisting of residual forms $\vartheta_{\xi} \in 
L^2_{\operatorname{res}} \left( \operatorname{GL}_2(F) \backslash \overline{G}({\bf{A}}_F), \omega \right)$ of 
respective weights $\kappa_{\xi} = (\kappa_{\xi, j})_{j=1}^d$ and spectral parameters $\nu_{\xi} = (\nu_{\xi, j})_{j=1}^d$. \\

\item An orthonormal basis $\lbrace \mathcal{E}_{\varpi} \rbrace_{\varpi}$ consisting of (contour integrals of) Eisenstein
series $\mathcal{\varpi}$ of respective weights $\kappa_{\varpi} = (\kappa_{\varpi, j})_{j=1}^d$ and spectral parameters 
$\nu_{s, \varpi} = (\nu_{s, \varpi, j})_{j=1}^d$ for $L^2_{\operatorname{cont}} \left( \operatorname{GL}_2(F) \backslash \overline{G}({\bf{A}}_F), \omega \right)$,
\begin{align*} L^2_{\operatorname{cont}} \left( \operatorname{GL}_2(F) \backslash \overline{G}({\bf{A}}_F), \omega \right) 
&= \int_{\Re(s)=1/2} \bigoplus_{\varpi} \mathcal{E}_{\varpi}(*, s) \frac{ds}{2 \pi i}. \\ \end{align*}

\end{itemize} 

Decomposing $\Phi = \phi \overline{\theta}_Q$ in terms of such a basis $\mathcal{B}$, 
we obtain for any $\overline{g} = (g, \zeta) \in \overline{G}({\bf{A}}_F)$ the decomposition
\begin{align}\label{SD} \Phi(\overline{g}) 
&= \sum_i \langle \Phi, f_i \rangle \cdot f_i (\overline{g})  + \sum_{\xi} \langle \Phi, \vartheta_{\xi} \rangle \cdot \vartheta_{\xi}(\overline{g})
+ \sum_{\varpi} \int_{\Re(s) = 1/2} \langle \Phi, \mathcal{E}_{\varpi}(*, s) \rangle \cdot \mathcal{E}_{\varpi}(\overline{g}, s) \frac{ds}{2 \pi i}. \end{align}
Here, $\langle \cdot, \cdot \rangle$ denotes the inner product on $L^2\left( \operatorname{GL}_2(F) \backslash \overline{G}({\bf{A}}_F), \omega \right)$.
To see that the sums over coefficients $\mathfrak{K}_i = \langle \Phi, f_i \rangle$, $\mathfrak{K}_{\xi} = \langle \Phi, \vartheta_{\xi} \rangle$, and 
$\mathfrak{K}_{\varpi} = \langle \Phi, \mathcal{E}_{\varpi}(*, s) \rangle$ appearing in $(\ref{SD})$ are bounded in a suitable way, 
we shall use the fact that $\Phi$ has convergent Sobolev norm. To be more precise, let us first recall the definition of the Sobolev norm
on any sufficiently smooth $L^2$-automorphic form $\phi$ on $\operatorname{GL}_2({\bf{A}}_F)$, i.e.~on any sufficiently smooth automorphic 
form $\phi \in L^2(\operatorname{GL}_2(F) \backslash \operatorname{GL}_2({\bf{A}}_F), \omega)$. We refer to \cite[$\S 2.10$]{BH10} 
(cf.~\cite{MiVe}, \cite[$\S$ 6]{TT}) for more background. In short, the action of of $\operatorname{GL}_2(F_{\infty})$ on 
$L^2(\operatorname{GL}_2(F) \backslash \operatorname{GL}_2({\bf{A}}_F), \omega)$ induces an action
of its Lie algebra $\mathfrak{gl}_2(F_{\infty})$ on this space, and hence an action of its Lie subalgebra 
$\mathfrak{g} = \mathfrak{sl}_2(F_{\infty})$ on this space. Writing $e_j = (0, \cdots, 0, 1, 0, \cdots, 0)$ with $1$ at position $j$
for each index $1 \leq j \leq d$, this latter action is generated by the linearly independent vectors 
\begin{align*} H_j = \left(\begin{array} {cc} e_j & 0 \\ 0 & e_j \end{array}\right)  \quad \quad  
R_j  =  \left(\begin{array} {cc} 0 & e_j \\ 0 & 0 \end{array}\right) \quad \quad 
L_j =   \left(\begin{array} {cc} 0 & 0 \\ e_j & 0 \end{array}\right). \end{align*}
We also have the action of the universal enveloping algebra $\mathcal{U}(\mathfrak{g})$ of $\mathfrak{g}$ on 
$L^2(\operatorname{GL}_2(F) \backslash \operatorname{GL}_2({\bf{A}}_F), \omega)$ via higher-order differential operators. 
Writing a generic element $k(\vartheta) \in \operatorname{SO}_2(F_{\infty})$ for $\vartheta = (\vartheta_j)_{j=1}^d \in ({\bf{R}}/{\bf{Z}})^d$ as 
\begin{align*} k(\vartheta) 
&=  \left(\begin{array} {cc} \cos \vartheta & \sin \vartheta \\ - \sin \vartheta & \cos \vartheta \end{array}\right) \in \operatorname{SO}_2(F_{\infty}), \end{align*}
the operators corresponding to the basis elements are given explicitly by 
\begin{align*} d H_j &= -2 y_j \sin(2 \vartheta_j) \partial_{x_j} + 2 y_j \cos(2 \vartheta_j) \partial_{y_j} + \sin(2 \vartheta_j) \partial_{\vartheta_j} \\ 
d R_J &= y_j \cos(2 \vartheta_j) \partial_{x_j} + y_j \sin(2 \vartheta_j) \partial_{y_j} + \sin^2(\vartheta_j) \partial_{\vartheta_j} \\
dL_j &= y_j \cos(2 \vartheta_j) \partial_{x_j} + y_j \sin(2 \vartheta_j) \partial_{y_j} - \cos^2(\vartheta_j) \partial_{\vartheta_j}.  \end{align*} 
Now, as explained \cite[$\S 2.10$]{BH10}, we know that the action of any differential operator $\mathcal{D} \in \mathcal{U}(\mathfrak{g})$ 
commutes with the corresponding spectral decomposition of $\phi$. One can extend this discussion in a natural way to the space of 
genuine metaplectic forms $L^2(\operatorname{GL}_2(F) \backslash \overline{G}({\bf{A}}_F), \omega)$ (see \cite[$\S 6.2$]{TT} and more generally 
\cite{MiVe}). In particular, there is a corresponding action of the universal enveloping algebra $\mathcal{U}(\mathfrak{g})$ on 
$L^2(\operatorname{GL}_2(F) \backslash \overline{G}({\bf{A}}_F), \omega)$. This action commutes with the spectral decomposition $(\ref{SD})$ 
of $\Phi \in L^2(\operatorname{GL}_2(F) \backslash \overline{G}({\bf{A}}_F), \omega)$ 
in the sense that for any $\mathcal{D} \in \mathcal{U}(\mathfrak{g})$, we have the relation 
\begin{align}\label{SDrelation} \vert \vert \mathcal{D}\Phi \vert\vert^2 
&= \sum_i \langle \Phi, f_i \rangle^2 \cdot \vert \vert \mathcal{D} \Phi \vert\vert^2 
+ \sum_{\xi} \langle \Phi, \vartheta_{\xi} \rangle^2 \cdot \vert \vert \mathcal{D} \vartheta_{\xi} \vert\vert^2 + 
\int_{\Re(s) = 1/2} \sum_{\varpi} \langle \Phi, \mathcal{E}_{\varpi}(*, s) \rangle^2 
\cdot \vert\vert \mathcal{D} \mathcal{E}_{\varpi} \vert\vert^2 \frac{ds}{2 \pi i}. \end{align}
Such a relation of course holds for any $\Phi \in L^2(\operatorname{GL}_2(F) \backslash \overline{G}({\bf{A}}_F), \omega)$,
and it is then natural for any choice of integer $B \geq 0$ to define the corresponding Sobolev norm 
\begin{align*} \vert \vert \Phi \vert \vert_{\mathcal{S}^B} 
&= \sum_{ \operatorname{ord}(\mathcal{D}) \leq B} \vert \vert \mathcal{D} \Phi \vert \vert^2,\end{align*}
where the sum runs over all monomials in $H_{j_1}$, $R_{j_2}$, and $L_{j_3}$ of order at most $B$.
It is well-known that any smooth $\Phi \in L^2(\operatorname{GL}_2(F) \backslash \overline{G}({\bf{A}}_F), \omega)$
is convergent in this Sobolev norm in the following sense. 

\begin{lemma}\label{SNC} 

Given $\phi$ any smooth automorphic form in $\operatorname{GL}_2(F) \backslash \operatorname{GL}_2({\bf{A}}_F)$ 
or more generally $\operatorname{GL}_2(F) \backslash \overline{G}({\bf{A}}_F)$ of a given central character (not necessarily $\mathcal{K}$-finite)
and $B \geq 0$ any integer, we have the Sobolev norm bound $\vert\vert \phi \vert\vert_{\mathcal{S}^B} \ll 1$. 
Here, the implied constant depends only on the choice of integer $B \geq 0$. \end{lemma}

\begin{proof} See \cite[Lemma 6.1]{TT} or more generally \cite[$\S$ 2.4]{MiVe}. \end{proof}

Using this result, we can deduce from the relation $(\ref{SDrelation})$ that the sums of coefficients in the spectral expansion 
$(\ref{SD})$ are bounded suitably in terms of the corresponding spectral parameters. We refer to the discussions in 
\cite[$\S$ 5-6]{TT} and \cite[$\S$ 2.10]{BH10} for more details. This convergence allows us to proceed with the proof of Theorem \ref{SCS} 
via the spectral decomposition $(\ref{SD})$ of the genuine metaplectic form $\Phi = \phi \overline{\theta}_Q$ as follows.
 
\subsubsection*{Proof of Theorem \ref{SCS}} 

Recall that via the integral presentation $(\ref{or})$ above, it will suffice to bound the Fourier-Whittaker coefficient at 
a give nonzero $F$-integer $q$ of the genuine metaplectic form $\Phi = \phi \overline{\theta}_Q$. 
To be more precise, it will do to consider only the spectral decomposition of the coefficient at $q$ of $\Phi$,
which we can expand out via $(\ref{SD})$ of $\Phi$ at the elements 
\begin{align*} \overline{g} &= (g, \zeta) 
= \left( \left(\begin{array} {cc} \frac{1}{Y_{\infty}} & x_{\infty} \\ ~ & 1 \end{array}\right) , 1 \right) \in \overline{G}(F_{\infty}), \end{align*}
suppressing the $\zeta \in C_2 = \lbrace \pm 1 \rbrace$ henceforth to lighten notation to derive the presentation 
\begin{align*} &\sum_{ \gamma \in F^{\times} } \frac{\lambda( \gamma^2 + q)}{\vert \gamma^2 + q \vert^{\frac{1}{2}} }  
W \left( \frac{ \gamma^2 + q }{ Y_{\infty} } \right) \\ 
&= \vert Y_{\infty} \vert^{\frac{1}{4}} \int_{ I \cong [0,1]^d \subset F_{\infty} } 
\Phi \left( \left(\begin{array} {cc} \frac{1}{Y_{\infty}} & x_{\infty} \\ ~ & 1 \end{array}\right) \right) \psi(- q x_{\infty} ) dx_{\infty}  \\
&= \vert Y_{\infty} \vert^{\frac{1}{4}} \sum_{i} \langle \Phi, f_i \rangle \cdot \int_{ I \cong [0,1]^d \subset F_{\infty} } 
f_i \left( \left(\begin{array} {cc} \frac{1}{Y_{\infty}} & x_{\infty} \\ ~ & 1 \end{array}\right) \right) \psi(- q x_{\infty} ) dx_{\infty} \\
&+ \vert Y_{\infty} \vert^{\frac{1}{4}} \sum_{\xi} \langle \Phi, \vartheta_{\xi} \rangle \cdot \int_{ I \cong [0,1]^d \subset F_{\infty} } 
\vartheta_{\xi} \left( \left(\begin{array} {cc} \frac{1}{Y_{\infty}} & x_{\infty} \\ ~ & 1 \end{array}\right) \right) \psi(- q x_{\infty} ) dx_{\infty} \\
&+ \vert Y_{\infty} \vert^{\frac{1}{4}} \sum_{\varpi} \int_{\Re(s)=1/2} \langle \Phi, \mathcal{E}_{\varpi}(*, s) \rangle \cdot 
\int_{ I \cong [0,1]^d \subset F_{\infty} } \mathcal{E}_{\varpi} \left( \left( \left(\begin{array} {cc} 
\frac{1}{Y_{\infty}} & x_{\infty} \\ ~ & 1 \end{array}\right) \right), s \right)
\psi (- q x_{\infty} ) dx_{\infty} \frac{ds}{2 \pi i}. \end{align*}
Note that the contributions of the spectral coefficients in this expression are bounded via the convergence of $\Phi$ in the Sobolev norm
(cf.~$(\ref{SDrelation})$ with Lemma \ref{SNC}). We can now follow essentially the same argument as given in \cite[$\S$ 6.6-6.8]{TT} to
derive the stated estimate of Theorem \ref{SCS} for the shifted convolution sum on the left hand side of this expression. 
To be more precise, let us first consider the second integral over residual terms $\lbrace \vartheta_{\xi} \rbrace_{\xi}$ 
in this spectral expansion, which after writing $c_{\vartheta_{\xi}}$ to denote the coefficients in the Dirichlet series of 
each respective Mellin transform ($L$-function) $L(s, \vartheta_{\xi})$ can be expressed equivalently as 
\begin{align*} \vert Y_{\infty} \vert^{\frac{1}{4}} \sum_{\xi} \langle \Phi, \vartheta_{\xi} \rangle \cdot \int_{ I \cong [0,1]^d \subset F_{\infty} } 
\vartheta_{\xi} \left( \left(\begin{array} {cc} \frac{1}{Y_{\infty}} & x_{\infty} \\ ~ & 1 \end{array}\right) \right) \psi(- q x_{\infty} ) dx_{\infty}
&= \vert Y_{\infty} \vert^{\frac{1}{4}} \sum_{\xi} \langle \phi \overline{\theta}_Q, \vartheta_{\xi} \rangle \cdot 
\frac{ c_{\vartheta_{\chi}(q)}}{\vert q \vert^{\frac{1}{2}}} W_{\vartheta_{\xi}} \left( \frac{q}{Y_{\infty}} \right). \end{align*}
Hence (cf.~\cite[$\S 6.8$]{TT}), we can derive the crude estimate 
\begin{align*} \vert Y_{\infty} \vert^{\frac{1}{4}} \sum_{\xi} \langle \Phi, \vartheta_{\xi} \rangle \cdot \int_{ I \cong [0,1]^d \subset F_{\infty} } 
\vartheta_{\xi} \left( \left(\begin{array} {cc} \frac{1}{Y_{\infty}} & x_{\infty} \\ ~ & 1 \end{array}\right) \right) \psi(- q x_{\infty} ) dx_{\infty} 
&= I(W) M_{\pi, q} \end{align*} 
for some linear functional $I(W)$ in the chosen weight function $W$, where $M_{\pi, q} \geq 0$ is a constant depending
only on our initial cuspidal automorphic representation $\pi$ of $\operatorname{GL}_2({\bf{A}}_F)$ and the chosen $F$-integer $q$.
More precisely, we see from inspection of the inner products $\langle \phi \overline{\theta}_Q, \vartheta_{\xi} \rangle$ that $M_{\pi, q}$ 
vanishes unless $\pi$ is dihedral and $q$ totally positive, i.e.~since the theory of the Shimura integral (see e.g.~\cite[$\S$ 4.7]{TT}) 
implies that the inner product vanishes unless the completed symmetric square $L$-function $\Lambda(s, \operatorname{Sym}^2 \pi)$ 
has a pole at $s=1$, and since the coefficients in the Fourier-Whittaker expansion of $\vartheta_{\xi}$ are only supported on totally positive 
$F$-integers. Since we assume throughout that the representation $\pi$ is non-dihedral, we deduce that $M_{\pi, q} =0$, and hence that this 
term does not contribute to the estimate. 
To bound the remaining contributions, we first argue that it will suffice to consider the cuspidal spectrum, as the contribution of the continuous 
spectrum can be estimated via a minor variation of the same argument (cf.~\cite{BH10}, \cite[$\S$ 6]{TT}). To estimate the cuspidal contribution, 
we again write the $L$-function coefficients of each cusp form $f_i$ as $c_{f_i}$, so that 
\begin{align}\label{Icusp} \vert Y_{\infty} \vert^{\frac{1}{4}} \sum_{i} \langle \Phi, f_i \rangle \cdot \int_{ I \cong [0,1]^d \subset F_{\infty} } 
f_i \left( \left(\begin{array} {cc} \frac{1}{Y_{\infty}} & x_{\infty} \\ ~ & 1 \end{array}\right) \right) \psi(- q x_{\infty} ) dx_{\infty} 
&= \vert Y_{\infty} \vert^{\frac{1}{4}} \sum_i \langle \phi \overline{\theta}_Q, f_i \rangle 
\cdot \frac{c_{f_i}(q)}{\vert q \vert^{\frac{1}{2}}} W_{f_i} \left(\frac{q}{Y_{\infty}}  \right). \end{align}
Let us now write $0 < \delta_0 \leq 1/2$ to denote the best exponent approximation for the Fourier coefficients of genuine
metaplectic forms, i.e.~so that $c_{f_i}(q) \ll_{\varepsilon} \vert q \vert^{\delta_0 + \varepsilon}$ for any 
$\varepsilon >0$ and nonzero $F$-integer $q$. Note that by the theorem of Kohnen-Zagier \cite{KZ} and more generally
Baruch-Mao \cite{BM07}, this exponent $\delta_0$ is seen to be equivalent to the best exponent approximation the generalized
Lindel\"of hypothesis for $\operatorname{GL}_2({\bf{A}}_F)$-automorphic forms, as mentioned above. 
Using Corollary \ref{totallyrealwhittaker} to bound the contribution of each archimedean Whittaker function 
$W_{f_i}(y_{\infty}) = W_{\kappa_i, \nu_i}(y_{\infty})$ of $y_{\infty} \in F_{\infty}^{\times}$, 
we deduce that $(\ref{Icusp})$ can be bounded above by the quantity 
\begin{align*} &\ll_{\varepsilon} \vert Y_{\infty} \vert^{\frac{1}{2}} \cdot 
\sum_i \langle \phi \overline{\theta}_Q, f_i \rangle \cdot \vert q \vert^{\delta_0 + \varepsilon - \frac{1}{2}}
\cdot \left\lvert  \frac{q}{Y_{\infty}} \right\rvert^{\frac{1}{2} - \frac{\theta_0}{2} - \varepsilon} \cdot 
\prod_{j=1}^d \left( \vert \kappa_{i, j} \vert + \vert \nu_{i, j} \vert + 1 \right)^A \cdot \vert \vert f_i \vert\vert \\ 
&= \vert Y_{\infty} \vert^{\frac{1}{4}} \cdot \sum_i \langle \phi \overline{\theta}_Q, f_i \rangle \cdot \vert q \vert^{\delta_0 + \varepsilon - \frac{1}{2}}
\cdot \left\lvert  \frac{q}{Y_{\infty}} \right\rvert^{\frac{1}{2} - \frac{\theta_0}{2} - \varepsilon} \cdot 
\prod_{j=1}^d \left( \vert \kappa_{i, j} \vert + \vert \nu_{i, j} \vert + 1 \right)^A  \end{align*} 
for any choice of $\varepsilon >0 $. Here, we use that $\vert \vert f_i \vert \vert =1$ for each $i$ (by our choice of orthonormal basis),
and that the spectral parameter $\nu$ can be approximated by $\theta_0/2$ (cf. \cite[Proposition 3.1 and Lemma 6.2]{TT}). We also 
choose the archimedean idele representative $Y_{\infty} \in F_{\infty}^{\times}$ in a suitable way so that the bounds of 
Corollary \ref{totallyrealwhittaker} can be applied. Now, using the Sobolev norm convergence (Lemma \ref{SNC}) 
and more precisely the Plancherel formula with itererated applications of the generalized Laplacian operator to deduce that 
\begin{align*} \sum_{i} \langle \phi \overline{\theta}_Q, f_i \rangle^2 \cdot 
\prod_{j=1}^d \left( \vert \kappa_{i, j} \vert + \vert \nu_{i, j} \vert + 1  \right)^{2A} 
\ll \vert\vert \phi \overline{\theta}_Q \vert\vert^2 \ll_{2A} 1, \end{align*} 
we can then apply the Cauchy-Schwarz inequality to the quantity in the previous bound to deduce that 
\begin{align*} \vert Y_{\infty} \vert^{\frac{1}{4}} \sum_{i} \langle \Phi, f_i \rangle \cdot \int_{ I \cong [0,1]^d \subset F_{\infty} } 
f_i \left( \left(\begin{array} {cc} \frac{1}{Y_{\infty}} & x_{\infty} \\ ~ & 1 \end{array}\right) \right) \psi(- q x_{\infty} ) dx_{\infty} 
&\ll_{\varepsilon, 2A} \vert Y_{\infty} \vert^{\frac{1}{4}} \cdot \vert q  \vert^{\delta_0 - \frac{1}{2} } 
\cdot \left\lvert  \frac{ q}{Y_{\infty}} \right\rvert^{\frac{1}{2} - \frac{\theta_0}{2} - \varepsilon}. \end{align*}
The contributions from the continuous spectrum are bounded in the same way, using a variation of the argument given above
to bound spectral coefficients (via adjoint properties of the inner product). The claimed estimate of Theorem \ref{SCS} follows. $\qed$

\subsubsection*{Proof of Theorem \ref{SCS2}}

We now give the following variation of the proof of Theorem \ref{SCS} to show Theorem \ref{SCS2}. 
Let us take for granted the setup described for the statement of Theorem \ref{SCS2}. 
To begin, fix $W \in L^2(F_{\infty}^{\times})$ any smooth weight 
function as in the statement, i.e.~after composing with the norm.
Let $\phi = \otimes_v \phi_v \in V_{\pi}$ be any pure tensor whose nonarchimedean local components are each essential Whittaker vectors, 
and whose archimedean local components are chosen so that as functions of $y_{\infty} \in F_{\infty}^{\times} \cong ({\bf{R}}^{\times})^d$,
\begin{align*} W_{\phi}(y_{\infty}) := W_{\phi} \left( \left( \begin{array}{cc} y_{\infty} & ~\\Ê~& 1 \end{array}  \right) \right) &= W(y_{\infty}). \end{align*}
It is then easy to deduce the properties of the Fourier-Whittaker expansion of $\phi$ outlined above that 
\begin{align}\label{quadsum} \sum_{\gamma \in F^{\times}} W_{\phi} \left( \left( \begin{array}{cc} \frac{q(\gamma)}{Y_{\infty}} & ~Ê\\Ê~Ê& 1 \end{array} \right) \right)
= \sum_{\gamma \in F^{\times}} \frac{\lambda_{\pi}(q(\gamma))}{\vert \gamma \vert^{\frac{1}{2}}} W\left( \frac{q(\gamma)}{Y_{\infty}} \right). \end{align}
Here, we use all of the same conventions as above, taking $Y_{\infty} \in F_{\infty}^{\times}$ to be a fixed, 
totally positive archimedean idele representative of norm $\vert Y_{\infty} \vert = Y$. 
Note that the sum $(\ref{quadsum})$ here is supported only on $F$-integers $\gamma = n \in \mathcal{O}_F/ \mathcal{O}_F^{\times}$. 
We shall henceforth write $(n)$ to denote the principal ideal generated by such an $F$-integer to simplify notations.
Our strategy now is to decompose the pure tensor $\phi$ into a linear combination of smooth Poincar\'e series
on $\operatorname{GL}_2({\bf{A}}_F)$ in the style of Blomer \cite[$\S 3$]{VB}, 
then apply Poisson summation to identify the coefficients in the sum with the Fourier-Whittaker coefficients of certain
genuine Poincar\'e series on the metaplectic cover $\overline{G}({\bf{A}}_F)$. 
Making such an identification, we can then decompose these genuine forms spectrally according to the discussion given above, 
passing to unipotent integrals describing the coefficients rather than invoking any Kuznetsov trace formula directly. 
In fact, this procedure allows us to give a ``soft" generalization of the argument of \cite{VB}, choosing suitable vectors in the 
Kirillov model and Schwartz functions parametrizing Poincar\'e series to reduce to a variation of the proof given for Theorem \ref{SCS} above. 

To make this argument rigorous, we first need to recall some more background about the Poincar\'e series. 
Let us again write $N_2 \subset \operatorname{GL}_2$ to denote the subgroup of upper-triangular unipotent matrices, 
and consider the space $\mathcal{S}(N_2({\bf{A}}_F)\backslash \overline{G}({\bf{A}}_F); \psi)$ of smooth functions $f$ on 
$\overline{G}({\bf{A}})$ whose action by $N_2({\bf{A}}_F)$ is given by a chosen additive character $\psi$ 
of ${\bf{A}}_F/F$ (extended in the natural way to $N_2(F) \backslash N_2({\bf{A}})$),
i.e.~so that $f(n \overline{g}) = \psi(n) f(\overline{g} )$ for all $n \in N_2({\bf{A}}_F)$ and $\overline{g} = (g, \zeta) \in \overline{G}({\bf{A}}_F)$
with $g \in \operatorname{GL}_2({\bf{A}}_F)$ and $\zeta \in Z_2$. Fixing $\Gamma \subset \operatorname{GL}_2({\bf{A}}_F)$ a discrete,
finite co-volume subgroup. Let $\Gamma_{\infty} \cong N_2(F)$ denote the subgroup stabilized by the cusp at infinity.
Let us also fix an idele class character $\omega$ of ${\bf{A}}_F$, viewed as a character of 
$\gamma \in \Gamma$ in the usual way via evaluation at the lower left entry.
Let us also fix a multiplier system $\vartheta: \operatorname{GL}_2 \rightarrow {\bf{C}}$, 
as described in \cite{Pro} and \cite[$\S 2$]{Ge} for each of the real places. 
We then consider the Poincar\'e series $P_{f, \omega, \vartheta}$ defined on $\overline{g} = (g, \zeta) \in \overline{G}({\bf{A}}_F)$ by 
\begin{align}\label{P} P_{f, \omega, \vartheta}(\overline{g}) 
&= \sum_{\gamma \in \Gamma_{\infty} \backslash \Gamma} \omega(\gamma) \cdot \overline{\vartheta(\gamma)} \cdot f(\gamma \overline{g}). \end{align}
This Poincar\'e series $P_{f, \omega, \vartheta}$ is absolutely convergent, uniformly on compact subsets, 
and determines a smooth $L^2$-automorphic on $\overline{G}({\bf{A}}_F)$ (see e.g.~\cite{CPS}). 
Note that by restricting to the component $\operatorname{GL}_2({\bf{A}}_F)$ of $\overline{G}({\bf{A}}_F)$ 
we obtain a natural subspace inclusion $\mathcal{S}(N_2({\bf{A}}_F) \backslash \operatorname{GL}_2({\bf{A}}_F); \psi)
\subseteq \mathcal{S}(N_2({\bf{A}}_F)\backslash \operatorname{GL}_2({\bf{A}}_F); \psi)$, 
and hence recover the corresponding construction of Poincar\'e series on $\operatorname{GL}_2({\bf{A}}_F)$ (taking the multiplier term to be trivial). 
We can also compute Fourier-Whittaker expansions as follows, 
noting that our setup here is simpler than the case of general number fields,
as we consider only sums over principal ideals in $(\ref{quadsum})$. 
Following \cite{CPS}, let us consider the set of $F$-rational numbers determined by the set 
\begin{align*} \Omega(\Gamma) &= \left\lbrace c \in F_{\infty}^{\times}: 
N_2(F_{\infty}) \cdot w \cdot \underline{c} \cdot N_2(F_{\infty}) \cap \Gamma \neq \emptyset \right\rbrace, \end{align*}
where we use the shorthand notations  
\begin{align*} \underline{c} &:= \left( \begin{array}{cc} c & ~\\ ~Ê& 1 \end{array} \right) \quad \text{and} \quad
w := \left( \begin{array}{cc} ~& -1 \\ 1 & ~\end{array} \right).\end{align*}
Given $c \in \Omega(\Gamma)$, we then consider the subgroup $\Gamma_c \subset \Gamma$ defined by 
\begin{align*} \Gamma_c &:=  N_2(F_{\infty}) \cdot w \cdot \underline{c} \cdot N_2(F_{\infty}) \cap  \Gamma, \end{align*}
which is left and right invariant by $\Gamma_{\infty}$. We can also decompose each  
$\gamma \in \Gamma_c$ according to the Bruhat decomposition as 
$\gamma = n_1(\gamma) \cdot w \cdot \underline{c} \cdot n_2(\gamma)$
with $n_j(\gamma) \in N_2(F_{\infty})$ for each of $j=1,2$. More explicitly, we have  
\begin{align*} \gamma = \left( \begin{array}{cc} a & b \\Êc & d \end{array} \right) \in \Gamma_c
\quad \implies \quad n_1(\gamma) = \left(\begin{array}{cc}1&ac^{-1}\\~&1\end{array} \right)
\quad \text{and} \quad n_2(\gamma) =  \left(\begin{array}{cc}1&dc^{-1}\\~&1\end{array} \right) ,\end{align*} 
which can be deduced from the elementary matrix decomposition 
\begin{align*} \left( \begin{array}{cc} a & b \\ c & d \end{array} \right) &= 
\left( \begin{array}{cc} 1 & ac^{-1} \\ ~ & 1 \end{array} \right)
\left( \begin{array}{cc} ~ & -1 \\ 1 &  \end{array} \right)
\left( \begin{array}{cc} c & ~ \\ ~ & c^{-1}(ad-bc) \end{array} \right)
\left( \begin{array}{cc} 1 & dc^{-1} \\ ~ & 1 \end{array} \right). \end{align*}

\begin{proposition}\label{PoincarŽ} 

Given nontrivial additive characters $\psi_1, \psi_2$ of ${\bf{A}}_F/F$, the Fourier-Whittaker coefficient 
\begin{align*} W_{P_{f, \omega, \vartheta}}(\overline{g}) = W_{P_{f, \omega, \vartheta}, \psi_2}(\overline{g}) 
&:= \int_{N_2(F) \backslash N_2({\bf{A}}_F)} P_{f, \omega, \vartheta} \left( n \overline{g} \right) \psi_2^{-1}(n) dn \end{align*}
with respect to $\psi_2$ of the Poincar\'e series $P_{f, \omega, \vartheta}$ constructed from some 
$f \in \mathcal{S}(N_2({\bf{A}}_F) \backslash \overline{G}({\bf{A}}_F); \psi_1)$
is given by the following formula: For any $\overline{g} = (g, \zeta) \in \overline{G}({\bf{A}}_F)$, we have that 
\begin{align*} W_{P_{f, \omega, \vartheta}}(\overline{g}) &= \int_{ {\bf{A}}_F/F }  f(\overline{g}) \psi_1(x)\psi_2(-x) dx \\
&+ \sum_{c \in \Omega(\Gamma)} 
\left(  \sum\limits_{\gamma \in \Gamma_{\infty} \backslash \Gamma_{c} / \Gamma_{\infty}  } 
\omega(\gamma) \cdot \overline{\vartheta (\gamma)} \cdot \psi_1(n_1(\gamma) ) \cdot \psi_2( n_2(\gamma) ) \right)
\int_{N_2({\bf{A}}_F)} f(w \cdot \underline{c} \cdot n \cdot \overline{g}) \psi_2^{-1}(n) dn. \end{align*} 

In particular, fixing $\psi$ to be the standard additive character, and then taking $\psi_1(x) = \psi_{\infty}(m x)$
and $\psi_2(x) = \psi_{\infty}(r x)$ for nonzero $F$-integers $m, r \in \mathcal{O}_F$, 
with $\vartheta(\gamma) = \left(\frac{c}{d} \right)^*$ as described (for each component) in \cite[$\S 2$]{Ge} and \cite{Pro}, 
we arrive at the simpler expression 
\begin{align*} W_{P_{f, \omega, \vartheta}}(\overline{g}) 
&= \int_{ {\bf{A}}_F/F }  f(\overline{g}) \psi_{\infty}( m x -r x) dx + \sum_{c \in \Omega(\Gamma)} 
\operatorname{Kl}_{\Gamma, \omega, \vartheta}(m, r; c) \cdot \mathcal{F}_{f, r, c}(\overline{g}) \end{align*}
where $\operatorname{Kl}_{\Gamma, \omega, \vartheta}(m, r; c)$ denotes the Kloosterman sum defined by 
\begin{align*} \operatorname{Kl}_{\Gamma, \omega, \vartheta}(m, r; c) 
&= \sum\limits_{\gamma \in \Gamma_{\infty} \backslash \Gamma_{c} / \Gamma_{\infty}  } 
\omega(\gamma) \cdot \overline{\vartheta (\gamma)} \cdot \psi_{\infty}( m  n_1(\gamma) + r  n_2(\gamma)) \\
&=  \sum\limits_{\gamma = \left( \begin{array}{cc} a & b \\ c & d \end{array} \right)  \in \Gamma_{\infty} \backslash \Gamma_{c} / \Gamma_{\infty}  } 
\omega(d) \left(\frac{c}{d} \right)^* \psi_{\infty}\left( \frac{m  a + r  d}{c} \right), \end{align*} 
and $\mathcal{F}_{f, ,r, c}$ the intertwining operator defined by 
\begin{align*} \mathcal{F}_{f, r, c}(\overline{g}) &= \int_{N_2({\bf{A}}_F)} f(w \cdot \underline{c} \cdot n \cdot \overline{g}) \psi_2^{-1}(n) dn
= \int_{{\bf{A}}_F} f \left( w \underline{c} \left( \begin{array}{cc} 1 & x \\~& 1 \end{array} \right) \overline{g}  \right) \psi_{\infty}(-r x) dx. \end{align*} 
That is, the Fourier-Whittaker coefficient of $P_{f, \omega, \vartheta}(\overline{g})$ at $r$ (with respect to the standard additive character) is described in this way. 
\end{proposition}

\begin{proof} 

Cf.~\cite[Proposition 2.5]{CPS}; the same calculation works here. 
Given $\overline{g} = (g, \zeta) \in \overline{G}({\bf{A}}_F)$, $\omega$, $u$, and
$f \in {S}(N_2({\bf{A}}_F) \backslash \overline{G}({\bf{A}}_F); \psi_1)$ as above, we compute 
\begin{align*} W_{P_{f, \omega, \vartheta}, \psi_2}(\overline{g}) 
&= \int_{N_2(F) \backslash N_2({\bf{A}}_F)} P_{f, \omega, \vartheta}(n \overline{g}) \psi_2^{-1}(n) dn 
= \int_{ {\bf{A}}_F/F } P_{f, \omega, \vartheta} \left( \left( \begin{array}{cc} 1 & x \\ ~Ê& 1 \end{array} \right) \overline{g} \right) \psi_2(-x)dx \\
&= \int_{ {\bf{A}}_F/F } \sum_{ \gamma \in \Gamma_{\infty} \backslash \Gamma} \omega(\gamma) \cdot \overline{\vartheta (\gamma)} \cdot
f \left( \gamma \left( \begin{array}{cc} 1 & x \\ ~Ê& 1 \end{array} \right) \overline{g}  \right) \cdot \psi_2(-x) dx , \end{align*}
which after using the Bruhat decomposition expansion 
$\Gamma = \Gamma_{\infty} \cup \bigcup_{c \in \Omega(\Gamma)} \Gamma_c$ (see \cite[$\S$ 2]{CPS}) equals 
\begin{align*} \int_{ {\bf{A}}_F/F } f \left( \left( \begin{array}{cc} 1 & x \\ ~Ê& 1 \end{array} \right) \overline{g}  \right) \psi_2(-x) dx 
+ \int_{ {\bf{A}}_F/F } \sum_{c \in \Omega(\Gamma)} \sum_{ \gamma \in \Gamma_{\infty} \backslash \Gamma_c} 
\omega(\gamma) \cdot \overline{ \vartheta (\gamma)} \cdot 
f \left( \gamma \left( \begin{array}{cc} 1 & x \\ ~Ê& 1 \end{array} \right) \overline{g}  \right) \cdot \psi_2(-x) dx. \end{align*}
Here, it is easy to see from the fact that $f \in \mathcal{S}(N_2({\bf{A}}_F) \backslash \overline{G}({\bf{A}}); \psi_1)$
that the first integral equals 
\begin{align*} \int_{ {\bf{A}}_F/F } f \left( \overline{g}  \right) \psi_1(x) \psi_2(-x) dx. \end{align*}
To compute the second integral, we use the identifications 
${\bf{A}}_F/F \cong N_2(F) \backslash N_2({\bf{A}}_F) \cong \Gamma_{\infty} \backslash N_2({\bf{A}}_F)$ to find 
\begin{align*} &\int_{ {\bf{A}}_F/F } \sum_{c \in \Omega(\Gamma)} \sum_{ \gamma \in \Gamma_{\infty} \backslash \Gamma_c} 
\omega(\gamma) \cdot \overline{ \vartheta (\gamma)} \cdot
f \left( \gamma \left( \begin{array}{cc} 1 & x \\ ~Ê& 1 \end{array} \right) \overline{g}  \right) \cdot \psi_2(-x) dx \\
&= \sum_{c \in \Omega(\Gamma)} \int_{ \Gamma_{\infty} \backslash N_2( {\bf{A}}_F) } 
\sum_{\gamma \in \Gamma_{\infty} \backslash \Gamma_c} \omega(\gamma) \cdot \overline{ \vartheta (\gamma)} \cdot
f( \gamma n \overline{g}) \cdot \psi_2^{-1}(n) dn \\ 
&= \sum_{c \in \Omega(\Gamma)} \sum_{ \gamma \in \Gamma_{\infty} \backslash \Gamma_c / \Gamma_{\infty}} 
\int_{ N_2( {\bf{A}}_F ) } f( \gamma n \overline{g} ) \psi_2^{-1}(n)dn, \end{align*}
which after decomposing each $\gamma \in \Gamma_{\infty} \backslash \Gamma_c / \Gamma_{\infty}$ in Bruhat form
as $\gamma = n_1(\gamma) \cdot w \cdot \underline{c} \cdot n_2(\gamma)$ and making a change of variables 
$n \rightarrow n_2(\gamma)^{-1} \cdot n$ is the same as the stated integral
\begin{align*} \sum_{ c \in \Omega(\Gamma) } \sum_{\gamma \in \Gamma_{\infty} \backslash \Gamma_c / \Gamma_{\infty} } 
\omega(\gamma) \cdot \overline{ \vartheta (\gamma)} \cdot \psi_1(n_1(\gamma)) \cdot \psi_2(n_2(\gamma)) 
\int_{ N_2( {\bf{A}}_F) } f(w \cdot \underline{c} \cdot n \cdot \overline{g}) \psi_2^{-1}(n)dn. \end{align*} \end{proof}

Let us now consider the latter Poincar\'e series above, i.e.~for a given nonzero $F$-integer $m \in \mathcal{O}_F$ fix $\psi_1(x) = \psi_{\infty}(mx)$,
and fix a suitable rapidly decaying Schwartz function $f \in \mathcal{S}(N_2({\bf{A}}_F) \backslash \operatorname{GL}_2({\bf{A}}_F); \psi_1)$, 
then consider the corresponding smooth Poincar\'e series $P_m$ on $g \in \operatorname{GL}_2({\bf{A}}_F)$ defined by
\begin{align*} P_m(g) := P_f(g) = \sum_{\Gamma_{\infty} \backslash \Gamma} f(\gamma g). \end{align*} 
Here, we take the multiplier $\vartheta$ in the definition above to be trivial, 
and drop the corresponding subscript from the notations, 
writing $\operatorname{Kl}_{\Gamma}(m, r; c) = \operatorname{Kl}_{\Gamma, \omega, 1}(m, r; c)$.
We now decompose our fixed pure tensor $\phi \in V_{\pi}$ into a linear combination of such Poincar\'e series as 
\begin{align*} \phi(g) &= \sum_m c_m(\phi) \cdot P_m(g), \quad c_m(\phi) \in {\bf{C}}^{\times} \end{align*}
so that our shifted convolution sum $(\ref{quadsum})$ can be decomposed accordingly into a linear combination 
\begin{align}\label{quadsumdecomp} \sum_{(n)} W_{\phi} \left( \left( \begin{array}{cc} \frac{q(n)}{Y_{\infty}}  & ~Ê\\ ~& 1 \end{array} \right) \right) 
&=  \sum_m c_m(\phi) \sum_{(n)} W_{P_m} \left( \left( \begin{array}{cc} \frac{q(n)}{Y_{\infty}} & ~\\Ê~Ê& 1 \end{array} \right) \right). \end{align}
Here, each coefficient 
\begin{align*} W_{P_m} \left( \left( \begin{array}{cc} \frac{q(n)}{Y_{\infty}} & ~Ê\\Ê~Ê& 1\end{array} \right) \right) 
&:= \int_{ {\bf{A}}_F/F } P_m \left( \left( \begin{array}{cc} 1 & x \\Ê~Ê& 1 \end{array} \right) 
\left( \begin{array}{cc} \frac{1}{Y_{\infty}} & ~Ê\\Ê~Ê& 1 \end{array} \right)  \right) \psi(-q(n)x)dx \\ 
&= \int_{I \cong [0,1]^d \subset F_{\infty}} P_m \left( \left( \begin{array}{cc} 1 & x_{\infty} \\ ~& 1 \end{array} \right) 
\left( \begin{array}{cc} \frac{1}{Y_{\infty}} & ~Ê\\ ~Ê& 1 \end{array} \right) \right) \psi(-q(n) x_{\infty}) dx_{\infty} \end{align*}
on the right hand side of $(\ref{quadsumdecomp})$ is the Fourier-Whittaker coefficient at $q(n)$
of \begin{align*} P_m\left( \left( \begin{array}{cc} \frac{1}{Y_{\infty}} & ~\\ ~Ê& 1\end{array} \right) \right)\end{align*} 
i.e.~so that we take $r = q(n)$, we use the second formula of Proposition 
\ref{PoincarŽ} to express $(\ref{quadsumdecomp})$ equivalently as 
\begin{equation}\begin{aligned}\label{QSD2} 
\sum_m c_m(\phi) \sum_{(n)} \left( \int_{ {\bf{A}}_F/F } f \left( \left( \begin{array}{cc}  \frac{1}{Y_{\infty}} & ~\\Ê~Ê&1 \end{array} \right) \right) \psi(mx - q(n)x) dx
+ \sum_{c \in \Omega(\Gamma)} \operatorname{Kl}_{\Gamma}(m, q(n); c) 
\mathcal{F}_{f, q(n), c} \left( \left( \begin{array}{cc} \frac{1}{Y_{\infty}} & ~\\Ê~Ê& 1 \end{array} \right) \right)  \right). \end{aligned}\end{equation}
Note that we are justified in making such a decomposition, as the Poincar\'e series we consider generate the closed subspace 
$L^2_{\operatorname{cont}}(\operatorname{GL}_2(F) \backslash \operatorname{GL}_2({\bf{A}}_F), \omega) \subset
L^2(\operatorname{GL}_2(F) \backslash \operatorname{GL}_2({\bf{A}}_F), \omega)$ of continuous forms. 
Here, we refer e.g.~to the classical discussions of spectral decompositions given in \cite[$\S$ 7.4]{Ku} and \cite{GJ} 
with the standard argument in \cite[Lemma 14.3, Corollary 14.4]{IK}.
We shall use a variation of the Sobolev norms argument above to deduce that the sum of inner products can be bounded
uniformly in terms of the spectral parameters of some fixed orthonormal basis of the ambient Hilbert space 
$L^2(\operatorname{GL}_2(F) \backslash \operatorname{GL}_2({\bf{A}}_F), \omega)$. 
We now argue that the integral term in this latter expression $(\ref{QSD2})$ (which typically vanishes by orthogonality) 
is negligible relative to the second term, so that it is enough to estimate the simpler sum 
\begin{align}\label{quadsumdecompred} \sum_m c_m(\phi) \sum_{(n)} 
\sum_{c \in \Omega(c)} \operatorname{Kl}_{\Gamma}(m, q(n); c) 
\mathcal{F}_{f, q(n), c} \left( \left( \begin{array}{cc} \frac{1}{Y_{\infty}} & ~\\Ê~Ê& 1 \end{array} \right) \right).\end{align}
Opening the Kloosterman sums $\operatorname{Kl}_{\Gamma}(m, q(n); c)$ and switching the order of summation, 
$(\ref{quadsumdecompred})$ is the same as 
\begin{align*} &\sum_m c_m(\phi) \sum_{(n)} \sum_{c \in \Omega(\Gamma)} 
\sum\limits_{\gamma = \left( \begin{array}{cc} a & b \\ c & d \end{array} \right) \in \Gamma_{\infty} \backslash \Gamma_c / \Gamma_{\infty}} \omega(d) 
\psi_{\infty} \left( \frac{m a + q(n) d}{c} \right) \mathcal{F}_{f, q(n), c} \left( \left( \begin{array}{cc} \frac{1}{Y_{\infty}}  & ~Ê\\Ê~Ê& 1 \end{array} \right) \right) \\
&= \sum_m c_m(\phi) \sum_{c \in \Omega(\Gamma)} 
\sum\limits_{\gamma = \left( \begin{array}{cc} a & b \\ c & d \end{array} \right) \in \Gamma_{\infty} \backslash \Gamma_c / \Gamma_{\infty}} \omega(d)
\psi_{\infty} \left( \frac{ma}{c} \right) \sum_{(n)} \psi_{\infty} \left( \frac{q(n) d}{c} \right) 
\mathcal{F}_{f, q(n), c} \left( \left( \begin{array}{cc} \frac{1}{Y_{\infty}}  & ~Ê\\Ê~Ê& 1 \end{array} \right) \right), \end{align*}
which after partitioning the $(n)$-sum into congruence classes $u \bmod c$ is the same as 
\begin{align}\label{quadsumdecompredcong} \sum_m c_m(\phi) \sum_{c \in \Omega(\Gamma)} 
\sum\limits_{\gamma = \left( \begin{array}{cc} a & b \\ c & d \end{array} \right) \in \Gamma_{\infty} \backslash \Gamma_c / \Gamma_{\infty}} \omega(d)
\psi_{\infty} \left( \frac{ma}{c} \right) \sum_{u \bmod c} \sum_{(n) \atop n \equiv u \bmod c} \psi_{\infty} \left( \frac{q(u) d}{c} \right) 
\mathcal{F}_{f, q(n), c} \left( \left( \begin{array}{cc} \frac{1}{Y_{\infty}}  & ~Ê\\Ê~Ê& 1 \end{array} \right) \right). \end{align}
Here, the sum over $u \bmod c$ denotes the sum over a full set of representatives for the classes $(\mathcal{O}_F/c\mathcal{O}_F)$. 
We now use Poisson summation (cf.~\cite[Lemma 1]{VB})
to evaluate the inner sum. That is, given $\mathcal{F}$ a sufficiently well-behaved Schwartz class function on 
$x_{\infty} \in F_{\infty} \cong {\bf{R}}^d$, we can use the Poisson summation formula 
\begin{align}\label{PS} \sum_{(n) \atop n \equiv u \bmod c} \mathcal{F}(n) 
&= \frac{1}{\vert c \vert } \sum_{(h)} \widehat{\mathcal{F}}\left( \frac{h}{c} \right) \psi_{\infty} \left( \frac{hu}{c} \right), \end{align} where 
\begin{align*} \widehat{\mathcal{F}}(x_{\infty}) 
&= \int_{F_{\infty} \cong {\bf{R}}^d} \mathcal{F}(z_{\infty}) \psi_{\infty}(-z_{\infty}x_{\infty}) dz_{\infty} \end{align*}
denotes the Fourier transform of $\mathcal{F}$. Applying $(\ref{PS})$ to the inner sum in $(\ref{quadsumdecompredcong})$ 
(as in \cite[(3.2)]{VB}) then gives us 
\begin{align}\label{quadsumdecompredcongPS} \sum_m c_m(\phi) \sum_{c \in \Omega(\Gamma)} 
\sum\limits_{\gamma = \left( \begin{array}{cc} a & b \\ c & d \end{array} \right) \in \Gamma_{\infty} \backslash \Gamma_c / \Gamma_{\infty}} \omega(d)
\psi_{\infty} \left( \frac{ma}{c} \right) \frac{1}{\vert c \vert } \sum_{u \bmod c} \sum_{(h)} \psi_{\infty} \left( \frac{q(u) d+ h u }{c} \right) 
\widehat{\mathcal{F}}_{f, \frac{q(h)}{c}, c} \left( \left( \begin{array}{cc} \frac{1}{Y_{\infty}}  & ~Ê\\Ê~Ê& 1 \end{array} \right) \right), \end{align}
where 
\begin{align*} &\widehat{\mathcal{F}}_{f, \frac{q(h)}{c}, c} \left( \left( \begin{array}{cc} \frac{1}{Y_{\infty}} & ~\\Ê~Ê& 1 \end{array} \right) \right) 
= \int_{F_{\infty} \cong {\bf{R}}^d} \mathcal{F}_{f, \frac{q(z_{\infty})}{c}, c} 
\left( \left( \begin{array}{cc} \frac{1}{Y_{\infty}} & ~\\ ~Ê& 1 \end{array} \right) \right) \psi_{\infty} \left( - h z_{\infty} \right) dz_{\infty}  \\
&= \int_{F_{\infty} \cong {\bf{R}}^d}  \left(  \int_{{\bf{A}}_F } f \left( w \cdot \underline{c} \cdot \left( \begin{array}{cc} 1 & x \\ ~& 1 \end{array} \right) 
\left( \begin{array}{cc} \frac{1}{Y_{\infty}} & ~ \\ ~& 1 \end{array} \right) \right) 
\psi_{\infty} \left( -\frac{ q(z_{\infty})}{c} \cdot x \right) dx \right) \psi_{\infty} \left( - h z_{\infty} \right) dz_{\infty} \end{align*}
can be assumed to be absolutely convergent and rapidly decreasing in if we choose we sufficiently well-behaved Schwartz functions
$f \in \mathcal{S}(N_2({\bf{A}}_F) \backslash \operatorname{GL}_2({\bf{A}}_F); \psi_1)$. 

Let us now consider each quadratic Gauss sum in the latter expression $(\ref{quadsumdecompredcongPS})$, 
i.e.~opening up the quadratic polynomials $q(u) = ru^2+su + t$ in the latter expression to find the sums
\begin{equation}\begin{aligned}\label{OGS} \sum_{u \bmod c} \psi_{\infty} \left(  \frac{q(u) d + h u}{c} \right) 
&= \sum_{u \bmod c} \psi_{\infty} \left( \frac{ d r u^2 + (ds + h)u + d(t+h) }{c} \right) \\ 
&= \psi_{\infty} \left( \frac{dt}{c} \right) \sum_{u \bmod c} \psi_{\infty} \left( \frac{dr u^2 + (ds + h)u}{c} \right) \\
&= \psi_{\infty} \left( \frac{dt}{c} \right) \sum\limits_{ {k = \gcd(dr, c) \atop dr = k d' r', c = k c'} \atop u \bmod c', k \overline{k} \equiv 1 \bmod c' }
\psi_{\infty} \left( \frac{ d' r' u^2 + \overline{k}(ds+h)u}{c'} \right) \end{aligned}\end{equation}
Here, we rewrite the latter sum in such a way that $(d' r', c')=1$.
That is, the sum runs over classes $u \bmod c'$, where $c' = c/ k = c/\gcd(c, dr)$.
We then write $\overline{r}'$ to denote the inverse class of $r' \bmod c'$, and later put $\overline{r} = \overline{r}' k$. 
Observe that since $\psi_{\infty}$ is the archimedean component of the standard additive character 
$\psi = \otimes_v \psi_v$ on ${\bf{A}}_F/F$, it is defined on any $z_{\infty} = (z_{\infty, j})_{j=1}^d \in F_{\infty} \cong {\bf{R}}^d$ by  
\begin{align*} \psi_{\infty}(z_{\infty}) &= e(\operatorname{Tr}(z_{\infty})) = e \left( \sum_{j=1}^{[F:{\bf{Q}}]} z_{\infty, j} \right) =
\exp \left(2 \pi i \cdot \sum_{j=1}^{[F:{\bf{Q}}]} z_{\infty, j} \right). \end{align*}
Here, $\operatorname{Tr}$ denotes sum over real embeddings $\sigma_j: F \rightarrow {\bf{R}}$.
We then see the inner sum in $(\ref{OGS})$ equals 
\begin{align}\label{OGS2} \sum_{u \bmod c'} \psi_{\infty} \left( \frac{d'r' u^2 + \overline{k} (ds + h)u}{c'} \right) 
&= \sum_{u \bmod c'} e \left( \operatorname{Tr} \left( \frac{d'r' u^2 + \overline{k}(ds + h) u }{c'} \right) \right). \end{align}
Writing $\alpha_j = \sigma_j(\alpha)$ to denote the image under $\sigma_j: F \rightarrow {\bf{R}}$ of $\alpha \in F$, 
we can then use distributivity to justify interchanging the sum and the product on the right of $(\ref{OGS2})$, which gives us 
\begin{align*} \sum_{u \bmod c'} \psi_{\infty} \left( \frac{d'r' u^2 + \overline{k} (ds + h)u}{c'} \right) 
&= \prod_{j=1}^{[F:{\bf{Q}}]} \sum_{u_j \bmod c_j'} e \left( \frac{d_j' r_j' u_j^2 + \overline{k}_j (d_j s_j + h_j) u_j}{c_j'} \right). \end{align*}
Now, we can evaluate each of the quadratic Gauss sums in his latter expression explicitly via a direct calculation 
as given in \cite[Lemma 7]{VB} (for instance). That is, assuming without loss of generality that each
$c_j'$ is divisible by $4$ after embedding into a larger level structure if needed (cf.~\cite[$\S 3$]{VB}), 
we have for each embedding $\sigma_j$ the quadratic Gauss sum formula 
\begin{equation}\begin{aligned}\label{GSlemma} 
&\sum_{u_j \bmod c_j'} e \left( \frac{d_j' r_j' u_j^2 + \overline{k}_j (d_j s_j + h_j) u_j}{c_j'} \right) \\
&= \begin{cases} 0 &\text{if $2 \nmid \overline{k}_j (d_j s_j +h_j)$} \\
(1 + i) \cdot \sqrt{c_j'} \cdot \left( \frac{c_j'}{d_j' r_j'} \right) \cdot \epsilon_{d_j' r_j'}^{-1} \cdot e \left( \frac{- \overline{d_j' r_j' k_j^2}  (d_j s_j + h_j)^2/4}{c_j'} \right)
&\text{if $2 \mid \overline{k}_j (d_j s_j + h_j)$} \end{cases}, \end{aligned}\end{equation}
where $\left( \frac{\cdot}{\cdot} \right)$ denotes the Legendre symbol, and 
\begin{align*} \epsilon_{Q} &= \begin{cases} 1 &\text{if $Q \equiv 1 \bmod 4$} \\ i &\text{if $Q \equiv 3 \bmod 4$}. \end{cases} \end{align*}
Hence, interpreting the cases in terms of the norm homomorphism, we see that $(\ref{OGS})$ can be evaluated as 
\begin{align*} &\sum_{u \bmod c} \psi_{\infty} \left( \frac{q(u)d + h u}{c}  \right)
= \psi_{\infty} \left( \frac{td}{c} \right) \sum_{u \bmod c' = c/k \atop t = \gcd(dr, c)} \psi_{\infty} \left( \frac{d' r' u^2 + \overline{k} (ds+h)u}{c'} \right) \\ 
&= \psi_{\infty} \left( \frac{d t}{c} \right) \cdot \begin{cases} 0 &\text{if $2 \nmid {\bf{N}}(\overline{k}(ds+h))$} \\ 
\prod_{j=1}^{[F:{\bf{Q}}]} 
(1 + i) \cdot \sqrt{c_j'} \cdot \left( \frac{c_j'}{d_j' r_j'} \right) \cdot \epsilon_{d_j' r_j'}^{-1} \cdot 
e \left( \frac{- \overline{d_j' r_j' k_j^2} (d_j s_j + h_j)^2/4}{c_j'} \right)
&\text{if $2 \mid {\bf{N}}( \overline{k}(ds+h))$} \end{cases}. \end{align*}
It follows (cf.~\cite[(3.5)]{VB}) after switching the order of summation that  
$(\ref{quadsumdecompredcongPS})$ can be evaluated explicitly as 
\begin{equation}\begin{aligned}\label{substitute} 
&\sum_m c_m(\phi) \sum_{c \in\Omega(\Gamma)} 
\sum\limits_{\gamma = \left( \begin{array}{cc} a & b \\ c & d  \end{array} \right) \in \Gamma_{\infty} \backslash \Gamma_c / \Gamma_{\infty}} 
\omega(d) \psi_{\infty} \left( \frac{ma}{c} \right) \frac{1}{\vert c \vert } 
\sum_{(h)} \widehat{\mathcal{F}}_{f, \frac{q(h)}{c}, c} \left( \left( \begin{array}{cc} \frac{1}{Y_{\infty}}  & ~\\Ê~Ê& 1 \end{array} \right) \right)
\sum_{u \bmod c} \psi_{\infty} \left( \frac{q(u) d  + h u}{c} \right) \\ 
&= \sum_m c_m(\phi) \sum_{c \in\Omega(\Gamma)} 
\sum\limits_{\gamma = \left( \begin{array}{cc} a & b \\ c & d  \end{array} \right) \in \Gamma_{\infty} \backslash \Gamma_c / \Gamma_{\infty}} 
\omega(d) \psi_{\infty} \left( \frac{ma}{c} \right) \frac{1}{\vert c \vert }  \sum\limits_{(h) \atop \vert \overline{k}(ds + h) \vert \equiv 0 \bmod 2 }  
\widehat{\mathcal{F}}_{f, \frac{q(h)}{c}, c} \left( \left( \begin{array}{cc} \frac{1}{Y_{\infty}}  & ~\\Ê~Ê& 1 \end{array} \right) \right)
\psi_{\infty} \left( \frac{d t}{c} \right) \\ &\times \prod_{j=1}^{[F:{\bf{Q}}]} 
(1 + i) \cdot \sqrt{c_j'} \cdot \left( \frac{c_j'}{d_j' r_j'} \right) \cdot \epsilon_{d_j' r_j'}^{-1} 
\cdot e \left( \frac{- \overline{d_j' r_j' k_j^2} (d_j s_j + h_j)^2/4}{c_j'} \right) \\
&= \sum_m c_m(\phi) \sum_{c \in\Omega(\Gamma)} 
\sum\limits_{\gamma = \left( \begin{array}{cc} a & b \\ c & d  \end{array} \right) \in \Gamma_{\infty} \backslash \Gamma_c / \Gamma_{\infty}} \omega(d)
\sum\limits_{(h) \atop \vert \overline{k}(ds + h) \vert \equiv 0 \bmod 2 }  
\frac{\vert k \vert (1+i)^{[F:{\bf{Q}}]}  }{\vert c \vert^{\frac{1}{2}}} \cdot 
\widehat{\mathcal{F}}_{f, \frac{q(h)}{c}, c} \left( \left( \begin{array}{cc} \frac{1}{Y_{\infty}}  & ~\\Ê~Ê& 1 \end{array} \right) \right) \\
&\times \left( \prod_{j=1}^{[F:{\bf{Q}}]} \left( \frac{c_j'}{d_j' r_j'} \right) \cdot \epsilon_{d_j' r_j'}^{-1} \right)
\psi_{\infty} \left( \frac{ma}{c} \right) \psi_{\infty} \left( \frac{d t }{c} \right) 
\psi_{\infty} \left( \frac{- \overline{dr}(d s +h)^2/4}{c} \right) \\ 
&= \sum_m c_m(\phi) \sum_{(h)} \sum_{c \in\Omega(\Gamma)} \psi_{\infty}\left( - \frac{s \overline{r} h/2}{c} \right)
\frac{\vert c' \vert^{\frac{1}{2}}  (1+i)^{[F:{\bf{Q}}]}  }{\vert c \vert } \cdot 
\widehat{\mathcal{F}}_{f, \frac{q(h)}{c}, c} \left( \left( \begin{array}{cc} \frac{1}{Y_{\infty}} & ~Ê\\ ~Ê& 1 \end{array} \right) \right) \\ &\times
\sum\limits_{ { \gamma \in \Gamma_{\infty} \backslash \Gamma_c / \Gamma_{\infty} \atop
\gamma = \left( \begin{array}{cc} a & b \\ c & d  \end{array} \right) } \atop \vert \overline{k}(ds + h) \vert \equiv 0 \bmod 2} \omega(d) 
\left( \frac{c}{dr} \right) \left( \prod_{j=1}^{[F:{\bf{Q}}]} \epsilon_{d_j r_j}^{-1} \right) \psi_{\infty} \left( \frac{a(m-\overline{r} h^2/4)}{ c} \right)
\psi_{\infty} \left( \frac{ d(t - \overline{r} s^2/4) }{c} \right).\end{aligned}\end{equation} 
Here, we use the elementary congruence relations 
\begin{equation*}\begin{aligned} &\frac{dt}{c} - \frac{\overline{d' r' k^2}(ds + h)^2/4}{c'} = \frac{dt - \overline{dr}(ds + h)^2/4}{c} 
= \frac{d(t - \overline{r} s^2/4) - \overline{r} sh/2 - \overline{d r} h^2/4}{c} \end{aligned}\end{equation*}
to deduce that 
\begin{align*} \psi_{\infty} \left( \frac{dt}{c}  \right) \psi_{\infty} \left( -  \frac{\overline{d' r' k^2}(ds+h)^2/4 }{c'} \right) 
&= \psi_{\infty} \left( \frac{d(t - \overline{r}s^2/4)}{c} \right) 
\psi_{\infty} \left(- \frac{a \overline{r} h^2/4}{ c} \right) \psi_{\infty} \left( - \frac{2 \overline{r} sh/4}{c} \right). \end{align*}
Switching the order of summation again, we then get the expression  
\begin{equation}\begin{aligned}\label{substitute2} 
&\sum_m c_m(\phi)  \sum_{c \in\Omega(\Gamma)} 
\frac{\vert c' \vert^{\frac{1}{2}}  (1+i)^{[F:{\bf{Q}}]}  }{\vert c \vert } \cdot
\sum\limits_{  \gamma = \left( \begin{array}{cc} a & b \\ c & d  \end{array} \right) \in \Gamma_{\infty} \backslash \Gamma_c / \Gamma_{\infty} } 
\omega(d) \left( \frac{c}{dr} \right) \left( \prod_{j=1}^{[F:{\bf{Q}}]} \epsilon_{d_j r_j}^{-1} \right) \psi_{\infty} \left( \frac{am}{ c} \right) 
\psi_{\infty} \left( \frac{ d(t - \overline{r} s^2/4) }{c} \right) \\ &\times 
\sum_{(h)  \atop \vert \overline{k}(ds + h) \vert \equiv 0 \bmod 2} 
\widehat{\mathcal{F}}_{f, \frac{q(h)}{c}, c} \left( \left( \begin{array}{cc} \frac{1}{Y_{\infty}} & ~Ê\\ ~Ê& 1 \end{array} \right) \right)
\psi_{\infty} \left( \frac{- a \overline{r} h^2/4 - s \overline{r} h/2}{c} \right), \end{aligned}\end{equation} 
which after opening up the Fourier transform and switching the order of summation is the same as 
\begin{equation}\begin{aligned}\label{substitute3} 
&\sum_m c_m(\phi)  \sum_{c \in\Omega(\Gamma)} 
\frac{\vert c' \vert^{\frac{1}{2}}  (1+i)^{[F:{\bf{Q}}]}  }{\vert c \vert } \cdot
\sum\limits_{  \gamma = \left( \begin{array}{cc} a & b \\ c & d  \end{array} \right) \in \Gamma_{\infty} \backslash \Gamma_c / \Gamma_{\infty} } 
\omega(d) \left( \frac{c}{dr} \right) \left( \prod_{j=1}^{[F:{\bf{Q}}]} \epsilon_{d_j r_j}^{-1} \right) \psi_{\infty} \left( \frac{am}{ c} \right) 
\psi_{\infty} \left( \frac{ d(t - \overline{r} s^2/4) }{c} \right) \\ &\times 
\sum_{(h)  \atop \vert \overline{k}(ds + h) \vert \equiv 0 \bmod 2} \psi_{\infty} \left( \frac{- a \overline{r} h^2/4 - s \overline{r} h/2}{c} \right) 
\int_{ {\bf{A}}_F } f \left( w \cdot \underline{c} \cdot \left( \begin{array}{cc} 1&x\\~&1 \end{array} \right) \cdot
\left( \begin{array}{cc} \frac{1}{Y_{\infty}}&~\\~&1 \end{array} \right) \right) 
\int_{F_{\infty}}  \psi_{\infty} \left( \frac{- q(z_{\infty}) x_{\infty} - h c z_{\infty}}{c} \right) dz_{\infty} dx. \end{aligned}\end{equation}

Now for each $x_{\infty} = (x_{\infty, j})_{j=1}^d \in F_{\infty}$ with norm $\vert x_{\infty} \vert \neq 0$ (and hence with each component $x_{\infty, j} \neq 0$), 
we can evaluate the inner integral in this latter expression of $(\ref{substitute3})$ as 
\begin{align*} \int_{F_{\infty}}  \psi_{\infty} \left( \frac{- q(z_{\infty}) x_{\infty} - h c z_{\infty}}{c} \right) dz_{\infty} &=
\int_{F_{\infty}} \psi_{\infty} \left( - \left( \frac{r x_{\infty} }{c} \right) z_{\infty}^2 - \left( \frac{s x_{\infty} + hc}{c} \right) z_{\infty} - \left( \frac{t x_{\infty} }{c} \right) \right) d z_{\infty} \end{align*}
as 
\begin{align*} \int_{F_{\infty}} \psi_{\infty} \left( - \left( \frac{rx_{\infty}}{c} \right) z_{\infty}^2 - \left( \frac{s x_{\infty} + hc}{c} \right) z_{\infty} - \left( \frac{t x_{\infty} }{c} \right) \right) d z_{\infty} 
&= \left\vert  \frac{c}{2 i r x_{\infty}} \right\vert^{\frac{1}{2}} \cdot \psi_{\infty} \left( \frac{(s^2 - 4 r t) x_{\infty}}{4 r c}  \right) \cdot \psi_{\infty} \left( \frac{2 s x_{\infty} h c + h^2 c^2}{4 crx_{\infty}} \right) \end{align*}
by the integral formula 
\begin{align*} \int_{-\infty}^{\infty} e^{- (A(x_{\infty, j}) z_{\infty, j}^2 + B(x_{\infty, j}) z_{\infty, j} + C(x_{\infty, j}))} d z_{\infty, j} 
&= \sqrt{ \frac{\pi}{A(x_{\infty, j})}  } \cdot e^{ \frac{B(x_{\infty, j})^2 - 4 A(x_{\infty, j}) C(x_{\infty, j})}{4 A(x_{\infty, j})}} \end{align*}
applied to each component of $z_{\infty} = (z_{\infty, j})_{j=1}^d \in F_{\infty} \in {\bf{R}}^d$ (in the notations described above) with 
\begin{align*} A(x_{\infty, j}) &= 2 \pi i \left( \frac{r_j x_{\infty, j}}{c_j} \right), \quad B(x_{\infty, j}) = 2 \pi i \left( \frac{s_j x_{\infty, j} + h_j c_j}{c_j} \right), 
\quad C(x_{\infty, j}) = 2 \pi i \left( \frac{t_j x_{\infty, j}}{c_j} \right). \end{align*}
In this way, we argue that for some suitable Schwartz function 
$f' \in \mathcal{S}(N_2({\bf{A}}_F) \backslash \overline{G}({\bf{A}}_F); \psi_1)$ modulo $N_2(F_{\infty})$, we can reduce to bounding the simpler expression 
\begin{equation}\begin{aligned}\label{substitute4} 
&\sum_m c_m(\phi)  \sum_{c \in\Omega(\Gamma)} 
\frac{\vert c' \vert^{\frac{1}{2}}  (1+i)^{[F:{\bf{Q}}]}  }{\vert c \vert } \cdot
\sum\limits_{  \gamma = \left( \begin{array}{cc} a & b \\ c & d  \end{array} \right) \in \Gamma_{\infty} \backslash \Gamma_c / \Gamma_{\infty} } 
\omega(d) \left( \frac{c}{dr} \right) \left( \prod_{j=1}^{[F:{\bf{Q}}]} \epsilon_{d_j r_j}^{-1} \right) \psi_{\infty} \left( \frac{am}{ c} \right) 
\psi_{\infty} \left( \frac{ d(t - \overline{r} s^2/4) }{c} \right) \\ &\times 
\int_{ {\bf{A}}_F } f' \left( w \cdot \underline{c} \cdot \left( \begin{array}{cc} 1&x\\~&1 \end{array} \right) \cdot 
\left( \begin{array}{cc} \frac{1}{Y_{\infty}}&~\\~&1 \end{array} \right) \right) \psi_{\infty} \left( \frac{\Delta}{4 r c} \cdot x  \right) dx. \end{aligned}\end{equation}
Indeed, after switching the order of summation, the inner $(h)$-sum in the expression $(\ref{substitute3})$ is equivalent to 
\begin{equation*}\begin{aligned} \int_{{\bf{A}}_F } f \left(w  \underline{c}  \left( \begin{array}{cc} 1 & x \\ ~& 1\end{array} \right) 
\left( \begin{array}{cc} \frac{1}{Y_{\infty}} & ~ \\ ~Ê& 1 \end{array} \right) \right) \left\lbrace 
\sum\limits_{(h) \atop \vert \overline{k} (ds + h) \vert \equiv 0 \bmod 2} 
\psi_{\infty} \left( \frac{  - a \overline{r} h^2/4 - s \overline{r} h/2 }{c}  \right) 
\int_{F_{\infty}} \psi_{\infty} \left(  \frac{ -q(z_{\infty}) x_{\infty} - hc z_{\infty} }{c} \right) d z_{\infty} \right\rbrace dx, \end{aligned}\end{equation*}
which we argue can be approximated in terms of the $\vert x_{\infty} \vert \neq 0$ contributions and evaluated as 
\begin{equation*}\begin{aligned}
&\int\limits_{x \in {\bf{A}}_F \atop \vert x_{\infty} \vert \neq 0} f \left(w  \underline{c}  \left( \begin{array}{cc} 1 & x \\ ~& 1\end{array} \right) 
\left( \begin{array}{cc} \frac{1}{Y_{\infty}} & ~ \\ ~Ê& 1 \end{array} \right) \right) \\ &\times \left\lbrace \left\vert \frac{c}{2 ir x_{\infty}} \right\vert^{\frac{1}{2}}
\sum\limits_{(h) \atop \vert \overline{k} (ds + h) \vert \equiv 0 \bmod 2} 
\psi_{\infty} \left( \frac{  - a \overline{r} h^2/4 - s \overline{r} h/2 }{c}  \right) \psi_{\infty} \left( \frac{shc/2}{cr} \right) 
\psi_{\infty} \left( \frac{h^2 c^2}{4cr} \right) \right\rbrace \psi_{\infty} \left( \frac{\Delta}{4 r c} \right) dx. \end{aligned}\end{equation*} 
We argue that the inner $(h)$-sum in this latter expression can be approximated by some Gaussian integral and treated as a constant. 
Now, to re-interpret ($\ref{substitute4}$) in the style of \cite[$\S$ 3]{VB}, we now make two observations. 
The first is that we can make a change of variables $c \rightarrow c'' = 4 cr$ to relate $(\ref{substitute4})$ to the simpler sum 
\begin{equation}\begin{aligned}\label{substitute5} 
&\sum_m c_m(\phi)  \sum_{c'' \in\Omega(\Gamma)} 
\frac{\vert c' \vert^{\frac{1}{2}}  (1+i)^{[F:{\bf{Q}}]}  }{\vert c'' \vert} \\ &\times
\sum\limits_{  \gamma = \left( \begin{array}{cc} a'' & b'' \\ c'' & d''  \end{array} \right)
 \in \Gamma_{\infty} \backslash \Gamma_{c''} / \Gamma_{\infty} }  \omega(d'')
\left( \frac{c''}{d''} \right) \left( \prod_{j=1}^{[F:{\bf{Q}}]} \epsilon_{d_j''}^{-1} \right) \psi_{\infty} \left( \frac{a'' m}{ c'' } \right) 
\psi_{\infty} \left( - \frac{ d'' \Delta }{c''} \right) 
\mathcal{F}_{f', -\Delta, c''} \left( \left( \begin{array}{cc} \frac{1}{Y_{\infty}}  &  ~\\Ê~Ê& 1 \end{array} \right) \right), \end{aligned}\end{equation}
again for some suitable function $f' \in \mathcal{S}(N_2( {\bf{A}}_F) \backslash \overline{G}({\bf{A}}_F); \psi_1)$. 
We then see that the inner sum can be related to the generalized Kloosterman sum of modulus $4 rc$, central character $\omega$, 
and half-integral weight theta multiplier $\vartheta$ given by the expansion 
\begin{align*} \operatorname{Kl}_{\Gamma, \omega, \vartheta}(4 rm, - \Delta; 4 r c) &:= 
\sum\limits_{\gamma = \left( \begin{array}{cc} a'' & b'' \\ c'' & d''  \end{array} \right) \in \Gamma_{\infty} \backslash \Gamma_{4 rc} / \Gamma_{\infty}}
\omega(d'') \left( \frac{c''}{d''} \right) \left( \prod_{j=1}^{[F:{\bf{Q}}]} \epsilon_{d_j'' }^{-1} \right) \psi_{\infty} \left( \frac{ a'' 4 rm}{c''} \right) 
\psi_{\infty} \left( \frac{- d'' \Delta }{ c'' } \right).\end{align*}
Here, we take for granted the description given in Gelbart \cite[Proposition 2.16]{Ge} (for instance), 
or the explicit classical realization described in Proskurin \cite{Pro} (which carries over component-wise) of $\vartheta$. 
Writing $\mathcal{P}_m$ to denote the corresponding Poincar\'e series defined on $\overline{g} \in \overline{G}({\bf{A}}_F)$ by 
\begin{align*} \mathcal{P}_m(\overline{g}) = P_{f', \omega, \vartheta} (\overline{g}) 
= \sum_{\gamma \in \Gamma_{\infty} \backslash \Gamma} \omega(\gamma) \cdot \overline{\vartheta(\gamma)} \cdot f'(\gamma \overline{g}) \end{align*} 
for each $m$, it is then simple to deduce from $(\ref{substitute})$ and $(\ref{substitute2})$ that it is enough to bound the sum 
\begin{align*} \sum_{m} c_m(\phi) \cdot 
W_{\mathcal{P}_m} \left( \left( \begin{array}{cc} - \frac{\Delta}{Y_{\infty}} & ~\\Ê~Ê& 1 \end{array} \right) \right) \end{align*}
to derive a bound for the initial sum $(\ref{quadsum})$, where for each $m$ we write $W_{\mathcal{P}_m}$ to denote the unipotent integral 
\begin{equation*}\begin{aligned}
W_{\mathcal{P}_m} \left( \left( \begin{array}{cc} -\frac{\Delta}{Y_{\infty}} & ~\\Ê~Ê& 1 \end{array} \right) \right) &= 
W_{\mathcal{P}_m} \left( \left( \begin{array}{cc} -\frac{\Delta}{Y_{\infty}} & ~\\Ê~Ê& 1 \end{array} \right), 1 \right) \\ 
&=\int_{ {\bf{A}}_F/F } \mathcal{P}_m \left( \left( \begin{array}{cc} 1 & x \\Ê~Ê& 1 \end{array} \right) 
\left( \begin{array}{cc} \frac{1}{Y_{\infty}} & ~\\Ê~Ê& 1 \end{array} \right), 1 \right) \psi(-\Delta x) dx \\ 
&= \int_{ I \cong [0,1]^d \subset F_{\infty} } \mathcal{P}_m \left( \left( \begin{array}{cc} 1 & x_{\infty} \\Ê~Ê& 1 \end{array} \right) 
\left( \begin{array}{cc} \frac{1}{Y_{\infty}} & ~\\Ê~Ê& 1 \end{array} \right), 1 \right) \psi(-\Delta x_{\infty} ) dx_{\infty}. \end{aligned}\end{equation*}
But this reduces us to a special case of the proof of Theorem \ref{SCS}, 
with the totally positive nonzero $F$-integer $-\Delta$ replacing $q$. That is, we decompose each of the Poincar\'e series $P_m$
spectrally, then pass to unipotent integrals to bypass the Kuznetsov trace formula. Since the level is divisible
by the leading coefficient $r$, we use Weyl's law to deduce that there will be roughly $ \ll_{\pi} {\bf{N}} r$ contributions. 
Applying the previous argument to the spectral decomposition(s) of the Poincar\'e series then gives us the bound 
\begin{align*} \sum_{(n)} W_{\phi} \left( \left( \begin{array}{cc} \frac{q(n)}{Y_{\infty}} & ~\\ ~Ê& 1 \end{array} \right) \right) &\ll
\sum_m c_m(\phi) \cdot W_{\mathcal{P}_m} \left( \left( \begin{array}{cc} -\frac{\Delta}{Y_{\infty}} & ~\\Ê~Ê& 1 \end{array}\right) \right)
\ll_{\pi, \varepsilon} Y^{\frac{1}{4}} \cdot {\bf{N}}r \cdot {\bf{N}} \Delta^{\delta_0 - \frac{1}{2}}
\left( \frac{ {\bf{N}} \Delta  }{Y} \right)^{\frac{1}{2} - \frac{\theta_0}{2}  - \varepsilon}. \end{align*}
Here, we have to multiply in the factor of $Y^{\frac{1}{4}}$ at the last step to compensate for the fact that
the Fourier expansions (and hence Fourier coefficients) of genuine parallel half-integral weight forms evaluated
at the matrix $\operatorname{diag}( 1/Y_{\infty}, 1)$ are proportional to $Y^{-\frac{1}{4}} = \vert Y_{\infty} \vert^{-\frac{1}{4}}$.

\section{Iwasawa main conjectures for $\operatorname{GL}_2$ over CM fields}

Let $F$ be a totally real number field of degree $d = [F: {\bf{Q}}]$, integers $\mathcal{O}_F$, and adeles ${\bf{A}}_F$.
Let $\pi$ be a cuspidal $\operatorname{GL}_2({\bf{A}}_F)$-automorphic representation of conductor 
$c(\pi) \subset \mathcal{O}_F$ and unitary central character $\omega = \omega_{\pi}$. 
Let $K$ be a totally imaginary quadratic extension of $F$ of relative discriminant $\mathfrak{D} = \mathfrak{D}_{K/F} \subset \mathcal{O}_F$. 
Fix a prime $\mathfrak{p} \mid p$ in $F$ of residue degree $\delta = \delta_{\mathfrak{p}} = [F_{\mathfrak{p}}: {\bf{Q}}_p]$. 
We explain in this self-contained section how to deduce many cases of the Iwasawa-Greenberg main conjectures for $\pi$ in the maximal pro-$p$ 
abelian extension of $K$ unramified outside of $\mathfrak{p}$ when $\pi$ corresponds to a holomorphic Hilbert modular form of parallel weight two. 
We then explain how to use such results together with the nontriviality of the corresponding $p$-adic $L$-functions 
to derive bounds for Mordell-Weil ranks of abelian varieties in this ${\bf{Z}}_p^r$-extension of $K$, where $r = \delta + 1$. In particular, 
we obtain results for the cyclotomic ${\bf{Z}}_p$-extension of $K$ without using a new Euler system 
construction, and hence without generalizing the construction of Kato \cite{Ka} to number fields.

\subsubsection*{Iwasawa main conjectures} 

We first describe some relevant divisibilities for the so-called $r$-variable Iwasawa-Greenberg 
main conjecture (see \cite[Conjecture 1]{XW}) that can be deduced from various works in the ``basechange anticyclotomic variables".
We assume $\pi = \otimes_v \pi_v$ is associated to a $p$-ordinary cuspidal Hilbert eigenform 
of parallel weight $2$ and trivial character.\footnote{Note that we also could consider the more general setting described in 
\cite[Conjecture]{XW} here, but at the expense of clarity.} We then have by constructions of Carayol \cite{Ca}, Taylor \cite{Ta}, and Wiles \cite{Wi}
a Galois representation $\rho_{\pi}: G_F := \operatorname{Gal}(\overline{{\bf{Q}}}/F) \longrightarrow \operatorname{GL}_2(\mathcal{O})$ 
subject to the usual conditions so that $L(s, \rho_{\pi}) = L(s, \pi)$. Let us use the setup of \cite[$\S 1$]{XW}, and hence assume the following

\begin{hypothesis}\label{XW} The following conditions are met: \\

\begin{itemize}

\item[(1)] The fixed prime $p$ is not ramified in $F$, and each prime $\mathfrak{p} \mid p \subset \mathcal{O}_F$ splits in $K$. \\

\item[(2)] The totally imaginary quadratic extension $K/F$ is not contained in the narrow class field of $F$. \\

\item[(3)] The residual representation $\overline{\rho}_{\pi}$ of $\rho_{\pi}$ is irreducible in the sense of \cite[(irred)]{XW}. \\

\item[(4)] The $\mathcal{O}_{L}^{\times}$-valued characters giving the actions of $G_{F_{\mathfrak{p}}}$ on $V_{\mathfrak{p}}$ and 
$V_{\mathfrak{p}} / V_{\mathfrak{p}}^{+}$ are distinct in the sense of \cite[(dist)]{XW}. Hence according to \cite[$\S 1.1$]{XW}, for each prime 
$\mathfrak{p} \mid p \subset \mathcal{O}_F$, the restriction of $\rho_{\pi}$ to the decomposition group 
$G_{F_{\mathfrak{p}}} \cong \operatorname{Gal}(\overline{\bf{Q}}_p/ F_{\mathfrak{p}})$ 
is isomorphic to some upper triangular representation $V_{\mathfrak{p}}$ admitting a distinguished one-dimensional subspace $V_{\mathfrak{p}}^{+}$ 
with proscribed Galois action. Fixing $L$ a finite extension of ${\bf{Q}}_p$ which is sufficiently large to contain the Hecke field ${\bf{Q}}(\pi)$, 
we assume this condition on the $\mathcal{O}_L^{\times}$-valued characters giving the actions of $G_{F_{\mathfrak{p}}}$. \end{itemize} \end{hypothesis} 

Let $K_{\infty}$ denote the ${\bf{Z}}_p^r$-extension of $K$, this being the compositum of the anticyclotomic ${\bf{Z}}_p^{\delta}$-extension
of $K$ with the cyclotomic ${\bf{Z}}_p$-extension of $K$. In general, for each prime $\mathfrak{p}$ dividing $ p$ in $\mathcal{O}_F$, 
we can associate to the Galois representation $\rho_{\pi} = V$ a Selmer group $\operatorname{Sel}(\pi, K_{\infty})$. 
In the setting where $\pi$ is associated to an abelian variety $A/F$ (as described below), 
this coincides with the classical $\mathfrak{p}$-primary Selmer group of $A$ over $K_{\infty}$.
The first part of the Iwasawa-Greenberg main conjecture for $\operatorname{GL}_2$ over $K$ asserts that the Pontryagin dual 
$X(\pi, K_{\infty})$ of $\operatorname{Sel}(\pi, K_{\infty})$ has the structure of a finitely generated torsion $\mathcal{O}[[G]]$-module, 
in which case it has a characteristic power series. Let us write  
$\operatorname{char}(X(\pi, K_{\infty})) = \operatorname{char}_{\mathcal{O}[[G]]}(X(\pi, K_{\infty})) \in \mathcal{O}[[G]]$
to denote this characteristic power series. We know thanks to various constructions (among them the one given in \cite{XW}, 
see also \cite{Da}, \cite{Di}, \cite{Har}, \cite{Hi91}, and \cite{Pa}) that there exists for each $\mathfrak{p} \mid p$  a $p$-adic $L$-function 
$\mathcal{L}_{\mathfrak{p}}(\pi, K_{\infty}) \in \mathcal{O}[[G]]$. The main conjecture for $\operatorname{GL}_2$ over $K$ then posits that we have an equality 
of principal ideals $\left( \mathcal{L}_{\mathfrak{p}}(\pi, K_{\infty}) \right) = \left( \operatorname{char}(X(\pi, K_{\infty}))\right)$ in $\mathcal{O}[[G]]$.

\begin{theorem}\label{IMC} 

Let $p\geq 5$ be a rational prime, and let $\pi = \otimes_v \pi_v$ be a cuspidal automorphic representation of $\operatorname{GL}_2({\bf{A}}_F)$ 
associated to a Hilbert modular eigenform of parallel weight $2$, level $c(\pi)$, and trivial nebentype character satisfying Hypothesis \ref{XW} above.
Assume that $(c(\pi), \mathfrak{D}\mathfrak{p}) = (\mathfrak{p}, \mathfrak{D}) = 1$. 
Writing $c(\pi)^+ c(\pi)^{-}$ to denote the decomposition of $c(\pi) \mathcal{O}_K$ into a product of split primes $c(\pi)^+$ and inert primes $c(\pi)^{-1}$, 
let us also assume (5) that $c(\pi)^{-}$ is the squarefree product of a number of primes $v \subset \mathcal{O}_F$ congruent to $d  \bmod 2$, 
hence the corresponding root number $\epsilon \in \lbrace \pm 1 \rbrace$ of $L(s, \pi_K)$ is $+1$. 
As well, let us assume (6) that the residual representation $\overline{\rho}_{\pi}$ is ramified at each prime $v \mid c(\pi)^{-} \subset \mathcal{O}_F$, 
and (7) that the following technical conditions are met: \\

\begin{itemize}

\item[(A)] The restriction $\overline{\rho}_{\pi} \vert_{F(\zeta_p)}$ of $\overline{\rho}_{\pi}$ to the field $F(\zeta_p)$ 
obtained by adjoining a primitive $p$-th root of unity $\zeta_p$ to $F$ is absolutely irreducible. \\

\item[(B)] The following case is excluded when $p=5$: the projective image $J$ of $\overline{\rho}_{\pi}$ is isomorphic to 
$\operatorname{PGL}_2({\bf{F}}_p)$, and the mod $p$ cyclotomic character factors through 
$G_F \longrightarrow J^{\operatorname{ab}} \approx {\bf{Z}} / 2 {\bf{Z}}$. \\

\item[(C)] There is a minimal modular lifting of $\overline{\rho}_{\pi}$. \\

\item[(D)] Ihara's lemma for Shimura curves over totally real number fields is true; see \cite[Hypothesis 11.5]{VO2}. \\

\item[(E)] For each finite place $v$ of $F$, if $\overline{\rho}_{\pi} \vert_{I_{F_v}}$ is absolutely irreducible, then 
the cardinality of the residue field at $v$ is not congruent to $-1 \bmod p$. \\
 
\item[(F)] The completed character group associated to the residual representation $\overline{\rho}_{\pi}$ 
satisfies the multiplicity one condition of \cite[Hypothesis 11.5]{VO2}. \\
 
\end{itemize} Then, the $r$-variable main conjecture is true: We have that 
$\left( \mathcal{L}_{\mathfrak{p}}(\pi, K_{\infty}) \right) = \left( \operatorname{char}(X(\pi, K_{\infty}))\right)$ in $\mathcal{O}[[G]]$. 

\end{theorem}

\begin{proof} 

The result can be derived from \cite[Proposition 4.3]{VO3}, which carries over with minor changes to deal with 
the more general setting of $\delta >1$, along with \cite[Theorem 3]{XW}, \cite[Theorem 1.5]{Lon}, and \cite[Theorem 1.2]{VO2}. 
That is, putting together \cite[Theorem 3]{XW} with \cite[Theorem 1.5]{Lon} shows that we have an equality of principal ideals 
$\left( \mathcal{L}_{\mathfrak{p}}(\pi, K_{\infty})\vert_{\Omega} \right) = \left( \operatorname{char}(X(\pi, K_{\infty})\vert_{\Omega})\right)$ 
in $\mathcal{O}[[\Omega]]$. Writing $K_n$ for an integer $n \geq 1$ to denote the degree-$p^n$ extension of $K$ contained in 
$K^{\operatorname{cyc}}$, with $\Omega^{(n)} = \operatorname{Gal}(K_n D_{\infty}/K_n) \approx {\bf{Z}}_p^{\delta}$, 
we obtain from \cite[Theorem 1.2]{VO2} (b) that for each $n \geq 1$, there is an inclusion of ideals
$\left( \mathcal{L}_{\mathfrak{p}}(\pi, K_{\infty})\vert_{\Omega^{(n)}} \right) \subseteq \left( \operatorname{char}(X(\pi, K_{\infty})\vert_{\Omega^{(n)}})\right)$ 
in $\mathcal{O}[[\Omega^{(n)}]]$. Using (a) and (b) as input for \cite[Proposition 4.3]{VO3}, we deduce there is an inclusion
of ideals $\left( \mathcal{L}_{\mathfrak{p}}(\pi, K_{\infty}) \right) \subseteq \left( \operatorname{char}(X(\pi, K_{\infty}))\right)$ in $\mathcal{O}[[G]]$. 
The other inclusion is the proven in \cite[Theorem 3]{XW}. \end{proof} 

\subsubsection*{Applications to Mordell-Weil ranks}

Let us again write ${\bf{Q}}(\pi)$ to denote the Hecke field of $\pi$. It is conjectured and known in many cases that 
one can associate to $\pi$ an abelian variety $A = A_{\pi}$ over $F$ (defined uniquely up to isogeny) for which the following properties hold: \\

\begin{itemize} 

\item[(i)] The dimension of $A$ equals the degree of the Hecke field of $\pi$, i.e.~$\dim(A) = [{\bf{Q}}(\pi): {\bf{Q}}]$. \\

\item[(ii)] The ring of endomorphisms of $A$ is given by the integers of the Hecke field of $\pi$, 
i.e.~$\operatorname{End}_F(A) = \mathcal{O}_{ {\bf{Q}}(\pi)}$. \\

\item[(iii)] The Hasse-Weil $L$-function $L(s, A/F)$ of $A$ over $F$ is described (Euler factor for Euler factor)
in terms of the finite part $L(s, \pi)$ of the standard $L$-function of $\pi$ by the relation
\begin{align*} L(s-1/2, \pi) &= \Gamma_{\bf{C}}(s) L(s, A/F), \text{ where  $ \Gamma_{\bf{C}}(s) := 2(2\pi)^{-1} \Gamma(s)$.} \end{align*} \end{itemize}
The first two conditions (i) and (ii) imply that $A/F$ is of {\it{$\operatorname{GL}_2$-type}}, 
and also of {\it{strict $\operatorname{GL}_2$-type}} in the sense of \cite[$\S 3.2$]{YZZ}. 
Let $\Sha(A/K_{\infty})$ denote the Tate-Shafarevich group of $A$ over $K_{\infty}$, with $\Sha(A/K_{\infty})[p^{\infty}]$ its $p$-primary subgroup. 
Using the Kummer exact sequence
\begin{align*} 0 \longrightarrow A(K_{\infty}) \otimes {\bf{Q}}_p / {\bf{Z}}_p \longrightarrow \operatorname{Sel}(\pi, K_{\infty}) 
\longrightarrow \Sha(A/K_{\infty})[p^{\infty}] \longrightarrow 0 \end{align*} 
with the interpolation property satisfied by $\mathcal{L}_{\mathfrak{p}}(\pi, K_{\infty})$ (described in Theorem \ref{mvplfn} above) 
we can derive the following result from Theorem \ref{IMC}. Given a finite-order character $\mathcal{W}$ of $G$, 
let us write $\pi(\mathcal{W})$ to denote the corresponding induced representation of $\operatorname{GL}_2({\bf{A}}_F)$, 
and also $L(s, \pi \times \mathcal{W}) = L(s, \pi \times \pi(\mathcal{W}))$ to the denote the corresponding 
$\operatorname{GL}_2({\bf{A}}_F) \times \operatorname{GL}_2({\bf{A}}_F)$ Rankin-Selberg $L$-function of $\pi$ times $\pi(\mathcal{W})$. 

\begin{corollary} Assume the conditions of Theorem \ref{IMC}. If for some finite-order character $\mathcal{W}$ of $G$ the central value 
$L(1/2, \pi \times \mathcal{W}) = L(1/2, \pi \times \pi(\mathcal{W}))$ does not vanish, then the $\mathcal{W}$-isotypical components of both 
$A(K_{\infty})$ and $\Sha(A/K_{\infty})[p^{\infty}]$ are finite. \end{corollary}

Hence, we obtain from Theorem \ref{GAnv} the following result.

\begin{theorem}\label{isotyp} 

Let $p\geq 5$ be a prime, and let $\pi \cong \widetilde{\pi}$ be a cuspidal $\operatorname{GL}_2({\bf{A}}_F)$-automorphic representation  
attached to a Hilbert modular eigenform of parallel weight $2$, level $c(\pi)$, and trivial character satisfying Hypothesis \ref{XW}.
Assume the hypotheses of Theorem \ref{IMC}. Assume as well that that the Hecke field ${\bf{Q}}(\pi)$ is linearly disjoint over ${\bf{Q}}$ 
to the cyclotomic tower ${\bf{Q}}(\zeta_{p^{\infty}}) = \bigcup_{n \geq 1} {\bf{Q}}(\zeta_{p^n})$. 
Let $K/F$ be a totally imaginary quadratic extension of relative discriminant $\mathfrak{D}\subset \mathcal{O}_F$. 
and absolute discriminant $D_K$.
Let $\rho = \rho_w$ be any ring class character factoring through $G = \operatorname{Gal}(K_{\infty}/K)$. 
There exists an integer $\beta_0(\rho)$ such that for all characters $\psi = \psi_w$ of the cyclotomic Galois group 
$\Gamma = \operatorname{Gal}(K^{\operatorname{cyc}}/K)$ 
of exact order $p^{\beta}$ with $\beta \geq \beta_0(\rho)$, the central value $L(1/2, \pi \times \rho \psi)$ does not vanish, and hence 
the corresponding $\rho \psi$-isotypical components of both $A(K_{\infty})$ and $\Sha(A/K_{\infty})[p^{\infty}]$ are finite. \end{theorem}

\begin{proof} The result follows from Theorem \ref{GAnv} (via Corollary \ref{RCGAnv}), 
using the Weierstrass preparation theorem in the cyclotomic variable of the corresponding 
$(\delta_{\mathfrak{p}}+1)$-variable $p$-adic $L$-function for each choice of $\rho$. \end{proof}

We can also deduce the following rank formula. 
Let $\epsilon \in \lbrace \pm 1 \rbrace$ denote the sign of the Hasse-Weil $L$-function $L(s, A/K)$ of $A$ over $K$, 
equivalently the root number $\epsilon(1/2, \pi_K)$ of the $L$-function $L(s, \pi_K)$ of the basechange $\pi_K = \operatorname{BC}_{K/F}(\pi)$ 
of $\pi$ to $K$. Let $K^{\operatorname{cyc}}$ denote the cyclotomic ${\bf{Z}}_p$-extension of $K$, 
with Galois group $\Gamma = \operatorname{Gal}(K^{\operatorname{cyc}}/K) \cong {\bf{Z}}_p$. 
Let $D_{\infty}$ denote the anticyclotomic ${\bf{Z}}_p^{\delta}$-extension of $K$, 
with Galois group $\Omega = \operatorname{Gal}(D_{\infty}/K) \cong {\bf{Z}}_p^{\delta}$. Let $K_{\infty}$ denote the compositum 
$D_{\infty}K^{\operatorname{cyc}}$ of these extensions, with Galois group $G = \operatorname{Gal}(K_{\infty}/K) \approx {\bf{Z}}_p^r$. 

\begin{theorem}\label{rank} 

Let $\pi \cong \widetilde{\pi}$ be a cuspidal automorphic representation of $\operatorname{GL}_2({\bf{A}}_F)$ 
associated to a Hilbert modular eigenform of parallel weight $2$, conductor $c(\pi)$, and trivial character. 
Let $K/F$ be a totally imaginary quadratic extension of relative discriminant $\mathfrak{D}\subset \mathcal{O}_F$ 
and absolute discriminant $D_K$.
Suppose that $\pi$ has associated to it an abelian variety $A/F$ satisfying conditions (i), (ii), and (iii), 
and that $(c(\pi), \mathfrak{D}\mathfrak{p}) = (\mathfrak{p}, \mathfrak{D}) = 1$. 
Assume additionally that the following conditions hold: 
(iv) if $A$ acquires CM after basechange to some quadratic extension $K_{\pi}/F$, then this extension $K_{\pi}$ is not contained in 
$K_{\infty}$ when $\epsilon = +1$, and (v) $A$ has good ordinary reduction at all primes above $p$ in $F$. Finally, if the residue degree 
$\delta = [F_{\mathfrak{p}}: {\bf{Q}}_p]$ is greater than one, then we also assume the conditions of Theorems \ref{IMC}, \ref{ACmu}, and \ref{GAnv} 
above (including the vanishing of the anticyclotomic $\mu$-invariant), and that the absolute discriminant $D_K$ is sufficiently large. 
Then, $A(K_{\infty})$ is finitely generated if $\epsilon = +1$, and otherwise $A(K_{\infty})/A(D_{\infty})$ is finitely generated if $\epsilon = -1$. \end{theorem}

\begin{proof} 

Note that if $\delta =1$, then we can use the argument of \cite[Proposition 3.14]{VO3} (cf. \cite[Part II, $\S 1$]{VO12}) 
to deduce the result. In brief, this argument shows that we have the rank formula
\begin{align}\label{rf} \operatorname{corank}_{\mathcal{O}[[H]]} \left( A(K_{\infty}) \otimes {\bf{Q}}_p / {\bf{Z}}_p\right) 
&=\begin{cases} 0 &\text{ if $\epsilon = +1$}\\ 1 &\text{ if $\epsilon = -1$},\end{cases} \end{align}
where $\mathcal{O}$ is a finite extension of ${\bf{Z}}_p$ large enough to contain the integers of 
${\bf{Q}}(\pi)$, $H= \operatorname{Gal}(K_{\infty}/K^{\operatorname{cyc}}) \cong {\bf{Z}}_p^{\delta}$, 
and $\mathcal{O}[[H]]$ is the $\mathcal{O}$-Iwasawa algebra of $H$, 
in other words the completed group ring coming from $\mathcal{O}$-valued measures on $H$.
We deduce from such a formula when $\delta = 1$ that $A(K_{\infty})$ is finitely generated when $\epsilon = +1$,
and (using the nontriviality theorem of \cite[Theorem 1.15]{CV} with the formula of \cite{YZZ}) 
that $A(K_{\infty})/A(K_{\mathfrak{p}^{\infty}})$ is finitely generated when $\epsilon = -1$.  

Suppose now that $\delta \geq 2$. If $\epsilon = -1$, then the stated formula can be deduced again from the argument of 
\cite[Proposition 3.14]{VO3}, using the nontriviality theorem of \cite[Theorem 1.15]{CV} with \cite[Theorem 1.2]{YZZ} 
to derive a suitable growth-of-rank formula for each of the anticyclotomic lines. 
To describe this, let us for each integer $n \geq 0$ write $K_n$ to denote the extension of degree $p^n$ 
contained in the cyclotomic extension $K^{\operatorname{cyc}}$.
Consider the anticyclotomic extension defined by the compositum $K_n D_{\infty}$, which is contained in the anticyclotomic 
${\bf{Z}}_p^{\delta p^n}$-extension of $K_n$. Let $\Omega^{(n)} = \operatorname{Gal}(K_n D_{\infty}/K_n) \cong {\bf{Z}}_p^{\delta}$ 
denote the corresponding Galois group. Using the argument with results mentioned above, we derive the rank formula 
\begin{align*} \operatorname{corank}_{O[[\Omega^{(n)}]]} \left( A(K_n D_{\infty}) \otimes {\bf{Q}}_p/ {\bf{Z}}_p \right) &=1 \end{align*}
for each integer $n \geq 0$, where $\mathcal{O}[[\Omega^{(n)}]]$ denotes the corresponding Iwasawa algebra. 
Passing to the limit implies the corresponding $\mathcal{O}[[H]]$-module formula $(\ref{rf})$. 
Suppose now that $\epsilon = +1$. Fix an integer $n \geq 0$. 
We can use \cite[Theorem (A$^{\prime}$)]{Nek}, fixing an auxiliary prime $l$ and 
$\mathfrak{l} \mid l$ in ${\bf{Q}}(\pi)$ satisfying conditions \cite[(A1$^{\prime}$), (A2$^{\prime}$), (A3$^{\prime}$)]{Nek}, 
to deduce the following implication: If for any finite-order character $\rho^{(n)}$ of $\Omega^{(n)}$ 
(arising via composition with the norm $\rho^{(n)} = \rho \circ {\bf{N}}$ from a character $\rho$ of $\Omega^{(0)}$) the central value
\begin{align*} L(1/2, \pi^{(n)} \times \rho^{(n)}) = \prod_{\psi \in \operatorname{Gal}(K_n/K)^{\vee}} L(1/2, \pi \times \rho \psi) = L(1, A/K_n, \rho^{(n)}) \end{align*} 
does not vanish, then the $\rho^{(n)}$-isotypical component of the Mordell-Weil group $A(K_n[c(\rho^{(n)})])$ is finite. 
Here, $\pi^{(n)}$ denotes the basechange of $\pi$ to the degree-$p^n$ extension of
$F$ contained in its cyclotomic ${\bf{Z}}_p$-extension, with $L(s, \pi^{(n)} \times \rho^{(n)})$ the corresponding Rankin-Selberg $L$-function, 
and the first equality denotes the Artin decomposition into Rankin-Selberg $L$-functions of $\operatorname{GL}_2({\bf{A}}_F)$-representations.
As well, $K_n[c(\rho^{(n)})]$ denotes the ring class extension of $K_n$ of conductor $c(\rho^{(n)})$, where $c(\rho^{(n)})$ 
denotes the conductor of $\rho^{(n)}$. This allows us to deduce from Theorem \ref{GAnv} that
\begin{align*} \operatorname{corank}_{O[[\Omega^{(n)}]]} \left( A(K_n D_{\infty}) \otimes {\bf{Q}}_p/ {\bf{Z}}_p \right) &=0 \end{align*}
for each integer $n \geq 0$. Passing to the limit again implies the corresponding $\mathcal{O}[[H]]$-module formula $(\ref{rf})$. 

To deduce the stated formula for Mordell-Weil ranks in either case on the sign $\varepsilon$ when $\delta \geq 2$, 
we argue as follows these $\mathcal{O}[[H]]$-module formulae will suffice. Let us now impose the conditions of Theorem \ref{IMC}. 
We are then reduced (via Theorem \ref{isotyp}) to considering specializations of the underlying $p$-adic $L$-function 
$\mathcal{L}_{\mathfrak{p}}({\bf{1}}; T_1, \cdots, T_{\delta}, T_{\delta + 1})$ to check the claim directly.
Assuming the conditions of Theorem \ref{ACmu}, we deduce that there exists a minimal anticyclotomic exponent $\alpha_0$ such that each 
character $\rho_w$ of $\Omega$ of exact order $p^{\alpha}$ with $\alpha \geq \alpha_0$, the corresponding specialization 
$\mathcal{L}_{\mathfrak{p}}({\bf{1}}; \rho_w(\gamma_1)-1, \cdots, \rho_w(\gamma_{\delta}-1, \psi_w(\gamma_{\delta + 1})-1)$
does not vanish for any character $\psi_w$ factoring through the cyclotomic Galois group $\Gamma \cong {\bf{Z}}_p$.
To deal with the remaining anticyclotomic exponents $\leq \alpha_0$, we apply the result of Corollary \ref{RCGAnv} for the setting 
associated to each character $\rho_w$ of $\Omega \cong {\bf{Z}}_p^{\delta}$ of exact order $p^{\alpha}$ 
with $0 \leq \alpha \leq \alpha_0$ and $\beta =1$ to deduce the claim, e.g.~after applying the Weierstrass preparation
theorem to the corresponding basechange cyclotomic variable. \end{proof}

\end{document}